\documentclass[11pt,reqno]{amsproc}

\usepackage[sc]{mathpazo}\renewcommand{\mathbf}{\mathbold}
\normalfont
\usepackage[T1]{fontenc}
\usepackage[ttdefault=true]{AnonymousPro}
\usepackage[margin=0.9in]{geometry}
\usepackage{comment}
\usepackage{scalerel,nicefrac}
\usepackage[all,cmtip]{xy}
\usepackage{longtable}
\usepackage{dirtytalk}
\usepackage[dvipsnames]{xcolor}
\usepackage{todonotes}
\usepackage[inline]{enumitem}

\usepackage{rotating, setspace}

\makeatletter
\renewcommand{\l@subsection}{\@tocline{2}{0pt}{2.5pc}{5pc}{}}
\renewcommand{\l@subsubsection}{\@tocline{2}{0pt}{5pc}{7.5pc}{}}
\makeatother

\usepackage{bbm}
\usepackage{listings}
\usepackage{mathtools,amssymb}
\mathtoolsset{showonlyrefs}
\usepackage{tikz-cd}
\usepackage{subfigure}
\usepackage{shuffle}
\usepackage{mathrsfs}
\usepackage{multicol}
\usepackage{booktabs}
\usepackage{makecell}
\usepackage{mleftright}
\mleftright

\makeatletter
\renewcommand{\@secnumfont}{\bfseries}
\makeatother

\makeatletter
\renewcommand{\@secnumfont}{\bfseries}
\makeatother

\usepackage{hyperref}
\hypersetup{
    colorlinks=true,
    linkcolor=blue,
    filecolor=magenta,      
    urlcolor=cyan,
    citecolor=blue
    }
\numberwithin{equation}{section}
\usepackage[thmtools-compat]{keytheorems}
\declaretheorem[numberwithin=section, style=remark]{remark}
\declaretheorem[sibling=remark]{conjecture, proposition, theorem,%
lemma, corollary}

\declaretheorem[sibling=remark, style=definition]{definition, example}

\usepackage[nameinlink]{cleveref}
\AddToHook{cmd/appendix/before}{\crefalias{section}{appendix}}

\newcommand{\bA}{\mathbf{A}}

\newcommand{\R}{\mathbb{R}}

\newcommand{\ip}[2]{\left \langle #1, #2 \right \rangle}
\newcommand{\lp}[2][]{\left \| #2 \right\|_{#1}}

\newcommand{\Expectation}[2][]{\mathbf{E}_{#1}\left[#2\right]}

\newcommand{\Indicator}[1]{\mathbf{1}_{#1}}

\newcommand{\Prob}[2][]{\mathbf{P}_{#1}\left(#2\right)}

\newcommand{\eqdist}{\stackrel{d}{=}}

\newcommand{\dto}{\stackrel{d}{\to}}

\newcommand{\set}[1]{\left\{#1 \right \}}
\newcommand{\bigO}[2][]{\mathcal{O}_{#1}\left( #2 \right)}

\newcommand{\tr}{\operatorname{tr}}
\newcommand{\sech}{\operatorname{sech}}
\newcommand{\sinch}{\operatorname{sinch}}

\newcommand{\so}{\mathfrak{so}}

\allowdisplaybreaks

\title{The Singular Values of L\'evy's Area Matrix}

\author{Danilo Jr Dela Cruz}
\address{Mathematical Institute, University of Oxford. Woodstock Rd, Oxford OX2 6GG UK.}
\email{delacruz@maths.ox.ac.uk}

\author{Harald Oberhauser}
\address{Mathematical Institute, University of Oxford. Woodstock Rd, Oxford OX2 6GG UK.}
\email{oberhauser@maths.ox.ac.uk}

\begin{document}
\begin{abstract}
    The matrix of L\'evy's areas of $d$-dimensional Brownian motion is a fundamental object in stochastic analysis.
    In this article, we study the singular values of this $d \times d$ skew-symmetric random matrix.
    First, we derive an explicit formula for the density of the singular values and, en passant, present a new short proof of the characteristic function of L\'evy's area when $d \ge 3$. 
    This also allows us to extend the well-known formula for the density of L\'evy's area to $d \ge 3$. 
    Next, we use these results to characterise the singular spectrum as a determinantal point process with its kernel in explicit form.  
    Finally, we study the asymptotics as $d \to \infty$: the empirical measure of singular values converges to an absolute Cauchy distribution, the largest singular values are of order $d$ with Gaussian fluctuations, the smallest singular values are of order $1/d$, and the local bulk spacings are of order $1/d$, with sine-kernel statistics after rescaling.
\end{abstract}

\maketitle

\section{Introduction}
Let $B=(B^1_t,\ldots,B^d_t)_{t \ge 0}$ be a $d$-dimensional Brownian motion. 
L\'evy's area process, $\bA=(\bA_t)_{t \ge 0}$ is defined as the $d \times d$-matrix-valued process
\[
\bA_t \coloneqq \operatorname{Anti}\left(\int_0^t B_s \otimes dB_s \right),
\]
where $\operatorname{Anti}\left(\cdot\right)$ denotes the antisymmetric part of a $d\times d$ matrix.
Equivalently in coordinates, the $(i,j)$-th entry of the process $\bA_t$ is given as
\[
\bA^{i,j}_t \coloneqq \frac{1}{2} \left(\int_0^t B_s^i dB_s^j - \int_0^t B_s^j dB_s^i \right) 
\]
which by Green's formula, has the natural interpretation as the signed area that lies under the curve $[0,t]\ni s \mapsto (B^i_s,B^j_s)$ and the chord connecting $0$ with $(B^i_t,B^j_t)$, see Figure \ref{fig:Planar Brownian Motion}.
\begin{figure}[ht]
    \centering
    \includegraphics[width=\linewidth]{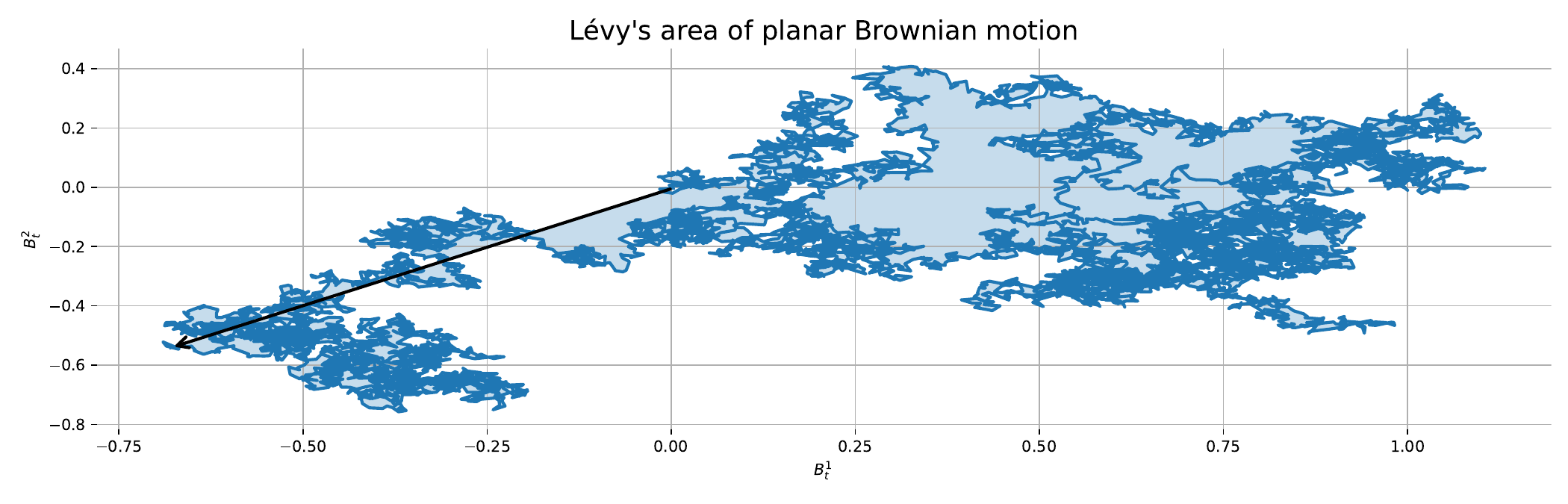}
    \caption{The trajectory of planar Brownian motion \((B_t^{1}, B^{2}_t)\) for \(t \in [0, 1]\), with the signed area enclosed between the path and its chord (black arrow) shaded in blue.}
    \label{fig:Planar Brownian Motion}
\end{figure}
\subsection*{The Singular Spectrum.}
Since $\bA_t$ is a (random) real-valued skew-symmetric matrix, it has $d$ (random) singular values $\sigma_k$; the non-zero singular values occur in pairs and if $d$ is odd, at least one of the singular values is $0$. 
Hence, the singular spectrum equals $\sigma_1,\sigma_1,\ldots,\sigma_n,\sigma_n$ when $d=2n$ is even and $\sigma_1,\sigma_1,\ldots,\sigma_{n},\sigma_{n},0$ when $d=2n+1$ is odd.
Empirically, one observes Figure \ref{fig:histogram_and_quantiles}.
This figure is intriguing: firstly, the heavy upper tails imply $\bA_t$ can be approximated by a low-rank matrix (via the singular value decomposition) and secondly, it indicates that in the limit as $d\to \infty$ the singular values converge to a deterministic limit distribution.
See also Section \ref{sec:geometry} for a geometric interpretation of singular values in the context of L\'evy's area.
\begin{figure}[ht]
    \centering
    \includegraphics[width=\linewidth]{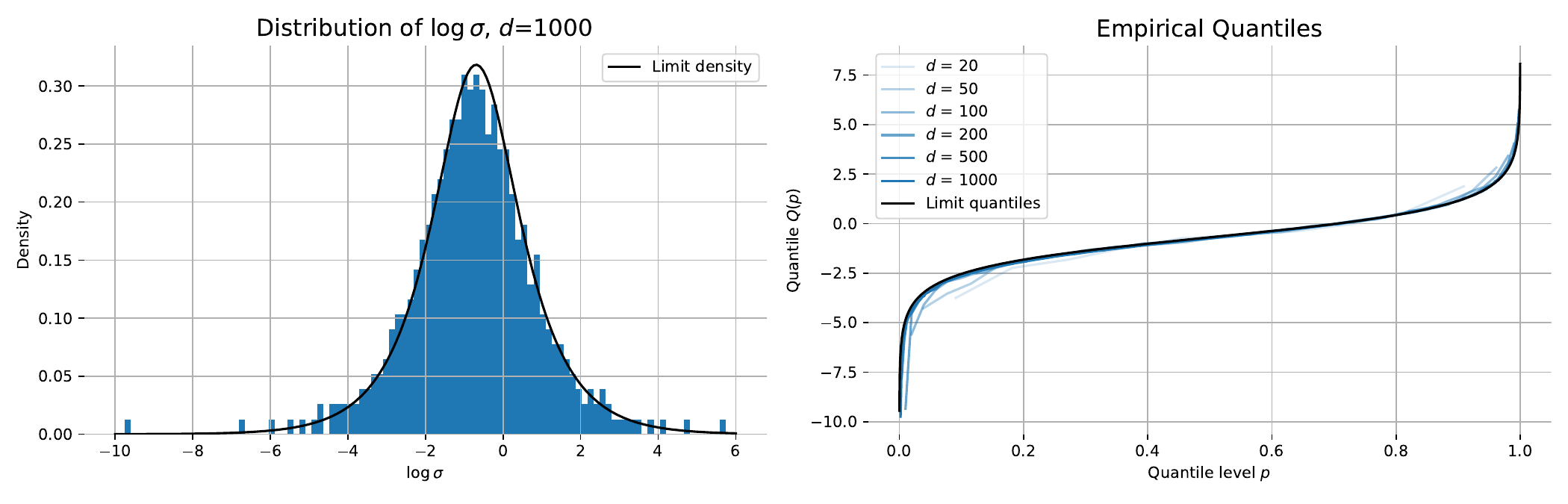}
    \caption{As the distribution over $\sigma_k$ is heavy tailed, we consider a log transform. (Left) Histogram of the empirical measure over $\log \sigma_k$ for $d=1000$ which agrees with the distribution of $\log |X|$ where $X$ is Cauchy. (Right) Quantiles of $\log \sigma_k$ for a range of $d$, showing the distribution is concentrated near zero with a heavy upper tail.}
    \label{fig:histogram_and_quantiles}
\end{figure}

\subsection*{Contributions.}
The goal of this article is to understand the distribution of the singular values of $\bA_t$, when $d$ is fixed and when $d\to \infty$.
Our main results are:
\begin{enumerate}
\item For any dimension $d$ we give an explicit formula for the joint distribution of the $n = \lfloor d /2 \rfloor$ distinct singular values of $\bA_t$. In the even case,
\[
 p(\sigma_1, \dots, \sigma_n)
    \propto
    \det\left[ \left( \frac{\pi}{t} \sigma_j \right)^{2(i-1)} \right]
    \det\left[ \sech\left( \frac{\pi}{t} \sigma_j \right)^{2(i-1)} \right]
    \prod_{i=1}^{n} \sech\left( \frac{\pi}{t} \sigma_i \right)
    .
\]
En passant, we give a new short proof of the characteristic function of $\bA_t$ when $d \ge 3$ and extend the well-known formula for the probability density of $\bA_t$ to $d \ge 3$.  
Further, we identify the SDE of the interacting particle system that governs the time evolution of $t \mapsto (\sigma_1(t)\,\ldots,\sigma_n(t))$ and show that $\sigma_k$ remain distinct.
\item For fixed time $t > 0$, we show that  $(\sigma_1,\ldots,\sigma_n)$ is a determinantal point process (DPP): for $k \le n$, the $k$-point correlation function is the determinant of the Gram matrix of a kernel.
\begin{equation*}
\qquad \qquad
\rho_k(\sigma_1,\ldots,\sigma_k)\coloneqq  \frac{n!}{(n-k)!}
    \int_{\mathbb{R}^{n-k}} p(\sigma_{1}, \dots, \sigma_{n}) d\sigma_{k+1} \dots d\sigma_{n}
=
\left(\frac{\pi}{t}\right)^{k}
\det\left[ K_{n}\left( \frac{\pi}{t} \sigma_{i}, \frac{\pi}{t} \sigma_{j}\right) \right]_{i, j = 1}^{k}
\end{equation*}
The kernel $K(x, y) = \sum_{m=0}^{n-1} p_m(x) q_m(y)$ is a sum of "biorthogonal functions" $p_m, q_m$, which are proportional to polynomials in $x^2$, $\sech(x)^2$ respectively.
Hence, they satisfy a recurrence relation whose coefficients we identify explicitly.
\[
\begin{array}{c|c|c}
  & d=2n & d=2n+1 \\
  \hline
p_m(x) & \mathrm{Re}\left( \left( x + i \frac{\pi}{2} \right)^{2m} \right) & \mathrm{Re}\left( \left( x + i \frac{\pi}{2} \right)^{2m+1} \right) \\
q_m(x) & \frac{2}{\pi}\frac{1}{ (2m)!} \sech^{(2m)}(x) & \frac{2}{\pi}\frac{-1}{(2m+1)!} \sech^{(2m+1)}(x)
\end{array}
.
\]
Besides connecting the study of L\'evy's area with the classic subject of DPPs, this representation gives an exact characterisation of the singular value law and suggests efficient sampling algorithms by reducing simulation to one-dimensional conditional DPP samplers.
However, it does not provide a joint sample of L\'evy's area and Brownian motion which would be required for SDE numerics.
\item We show that when $d \to \infty$, the empirical measure $\mu_{d}$ over the singular values converges weakly, with probability 1, to an absolute Cauchy distribution with scale parameter \(t / 2\).
That is
\begin{equation*}
    \mu_{d} := \frac{1}{n} \sum_{k=1}^{n} \delta_{\sigma_k} \dto \left|\mathrm{Cauchy}\left( 0, \frac{t}{2} \right)\right|\; a.s.
\end{equation*}
We then show the largest singular values are of order $n$ with Gaussian fluctuations, see Figure \ref{fig:largest singular values}. More precisely, let $\sigma_{(k)}$ be the $k$-th largest singular value, then
\begin{equation*}
    \sqrt{n}\left(\frac{\sigma_{(k)}}{n} - \frac{2}{2k-1} \frac{t}{\pi}\right)
    \dto
    N\left(0, \frac{1}{2} \left(\frac{2}{2k-1} \frac{t}{\pi}\right)^2\right)    .
\end{equation*}
On the other hand, the smallest and the bulk singular values are order $1/n$ and have fluctuations governed by the "sine kernel", see Figure \ref{fig:smallest singular values}.

\begin{figure}[ht]
    \centering
    \includegraphics[width=\linewidth]{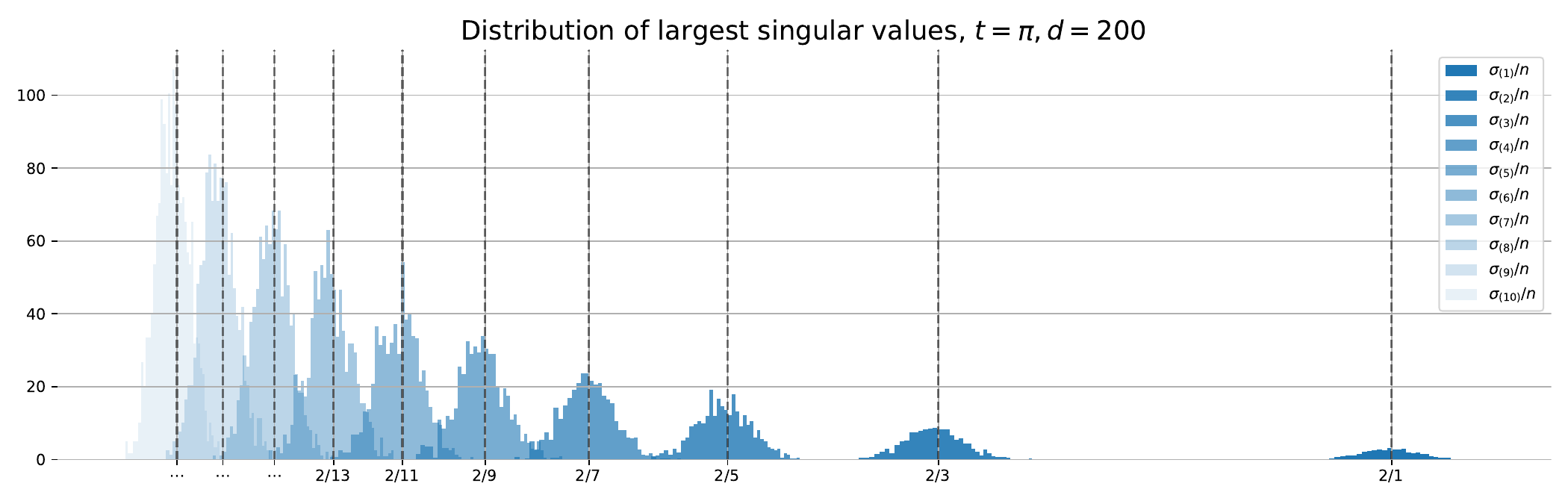}
    \caption{Histogram of the $k$-th largest singular value $\sigma_{(k)} / n$ for each $k \in [10]$, for $t = \pi, d=200$ over 1000 samples. The distribution of $\sigma_{(k)}$ is centred at $2 / (2k-1)$.}
    \label{fig:largest singular values}
\end{figure}

\begin{figure}[ht]
    \centering
    \includegraphics[width=\linewidth]{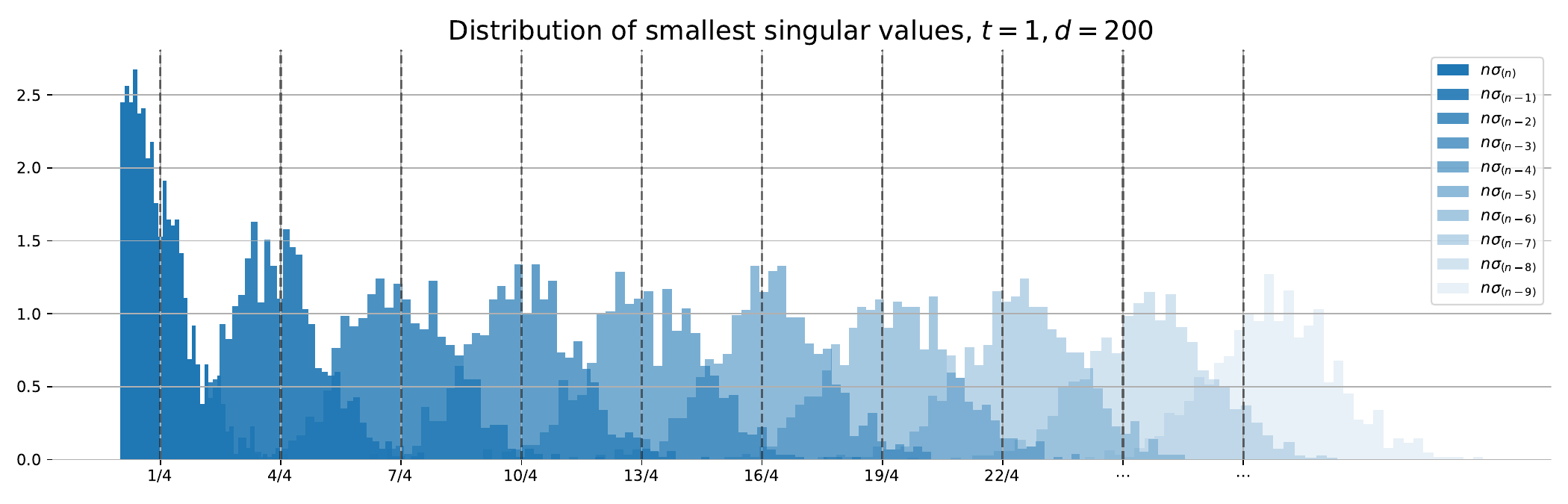}
    \caption{Histogram of the $k$-th smallest singular value $n\sigma_{(n-k)}$ for each $k \in [10]$, for $t = 1, d=200$ over 1000 samples.}
    \label{fig:smallest singular values}
\end{figure}
\end{enumerate}

\subsection*{Previous Work.}
L\'evy's area is a fundamental object in stochastic analysis with a vast literature; we only give a short overview that focuses on articles that are directly relevant for our purposes.
L\'evy's area process is named after Paul L\'evy who initiated its study in 1940 \cite{levyMouvementBrownienPlan1940} and subsequently in \cite{levyWienersRandomFunction1951} derived explicit formulas for its characteristic function $u \mapsto \mathbb{E}[e^{i u \bA^{i,j}_t}]$ and its characteristic function conditioned on the Brownian motion $(x,y,u)\mapsto \mathbb{E}[e^{iu \bA^{i,j}_t}|B^{i}_t=x, B^j_t=y]$. 
Note that L\'evy's original work and many subsequent articles focus on two-dimensional Brownian motion ($\bA^{i,j}_t$ as opposed to $\bA_t$ in the above characteristic functions) which makes computation and finding explicit formulas much easier; more on this below.
After L\'evy's contributions, the works of Gaveau \cite{gaveauPrincipeMoindreAction1977} and Hulanicki \cite{hulanicki1976} identified two-dimensional Brownian motion and its L\'evy's area as the "canonical Brownian motion" on the Heisenberg group; see \cite{baudoinStochasticAreasHorizontal2023} for a modern perspective on this approach that treats L\'evy's area simultaneously as a Brownian functional, a geometric object, and a source of explicit hypoelliptic PDE analysis.
This geometric, sub-Riemannian interpretation allowed Gaveau in \cite{gaveauPrincipeMoindreAction1977} to use PDE results about the fundamental solution of the Laplacian on the Heisenberg group to derive an "explicit formula" for the density of the three-dimensional process $(B^i_t,B^j_t,\bA^{i,j}_t)$.
After these classic results, Marc Yor returned to L\'evy's formula for $\mathbb{E}[e^{iu \bA_t^{i,j}}|B^{i}_t=x, B^j_t=y]$ and gave a short and elegant proof of this formula by purely stochastic calculus computations using Bessel processes combined with a result of David Williams \cite{yor1980remarques}.
Yet another related and more recent approach to  L\'evy's formula for the characteristic function is given by so-called affine processes, see \cite{cuchieroSignatureSDEsAffine2025}.
While much of the work we discussed above focused nearly exclusively on L\'evy's area for two-dimensional Brownian motion, Helmes and Schwane \cite{helmesLysStochasticArea}  extended the formula for characteristic functions given by L\'evy \cite{levyWienersRandomFunction1951} for the two-dimensional case, $d=2$, to the case $d\ge 3$.
Further, Helmes and Schwane \cite{helmesLysStochasticArea} also provide a formula for the joint density of the $d$-dimensional Brownian motion $B_t$ and the trace of $M\cdot \bA_t$ where $M$ is a skew-symmetric matrix, as well as a generalisation of this formula when $M\cdot \bA_t$ is replaced by the sum of several anti-commuting, orthogonal, and skew-symmetric matrices acting on $\bA_t$.
However, neither of these formulas implies a formula for the joint density of $(B_t,\bA_t)$ when $d \ge 2$. Ultimately, the reason is that the differential geometric interpretation of Gaveau relies on explicit formulas for the fundamental solutions of the Laplacian on a sub-Riemannian manifold which seem to be only available in closed form when $d=2$ and the manifold accurately captures $(B_t,\bA_t)$; explicit formulas for fundamental solutions are of course known for other manifolds, but these do not correspond to Brownian motion and its area process which seems to be a fundamental obstacle for the PDE approach to derive explicit density formulas when $d \ge 3$.
Another important theme in the study of L\'evy's area is series expansions. 
These go back to L\'evy's original work \cite{levyWienersRandomFunction1951} which provided two different proofs "L\'evy's formula": one via series expansion of Brownian motion, the other via rewriting Brownian motion with polar coordinates and stochastic integration.
Such series expansions have been especially important for numerics since truncation of such series allows one to approximately sample from $\bA_t$; we mention pars pro toto \cite{kloedenApproximationMultipleStochastic1992},\cite{wiktorssonJointCharacteristicFunction2001} and \cite{fosterApproximatingSignatureBrownian2025} among many articles.

\subsection*{Outline of this Article.}
We start Section \ref{section:Density of the L\'evy's Area Matrix} with a new short proof of the characteristic function of the matrix of L\'evy's area and Brownian motion when $d \ge 3$. 
The key idea is to switch to the coordinate system given by block diagonalisation of $\bA_t$ and use orthogonal invariance and independence to factorise the characteristic function.
While straightforward and natural, this highlights structure which then allows one to invert the characteristic function and obtain an explicit formula for the density when $d\ge3 $, Theorem \ref{thm:Density of L\'evy's Area Matrix}.
Although the explicit formula for the density of $\bA_t$ is classic when $d=2$, no such formula seems to be known for $d\ge 3$. 
We suspect the main reason is that inversion of the characteristic function appeared to be analytically intractable. However, with block diagonalisation, we will be able to use integral identities to obtain a closed formula for the density.
Furthermore, due to orthogonal invariance the singular values and vectors are independent, and the density of $\bA_t$ implies the density of the singular values, Corollary \ref{corollary:Joint density of L\'evy's area singular values}. We then show how this approach can be extended. In Section \ref{section:Density of L\'evy's areas and Brownian increments} we discuss the joint density of Brownian motion and L\'evy's areas and in Section \ref{section:Extension to other stochastic processes} we consider the L\'evy's areas of Gaussian processes with i.i.d.~coordinates.

Section \ref{sec:SDE} identifies the time-dynamics of $t \mapsto (\sigma_1(t),\ldots,\sigma_n(t))$ as an interacting particle system. 
The main idea is again to block diagonalise $\bA_t$ and use the Brownian motions that arise after projecting to the singular subspaces as driving noise for an SDE. 
This shows the singular values are repulsive and never collide.

Section \ref{section:Biorthogonal Ensemble}, shows that for fixed $t>0$, the singular values are a DPP. 
The kernel is in explicit form in terms of biorthogonal functions and we identify the coefficients of their recurrence relation, see Theorem \ref{theorem: Singular values are a DPP}. 
This links the study of the spectrum of L\'evy's area matrix to a classical model for repulsive systems. 
Since the kernel is explicit, the usual advantages of DPPs apply:  typically challenging computations have closed forms and existence of simple sampling schemes.  

Section \ref{section:Establishing Spectral Decay} characterises for fixed $t>0$ the behaviour of the singular values as $d \to \infty$. 
In view of the DPP representation from Section \ref{section:Biorthogonal Ensemble}, this reduces to finding limits of rescaled versions of the kernel. 
On a global scale, Theorem \ref{theorem:Global Law} proves the empirical measure over the singular values converges to an absolute Cauchy distribution. 
On a local scale, Theorems \ref{theorem: Smallest singular values}, \ref{theorem: Bulk singular values} and \ref{theorem: Largest singular values} obtain the limiting distribution of the smallest, bulk and largest singular values respectively.

Section \ref{section:Conclusion} summarises the article and gives further connections and directions for future research.

\subsection*{Notation and conventions.}
We denote the density of a random variable $X$ with $p_{X}(X)$ and omit the subscript to $p(X)$ if the random variable is clear.
Throughout, $B_t$ denotes a $d$-dimensional Brownian motion and $\bA_t$ its matrix of L\'evy's areas. We let $n = \lfloor d / 2 \rfloor$. For block diagonalisation, we consider the $2 \times 2$ matrix which we use as a building block.
$$
J = \begin{bmatrix}
    0 & 1 \\
    -1 & 0
\end{bmatrix}
.
$$
We summarise the remaining notation with the following table.
\begin{center}
 \begin{longtable}{lll}\toprule
 Symbol & Meaning & Page\\
 \toprule
      \multicolumn{2}{c}{Matrices and Vectors}\\
      \midrule
      $J(\varphi)$ & Rescaled $2 \times 2$ rotation matrix $\varphi J$ & \pageref{section:Density of the L\'evy's Area Matrix}  \\
      $J(\Phi)$ & Block diagonal matrix $\bigoplus_{k=1}^{n} J(\varphi_k) \oplus 0_{d - 2n}$ & \pageref{section:Density of the L\'evy's Area Matrix} \\
      $D_{\Phi}$ & Diagonal Matrix $\bigoplus_{k=1}^{n} \varphi_k I_2 \oplus 0_{d - 2n}$ & \pageref{theorem:new proof of cf}  \\
      $|\Theta|$ & Matrix modulus $|\Theta| = \sqrt{\Theta^{\top} \Theta} = Q D_{|\Phi|} Q^{\top}$ & \pageref{theorem:new proof of cf} \\
      $\so(d)$ & Space of $d\times d$ real skew-symmetric matrices & \pageref{theorem:new proof of cf} \\
      $O(d)$ & Space of $d\times d$ orthogonal matrices & \pageref{section:Density of the L\'evy's Area Matrix} \\
      $g(\Phi)$ & Jacobian for block diagonalisation (Lemma \ref{lemma:Jacobian for Block Diagonalisation of Skew-symmetric Matrices}) & \pageref{corollary:Joint density of L\'evy's area singular values} \\
      \midrule
      \multicolumn{2}{c}{Singular Values}\\
      \midrule
      $\lambda_k$ & Signed singular value of $\bA_t$ & \pageref{section:Density of the L\'evy's Area Matrix} \\
      $\sigma_k$ & Singular value of $\bA_t$, equal to $|\lambda_k|$  & \pageref{section:Density of the L\'evy's Area Matrix} \\
      $x_k$ & Rescaled singular value, equal in distribution to $\pi \sigma_k / t$ & \pageref{proof of density of levy's area matrix} \\
      $\sigma_{(k)}$ & The $k$-th largest singular value & \pageref{theorem: Largest singular values} \\
      $\varphi_k$ & Signed singular value of test matrix $\Theta$ & \pageref{section:Density of the L\'evy's Area Matrix}\\
      $\Phi$ & $\Phi = (\varphi_1,\dots, \varphi_n)$ Uppercase to indicate a collection of $\lambda_i$ & \pageref{section:Density of the L\'evy's Area Matrix} \\
      $u_k$ & Squared subspace norm & \pageref{corollary:joint density of sigma, u} \\
\midrule
\multicolumn{2}{c}{Determinantal Point Process}\\
\midrule
      $\mu_d$ &Empirical measure over singular values $\mu_d =\frac{1}{n} \sum_{k=1}^{n} \delta_{\sigma_k}$& \pageref{theorem:Global Law} \\
      $\nu_d$ & Counting measure over rescaled singular values $\nu_d = \sum_{k=1}^{n} \delta_{x_k}$& \pageref{theorem:Global Law} \\
      $K_n(x, y)$ & Kernel of the determinantal point process $x_i$& \pageref{theorem: Singular values are a DPP} \\
      $\rho_k$ & $k$-point correlation function & \pageref{definition:Determinantal Point Process} \\
      $p_m(x), q_m(x)$ & Biorthogonal functions & \pageref{theorem: Singular values are a DPP} \\
      $w(x)$ & Various weight functions & \pageref{thm:Density of L\'evy's Area Matrix} \\
      $\varrho$ & Limiting density of empirical measure & \pageref{theorem:Global Law} \\
 \bottomrule
 \end{longtable}
\end{center}

\section{The Characteristic Function and Density of L\'evy's Area}
\label{section:Density of the L\'evy's Area Matrix}
The main results of this section are: 
\begin{enumerate}[label=\roman*.]
\item a new proof for the formula for the characteristic function of L\'evy's area when $d \ge 3$,
\item an explicit formula for the density of L\'evy's area  when $d \ge 3$,
\item an explicit formula for the density of the singular values of L\'evy's area when $d \ge 3$.
\end{enumerate}
We first show that the characteristic function of $d$-dimensional Brownian motion and its L\'evy's areas "factorises" which allows one to recover the well-known formula of the characteristic function for general $d \ge 3$ from the case $d=2$. This follows from considering the block diagonal form of the skew-symmetric test matrix $\Theta \in \so(d)$ (Lemma \ref{lemma:Spectral Theorem for Skew-symmetric Matrices})
$$
\Theta = Q J(\Phi) Q^{\top},\quad J(\Phi) = \bigoplus_{k=1}^{n} J(\varphi_{k}) \oplus 0_{d-2n},\quad
J(\varphi_{k}) = \begin{bmatrix}
0 & \varphi_{k} \\
-\varphi_{k} & 0
\end{bmatrix}
$$
where $Q$ is an orthogonal matrix and $\varphi_k \in \R$ are its signed singular values. (For $\bA_t$, the signed singular values, which we denote $\lambda_k$, are more natural as they correspond to L\'evy's areas. E.g. for $d=2$, $\lambda_1 = \bA^{1, 2}$ and $\sigma_1 = |\lambda_1|$.)
While the proof of the characteristic function we present is straightforward and natural, it stands in stark contrast with previous approaches \cite{helmesLysStochasticArea}. 
Moreover, for the derivation of the density of L\'evy's area when $d \ge 3$ the switch to the block diagonal coordinate system is essential and allows the use of integral identities to derive an explicit formula.

\subsection{A New Proof for the Characteristic Function of L\'evy's Area}
Structurally, $\bA_t$ almost has independent entries: if $\set{i, j} \cap \set{k, l} = \emptyset$, then $(\bA_t^{i,j}, \bA_t^{l,k})$ are independent as they are functions of independent Brownian motions. Hence, there should be some form of factorisation. The following proof shows how this can be realised using orthogonal invariance of $B_t$ and allows us to recover the characteristic function for general $d$ from $d = 2$.

\begin{theorem} \label{theorem:new proof of cf}
Let $d \ge 2$, $\gamma \in \R^d$, and $\Theta \in \so(d)$. The characteristic function of $\bA_t$ and $B_t$ is
\begin{align}
\label{equation:CF of L\'evy's Area and Brownian motion}
\begin{split}
    \Expectation{\exp\left(i \sum_{i < j} \Theta_{i,j} \bA_{t}^{i, j} + i\sum_{i} \gamma_{i} B_{t}^{i}\right)}
    =
     \mathbb{E}\left[{\exp\left(\frac{i}{2}\tr{\Theta^{\top} \bA_t} + i\ip{\gamma}{B_t}\right)}\right] \\
    =
    \det\left( \cosh \frac{t}{2} |\Theta|\right)^{-1/2}  \times \exp\left(
    -\frac{t}{2}
    \gamma^{\top}
    \sinch\left(\frac{t}{2} |\Theta| \right)
    \cosh\left(\frac{t}{2} |\Theta| \right)^{-1}
    \gamma
    \right)
\end{split}
\end{align}
where $\sinch(x) = \sinh(x) / x$ and $|\Theta| = \sqrt{\Theta^{\top} \Theta}$ is the matrix modulus.
\end{theorem}
\begin{proof}
    Let \(\Theta = Q J(\Phi) Q^{\top}\), by orthogonal invariance of $B_t$, \((B_t, \bA_t)\eqdist (QB_t, Q \bA_t Q^{\top})\) and
    \begin{equation*}
        \Expectation{\exp\left(\frac{i}{2}\tr{\Theta^{\top} \bA_t} + i\ip{\gamma}{B_t}\right)}
        =
        \Expectation{\exp\left(\frac{i}{2}\tr{J(\Phi)^{\top} \bA_t} + i\ip{Q^{\top}\gamma}{B_t}\right)}
        .
    \end{equation*}
    As \(B^i_t\) are independent, the exponent can be decomposed into independent terms
    \begin{equation*}
        \frac{i}{2}\tr{J(\Phi)^{\top} \bA_t} + i\ip{\gamma}{Q^{\top}B_t}
        =
        i\sum_{k=1}^{n} \set{\varphi_k \bA_t^{2k-1, 2k} +  \ip{Q^{\top}\gamma}{B_t}_{k}}
        + \Indicator{d - 2n > 0} i(Q^{\top} \gamma)_d B^d_t
    \end{equation*}
    where \(n = \lfloor d / 2 \rfloor\) and \(\ip{u}{v}_{k} := u_{2k-1} v_{2k-1} + u_{2k} v_{2k}\). Now the characteristic function factorises 
    \begin{align*}
        \Expectation{\exp\left(\frac{i}{2}\tr{J(\Phi)^{\top} \bA} + i\ip{Q^{\top}\gamma}{B}\right)}
        &=
        \prod_{k=1}^{n} \Expectation{\exp(i \varphi_k \bA^{2k-1, 2k} + i\ip{Q^{\top}\gamma}{B}_{k})} \\
        &\times
        \Expectation{\exp(\Indicator{d - 2n > 0} i(Q^{\top} \gamma)_d B^d)} \\
        &=
        \prod_{k=1}^{n} \cosh\left(\frac{t}{2} \varphi_k \right)^{-1}
        \exp\left(
        -\frac{t}{2}\frac{\sinch}{\cosh}\left(\frac{t}{2} \varphi_k \right)
        \ip{Q^{\top}\gamma}{Q^{\top}\gamma}_{k}
        \right) \\
        &\times
        \exp\left(
        - \Indicator{d - 2n > 0}
        \frac{t}{2}\frac{\sinch}{\cosh}\left(\frac{t}{2} 0\right)
        (Q^{\top}\gamma)_{d}^2
        \right)
    \end{align*}
and we applied the formula for $d=2$, see \cite{helmesLysStochasticArea}.
$$
\Expectation{\exp\left(i \varphi \bA^{1, 2}_{t} + i \gamma_{1} B_{t}^{1} + i \gamma_{2} B_{t}^{2}\right)}
=
\cosh\left( \frac{t}{2} \varphi \right)^{-1} \exp\left( - \frac{t}{2} \frac{\sinch}{\cosh}\left( \frac{t}{2} \varphi \right) (\gamma_{1}^{2} + \gamma_{2}^{2}) \right)
$$
To finish, we express in terms of $\Theta$.
First, $\prod_{k=1}^{n} \cosh\left(\frac{t}{2} \varphi_k \right)^{-1} = \det\left( \cosh \frac{t}{2} |\Theta|\right)^{-1/2}$ as \((\det Q)^2 = 1\) and \(\cosh(0) = 1\) (for \(d\) odd). Finally,
\begin{equation*}
    \sum_{k=1}^{n}
    \frac{\sinch}{\cosh}\left(\frac{t}{2} \varphi_k \right)
    \ip{Q^{\top}\gamma}{Q^{\top}\gamma}_{k}
    +
    \Indicator{d - 2n > 0}
       \frac{\sinch}{\cosh}\left(\frac{t}{2}0\right)
        (Q^{\top}\gamma)_{d}^2
    =
    \gamma^{\top} Q \frac{\sinch}{\cosh}\left(\frac{t}{2} D_{\Phi} \right) Q^{\top} \gamma
\end{equation*}
which written in terms of \(\Theta\) is
\begin{equation*}
    Q \frac{\sinch}{\cosh}\left(\frac{t}{2} D_{\Phi}\right) Q^{\top}
    =
    Q \sinch\left(\frac{t}{2} D_{\Phi}\right) Q^{\top}
    Q \cosh\left(\frac{t}{2} D_{\Phi}\right)^{-1} Q^{\top}
    =
    \sinch\left(\frac{t}{2} |\Theta| \right)
    \cosh\left(\frac{t}{2} |\Theta| \right)^{-1}
    .
\end{equation*}
This follows as $|\Theta| = Q D_{|\Phi|} Q^{\top}$ where $D_{|\Phi|} = \bigoplus_{i=1}^{n} |\varphi_i| I_2 \oplus 0_{d - 2n}$ and $\cosh, \sinch$ are even functions.
\end{proof}

\subsection{Density of L\'evy's Area and its Singular Values}
Although the formula for the characteristic function was known for a long time, just applying the inverse Fourier transform to \eqref{equation:CF of L\'evy's Area and Brownian motion} to find the density results in an integral over \(\Theta \in \so(d) \cong \R^{d(d-1)/2}\) which seems intractable.
Using the change of variables \(\Theta = Q J(\Phi) Q^{\top} \to (Q, \Phi)\) in our proof of the characteristic function, the density can be expressed in terms of an integral over the orthogonal group $Q$ and the singular values $\Phi$.
In the case of $\bA_t$ marginally,  both admit closed forms: the integral over $Q$ is a Harish-Chandra integral (Lemma \ref{lemma:Harish-Chandra Integrals over the Orthogonal Group}) and
the integral over $\Phi$ is in the form of Andr\'eief identity (Lemma \ref{lemma:Andr\'eief Identity}) due to the factorisation of the characteristic function.
Furthermore, properties of $\sech(x) = \cosh(x)^{-1}$ allow for further simplification: it is its own Fourier transform and its derivatives are polynomials of $\sech(x)$.
This yields an explicit expression for the density of $\bA_t$.

\begin{theorem}[Density of L\'evy's Areas]
\label{thm:Density of L\'evy's Area Matrix}
    Let \(\bA_t\) be the \(d \times d\) matrix of L\'evy's areas of $d$-dimensional Brownian motion and express $\bA_t = P J(\Lambda) P^{\top}$ in its block diagonal form, then the density of \(\bA_t\) is given by
    \begin{equation}
        p(A) =
        \left( \frac{\pi}{t} \right)^{d(d-1) / 2}
        \frac{2^{n} C_{d}}{(2\pi)^{n^{2}}}
        \;
        \prod_{1 \le i < j \le n} \frac{\mathrm{sech}\left( \frac{\pi}{t}\lambda_{j} \right)^{2} - \mathrm{sech}\left( \frac{\pi}{t}\lambda_{i} \right)^{2}}{\left( \frac{\pi}{t} \lambda_{i} \right)^{2} - \left( \frac{\pi}{t} \lambda_{j} \right)^{2}}
\prod_{i=1}^{n} w\left( \frac{\pi}{t} \lambda_i \right)
    \end{equation}
    where
    \begin{align}
    \label{equation:C_d}
        C_d &= \begin{cases} 
        \prod_{p=1}^{n-1} (2p)!, & d = 2n \\
        \prod_{p=1}^{n-1} (2p+1)!, & d = 2n+ 1
        \end{cases},\quad \\
        w(\lambda) &=
        \mathrm{sech}\left( \lambda \right)
        \begin{cases}
        1, & d = 2n \\
        \frac{1}{2\pi}\frac{1}{\lambda}\tanh\left(\lambda \right), & d = 2n+ 1
        \end{cases}
        .
    \end{align}
\end{theorem}

\begin{remark}[Non-uniqueness of block diagonalisation and change of variables]
Strictly speaking, the change of variables $\bA_t = P J(\Lambda) P^{\top}$ is not injective as the decomposition is unique up to the sign of $\lambda_k$, multiplication of $P$ by a block diagonal matrix of rotation blocks and permutation of $\lambda$ and the corresponding blocks. To obtain a genuine parametrisation, we consider the transformation to the quotient space of block diagonalisation: the ordered singular $\sigma_{(k)}$ values and $\overline{P} \in O(d) / SO(2)^{n}$.
For considering integrals and distributions, this may be less convenient.
By multiplying by the redundancy factors, we can recover the unordered, signed singular values and matrix in $O(d)$. The overall Jacobian is $g(\Lambda)$, see Lemma \ref{lemma:Jacobian for Block Diagonalisation of Skew-symmetric Matrices}.
{\small
\begin{equation*}
    g(\Lambda) =
    \widetilde{\Delta}_{d}(\Lambda)^2
    \frac{C_d^{-1}}{n!}
    \cdot
    \begin{cases}
    (2\pi)^{n(n-1)},& d = 2n \\
    (2\pi)^{n^2},& d = 2n+1
    \end{cases}
    ,\quad
    \widetilde{\Delta}_{d}(\Lambda)
    =
    \prod_{1 \le i<j\le n} (\lambda_{i}^{2} - \lambda_{j}^{2})
    \begin{cases}
    1, & d = 2n \\
    \prod_{i} \lambda_{i},& d = 2n+1
    \end{cases}
    .
\end{equation*}
}
To recover the singular values $\sigma_k = |\lambda_k|$, the Jacobian is $2^n g(\Sigma)$ (noting that $g(\Lambda) = g(\Sigma)$).
\end{remark}

As the distribution of $\bA_t = P J(\Sigma) P^{\top}$ is invariant to orthogonal conjugation, then $\sigma_k$ and $P$ are independent with $P$ Haar-distributed.
This follows as the density of $\bA_t$ only depends on $\sigma_k$. Changing to the block diagonal form incurs a Jacobian which only depends on $\sigma_k$. Hence, the joint density of $(\Sigma, P)$ trivially factorises, with the density of $P$ being 1, which implies their independence and that $P \in O(d)$ is Haar-distributed. The density of $\sigma_k$ is given by the following corollary.

\begin{corollary}[Density of the Singular Values]
\label{corollary:Joint density of L\'evy's area singular values}
Express $\bA_t = P J(\Lambda) P^{\top}$ in its block diagonal form. The density of its singular values \(\sigma_k = |\lambda_k|\) is
    \begin{equation}
        p(\sigma_1, \dots, \sigma_n) =
        \left( \frac{2}{t} \right)^{n}
        \frac{(-1)^{n(n-1)/2}}{n!}
        \det\left[ \sech\left( \frac{\pi}{t}\sigma_{j} \right)^{2(i-1)} \right]
        \det\left[ \left(\frac{\pi}{t}\sigma_{j}\right)^{2(i-1)} \right]
        \prod_{i=1}^{n} w\left(\frac{\pi}{t}\sigma_i\right)
    \end{equation}
    where
    \begin{align}
        w(\sigma) &=
        \mathrm{sech}\left( \sigma \right)
        \begin{cases}
        1, & d = 2n \\
        \sigma \tanh\left(\sigma\right), & d = 2n+ 1
        \end{cases}
        .
    \end{align}
Furthermore, $\sigma_k$ is independent to the singular vectors $P$ which is a Haar-distributed orthogonal matrix.
\end{corollary}
\begin{proof}
    Perform the change of variables \(\bA = Q J(\Sigma) Q^{\top} \to (Q, \Sigma)\) to its orthogonal matrix and singular values $\Sigma$. By Lemma \ref{lemma:Jacobian for Block Diagonalisation of Skew-symmetric Matrices}, this incurs the Jacobian $2^{n}g(\Sigma)$.
Hence \(p_{\Sigma, Q}(\Sigma, Q) = p_{A}(P J(\Sigma) P^{\top}) 2^{n} g(\Sigma)\).
Since the density of \(\bA\) and the Jacobian only depend on $\Sigma$, then $\Sigma, Q$ are independent with $Q$ uniformly distributed on $O(d)$ (i.e. Haar-distributed). Integrating out \(Q\) yields \(p_{\Sigma}(\Sigma) =  p_{A}(P J(\Sigma) P^{\top}) 2^{n} g(\Sigma)\). Finally, we simplify the expression by recognising the Vandermonde determinants and collecting factors. We incur a minus sign as
{\small
\begin{equation*}
    \prod_{1 \le i < j \le n} \left(\mathrm{sech}\left( \frac{\pi}{t}\sigma_{j} \right)^{2} - \mathrm{sech}\left( \frac{\pi}{t}\sigma_{i} \right)^{2}\right) =
    (-1)^{n(n-1)/2}
    \det\left[ \sech\left( \pi \sigma_{j} / t \right)^{2(i-1)} \right]
    .
\end{equation*}
}
\end{proof}
\begin{proof}[Proof of Theorem~\ref{thm:Density of L\'evy's Area Matrix}]\label{proof of density of levy's area matrix}
Set $\gamma = 0$ in \eqref{equation:CF of L\'evy's Area and Brownian motion} and invert the characteristic function
\begin{equation*}
    p(A) = \frac{1}{(2\pi)^{d(d-1)/2}}\int_{\R^{d(d-1)/2}}
    \exp\left(-\frac{i}{2}\tr{\Theta^{\top} A}\right)
    \det\left( \cosh \frac{t}{2} |\Theta|\right)^{-1/2}
    d \Theta
    .
\end{equation*}
Changing variables \(\Theta = Q J(\Phi) Q^{\top} \to (Q, \Phi)\) incurs the Jacobian \(g(\Phi)\) (Lemma \ref{lemma:Jacobian for Block Diagonalisation of Skew-symmetric Matrices}).
This decomposes $p(A)$ into an integral over the signed singular values of $\Theta$
\begin{equation*}
    p(A) = \frac{1}{(2\pi)^{d(d-1)/2}}
    \int_{\mathbb{R}^{n}} I(A, \Phi) g(\Phi) \prod_{i=1}^{n} \mathrm{sech}\left( \frac{t}{2} \varphi_{i} \right) 
    d\Phi 
\end{equation*}
and an integral over the orthogonal group.
\begin{equation*}
    I(A, \Phi) = \int_{O(d)} \exp\left( \frac{i}{2} \mathrm{tr}\left(Q J(\Phi) Q^{\top} A \right) \right)  dQ
    .
\end{equation*}
To simplify subsequent computations, consider $\widetilde{\bA} = \bA_{\pi}$, denoting its singular values with $x_i$. To recover general $t$, we use scale invariance of Brownian motion $\widetilde{\bA} \eqdist \pi \bA_t / t$ and rescale at the end.

\textbf{Integral over the orthogonal group}
Applying the Harish-Chandra integral (Lemma \ref{lemma:Harish-Chandra Integrals over the Orthogonal Group}) yields
\begin{equation*}
I(\widetilde{A}, \Phi)
=
\frac{C_d (-1)^{n(n-1) / 2}}{\widetilde{\Delta}_{d}(X) \widetilde{\Delta}_{d}(\Phi)}
\begin{cases}
\det[\cos(x_{i}\varphi_{j} )], & d = 2n \\
\det[\sin(x_{i} \varphi_{j} )], & d = 2n+ 1
.
\end{cases}
\end{equation*}
The output simplifies when multiplied with \(g(\Phi)\). We also relate \(\widetilde{\Delta}(\Phi)\) to the Vandermonde determinant of $\Phi^2$ and insert the additional factor from the odd case into the determinant.
\begin{equation*}
    I(\widetilde{A}, \Phi) g(\Phi) =
    \widetilde{C}_d
    \frac{\det[\varphi_{j}^{2(i-1)}]}{\det[x_{j}^{2(i-1)}]}
        \begin{cases}
        \det[\cos(x_{i} \varphi_{j} )], & d = 2n \\
        \det\left[ \frac{\varphi_{j}}{x_{i}}\sin(x_{i}\varphi_{j} ) \right], & d = 2n+ 1
        \end{cases}
\end{equation*}
where
\begin{equation*}
    \widetilde{C}_{d}
    =
    \frac{(-1)^{n(n-1) / 2}}{n!}
    \begin{cases}
    (2\pi)^{n(n-1)}, & d = 2n \\
    (2\pi)^{n^2}, & d = 2n + 1
    \end{cases}
    .
\end{equation*}

\textbf{Integral over the signed singular values} The resulting integral over $\Phi$ is in the form of Andr\'eief Identity (Lemma \ref{lemma:Andr\'eief Identity}).
In the case \(d = 2n\) is even, we show
\begin{equation*}
    \begin{split}
    \int_{\mathbb{R}^{n}}
    \det[\varphi_{j}^{2(i-1)}]
    \det[\cos(x_{i} \varphi_{j} )]
    \prod_{i=1}^{n}\mathrm{sech}\left( \frac{\pi}{2} \varphi_{i} \right) d\Phi
    =
    n! \det\left[
    \int_{\R}
    \varphi^{2(i-1)}
    \cos(x_j \varphi)
    \sech\left( \frac{\pi}{2} \varphi \right) d \varphi
    \right] \\ =
    n!
    2^{n}
    \prod_{p=1}^{n-1} (2p)!
    \prod_{i=1}^{n} \mathrm{sech}\left( x_{i} \right)
    \prod_{i < j} \left( \mathrm{sech}\left(x_{i} \right)^{2} - \mathrm{sech}\left(x_{j} \right)^{2} \right)
    .
    \end{split}
\end{equation*}
By taking derivatives of the Fourier transform of $\sech$ $F(x) = \int_{\mathbb{R}}\cos(x_{j} \varphi)\sech\left( \frac{\pi}{2} \varphi \right) d \varphi = 2 \sech(x)$,
\begin{equation*}
    \int_{\mathbb{R}}
    \varphi^{2(i-1)}
    \cos(x_{j} \varphi)
    \sech\left( \frac{\pi}{2} \varphi \right) d \varphi
    =
    (-1)^{(i-1)}
     \partial_{x}^{2(i-1)} F(x_j)
     .
\end{equation*}
Next the derivative of $\sech$ are polynomials in $\sech$ where $W_{i, j}$ are given by Lemma \ref{lemma:Derivatives of sech}
\begin{equation*}
(-1)^{(i-1)} \partial_{x}^{2(i-1)} F(x_j)
=
2
\sum_{l=0}^{i-1} (-1)^{i-1-l}W_{2l+1, i-1-l} (2l)! \;\mathrm{sech}\left( x_{j} \right)^{2l+ 1}
.
\end{equation*}
As the \(i\)-th row is a polynomial in \(\sech\) of degree \(2i - 1\), we can simplify the determinant by removing lower degree terms. This leaves only the leading degree which have coefficients $(2i-2)!$, as $W_{2i-1, 0} = 1$, and reveals a Vandermonde determinant.
\begin{equation*}
\begin{split}
\det\left[
2
(2i-2)! \;\mathrm{sech}\left( x_{j} \right)^{2(i-1)+ 1}\right]
=
2^n
\prod_{p=1}^{n-1} (2p)!
\prod_{i=1}^{n} \mathrm{sech}\left( x_{i} \right)
\det\left[
\mathrm{sech}\left( x_{j} \right)^{2(i-1)}\right]
\end{split}
\end{equation*}
Similarly, in the case \(d = 2n + 1\) is odd
{\small
\begin{align*}
\int_{\mathbb{R}^{n}}
\det[\varphi_{j}^{2(i-1)}]
\det\left[ \frac{\varphi_{j}}{x_{i}}\sin(x_{i} \varphi_{j} ) \right]
\prod_{i=1}^{n}\mathrm{sech}\left( \frac{\pi}{2} \varphi_{i} \right) d\Phi
=
n! \det\left[
\int_{\R}
\varphi^{2(i-1)}
\frac{\varphi}{x_j}
\sin(x_j \varphi)
\sech\left( \frac{\pi}{2} \varphi \right) d \varphi
\right] \\ =
n!
2^{n}
\prod_{p=1}^{n-1} (2p+1)!
\prod_{i=1}^{n} \mathrm{sech}\left( x_{i} \right)
\frac{\tanh\left( x_{i} \right)}{x_{i}}
\prod_{i < j} \left( \mathrm{sech}\left( x_{i} \right)^{2} - \mathrm{sech}\left( x_{j} \right)^{2} \right)
.
\end{align*}
}
We consider {\small $\int_{\mathbb{R}} \varphi^{2(i-1)+1} \sin(x_{j} \varphi) \mathrm{sech}\left( \frac{\pi}{2} \varphi \right) d \varphi
=
(-1)^{i} \partial_{x}^{2(i-1)+1} F(x_j)$} and Lemma \ref{lemma:Derivatives of sech} to express
$$
\partial_{x}^{2(i-1)+1} F(x_j) = 2
\tanh\left( x_{j} \right)
\sum_{l=0}^{i-1} (-1)^{i-1-l}W_{2l+1, i-1-l} (2l+1)! \;\mathrm{sech}\left( x_{j} \right)^{2l+ 1}
.
$$
Applying the same technique with determinants yields {\small $\det\left[
2
\tanh\left( x_{j} \right)
(2i - 1)! \;\mathrm{sech}\left( x_{j} \right)^{2(i-1)+ 1}\right]$ }.

\textbf{Constant factors}
The \(n!\) cancels and the \((-1)^{n(n-1)/2}\) is used to flip the products in the \(\sech(x)^2\) which ensures the expression is positive as \(\sech(|x|)\) is decreasing. The factorials form \(C_d\). 
Next we consider the $(2\pi)^{d(d-1)/2}$ from the inverse Fourier transform.
In the even case \(d = 2n, d(d-1) / 2= 2n^2 -n\) and the overall constant factor is $2^{n} (2\pi)^{-n^2}$ and in the odd case \(d = 2n+1, d(d-1) / 2 = 2n^2 +n\) and the factor is $2^{n} (2\pi)^{-n^2 - n}$.
In the odd case we absorb $(2\pi)^{-n}$ into the weight function.

Finally, rescaling back to $\bA_t$ yields a factor of $(\pi / t)^{d(d-1)/2}$ and we set $x_i = \pi \lambda_i / t$.
\end{proof}

\subsection{Density of L\'evy's Area and Brownian Motion} \label{section:Density of L\'evy's areas and Brownian increments}
The same technique can be applied to study the density of \((B_t, \bA_t)\). However, the associated integral over the orthogonal group is more complicated and the authors are unaware if it has a closed form.

\begin{lemma}[Density of L\'evy's Areas and Brownian Motion]
\label{lemma:Joint Density of L\'evy's area and Brownian Motion}
Let \(\bA_t\) be the \(d \times d\) matrix of L\'evy's areas of $d$-dimensional Brownian motion $B_t$. Express \(\bA_t = P J(\Lambda) P^{\top}\) in its block diagonal form, then the density of \((B_t, \bA_t)\) is given by
    \begin{equation}
        p(B, A)
        =
        \frac{1}{(2\pi)^{d^2/2} \; t^{d/2}}
        \int_{\R^n}
        I(J(\Lambda), P^{\top}B, J(\Phi))
        g(\Phi)
        \prod_{k=1}^n \sinch\left(\frac{t}{2} \varphi_k\right)^{-1}
        d \Phi
    \end{equation}
    where \(g(\Phi)\) is the block diagonalisation from Jacobian (Lemma \ref{lemma:Jacobian for Block Diagonalisation of Skew-symmetric Matrices}) and 
    \begin{equation}
I(J(\Lambda), B, J(\Phi))
=
\int_{O(d)}
\exp\left( \mathrm{tr}\frac{i}{2} Q J(\Phi) Q^{\top} J(\Lambda) - \frac{1}{2t}B^{\top}Q
\frac{\cosh}{\sinch}\left( \frac{t}{2} D_{\Phi} \right)Q^{\top}B\right)
dQ
.
    \end{equation}
\end{lemma}
\begin{proof}
Take the joint characteristic function \(\varphi_{B, A}(\gamma, \Theta)\) (\eqref{equation:CF of L\'evy's Area and Brownian motion}) and invert wrt to $\gamma$ and then $\Theta$. The integral over \(\gamma\) corresponds to the density of a centred normal distribution with covariance \(t  \sinch\left(\frac{t}{2} |\Theta| \right) \cosh\left(\frac{t}{2} |\Theta| \right)^{-1}\).
Hence,
\begin{align*}
    p(B, A)
&=
(2\pi)^{-d(d-1)/2 - d} \int_{\mathbb{R}^{d(d-1)/2}} \int_{\mathbb{R}^{d}}    
\exp\left(-\frac{i}{2}\tr{\Theta^{\top} \bA} - i\ip{\gamma}{B}\right)
\varphi_{B, A}(\gamma, \Theta)
d\gamma d\Theta \\
&= (2\pi)^{-d(d-1)/2 - \frac{d}{2}} t^{-\frac{d}{2}} \int_{\mathbb{R}^{d(d-1)/2}}
    \det\left( \sinch \frac{t}{2} |\Theta|\right)^{-1/2} \\
&\times
    \exp\left(
    -\frac{i}{2}\tr{\Theta^{\top} \bA}
    -\frac{1}{2t}
    B^{\top}
    \cosh\left(\frac{t}{2} |\Theta| \right)
    \sinch\left(\frac{t}{2} |\Theta| \right)^{-1}
    B
\right)
d\Theta
\end{align*}
This results in an integral over \(\R^{d(d-1)/2}\).
Setting \(\Theta = Q J(\Phi) Q^{\top} \to (Q, \Phi)\) yields
\begin{equation*}
p(B, A) = \frac{1}{(2\pi)^{d^2/2}\; t^{d/2}}
    \int_{\mathbb{R}^{n}}
    I(A, B, \Phi) g(\Phi) \prod_{i=1}^{n} \sinch\left( \frac{t}{2} \varphi_{i} \right)^{-1}
d\Phi
.
\end{equation*}
Block diagonalise $A = P J(\Lambda) P^{\top}$, by setting $P^{\top} Q \to Q$. The Jacobian is 1 as the Haar measure is invariant to orthogonal transformation which yields $I(J(\Lambda), P^{\top} B, J(\Phi))$.
\end{proof}

\begin{remark}[Simplification for $d=2$] \label{remark:Orthogonal Group Integral d=2}
When \(d = 2\) the second term no longer depends on $Q$ as $D_{\Phi} = \varphi I_2$ and $Q D_{\Phi} Q^{\top} = \varphi I_2$ which leaves a Harish-Chandra integral.
\begin{equation*}
I(J(\Lambda), W, J(\Phi))
=
\exp\left(-\frac{\lp[2]{B}^2}{2t} \frac{\cosh}{\sinch}\left( \frac{t}{2} \varphi \right) \right)
\int_{O(2)}
\exp\left( \mathrm{tr}\frac{i}{2} Q J(\Phi) Q^{\top} J(\Lambda)\right)
dQ
.
\end{equation*}
This recovers the joint density for \(d = 2\) as in \cite{gainesRandomGenerationStochastic1994}.
\begin{equation*}
p(B, A)
=
\frac{1}{(2\pi)^2 t}
\int_{\R}
\cos(\sigma \varphi)
\exp\left(-\frac{\lp[2]{B}^2}{2t} \frac{\cosh}{\sinch}\left( \frac{t}{2} \varphi \right) \right)
\sinch\left(\frac{t}{2} \varphi\right)^{-1}
d \varphi
\end{equation*}
\end{remark}

\begin{remark}[Weingarten Calculus]
For higher dimensions, $I(J(\Lambda), B, J(\Phi))$ is in principle computable via Weingarten Calculus \cite{collinsIntegrationRespectHaar2006, collinsPropertiesOrthogonalWeingarten2009} which gives expression for mixed moments of the entries of an orthogonal matrix.
I.e., let \(C = -\frac{1}{2t}\frac{\cosh}{\sinch}\left( \frac{t}{2} D_{\Phi} \right)\).
\begin{equation}
I(J(\Lambda), B, J(\Phi))
= \sum_{m, n = 1}^{\infty} \frac{1}{m! n!} \int_{O(d)} \mathrm{tr}\left( \frac{i}{2} J(\Lambda) Q J(\Phi) Q^{\top} \right)^{m} (B^{\top} Q C Q^{\top} B)^{n}\, dQ
\end{equation}
where
\begin{equation}
    \mathrm{tr} Q J(\Lambda) Q^{\top} J(\Phi)
    =
    \sum_{i, j \in [n]} \lambda_i \varphi_j \det Q_{[ij]},\quad
    Q_{[ij]}
    :=
    \begin{bmatrix}
        Q_{2i -1 , 2j -1} & Q_{2i -1 , 2j} \\
        Q_{2i , 2j -1} & Q_{2i , 2j}
    \end{bmatrix}
\end{equation}
and
\begin{equation}
    B^{\top} Q C Q^{\top} B
    = \sum_{i \in [n]} c_i \left[(Q^{\top} B)_{2i-1}^2 + (Q^{\top} B)_{2i}^2 \right]
    .
\end{equation}
However, it is unclear if this would yield a form that is useful for the remaining integral over \(\Phi\).
\end{remark}

As another simplification, we consider the marginal density of $(B, \Sigma)$ as integrating over $P$ removes the coupling in $I$. When $d=2n$, this yields an integral over the simplex (Lemma \ref{lemma:Hermite-Genocchi Integral}) which has an explicit form. However, the integral over $\Phi$ remains a challenge as it is not directly compatible with Andr\'eief identity.
\begin{lemma}[Density of Brownian Motion and the Singular Values] \label{lemma: Density of u and sigma_k}
Let $d = 2n$ and let \(h(\varphi) = \frac{\cosh}{\sinch}(\frac{t}{2} \varphi)\). The density of $B_t$ and the singular values $\sigma_i$ of $\bA_t$ is
\begin{equation}
\begin{split}
p(B, \sigma_1, \dots, \sigma_n)
=
\frac{(-1)^{(n-1)(n-2) / 2}\; C_d^{-1}}{2t\, \pi^{2n}\;n\cdot n! \lp[2]{B}^{2(n-1)}}
\det[\sigma_{j}^{2(i-1)}]
\int_{\mathbb{R}^{n}}
\det[\varphi_{j}^{2(i-1)}]
\det[\cos(\sigma_{i}\varphi_{j} )] \\
\times \prod_{k=1}^n \sinch\left(\frac{t}{2} \varphi_k\right)^{-1}
\left( \sum_{k=1}^{n} \frac{\exp(- \frac{\lp[2]{B}^2}{2t} \;h(\varphi_{k}))}{\prod_{l \neq k} (h(\varphi_{k}) - h(\varphi_{l}))} \right)
d \Phi
\end{split}
\end{equation}
where $C_d$ is the same as in \eqref{equation:C_d}.
\end{lemma}
\begin{proof}
Consider the density $p_{B, \Sigma, P}(B, \Sigma, P) = 2^{n}\, p_{B, A}(B, P J(\Sigma) P^\top) g(\Sigma)$.
To get the density of $(B, \Sigma)$, we integrate over $P$ which only affects the second factor of $I(J(\Sigma), P^{\top}B, J(\Phi))$
\[
\int_{O(d)}
\exp\left(
- \frac{1}{2t}B^{\top}(PQ) \frac{\cosh}{\sinch}\left( \frac{t}{2} D_{\Phi} \right)(PQ)^{\top} B\right)
dP
.
\]
The change of variables $PQ \to P^{\top}$ (which has unit Jacobian) removes dependence on $Q$. Therefore the intractable integral $I$ now factorises into two integrals over $O(d)$
\[
\int_{O(d)} I(\Sigma, P^{\top} B, \Phi) dP
=
\int_{O(d)}
\exp\left( \mathrm{tr}\frac{i}{2} Q J(\Phi) Q^{\top} J(\Sigma)\right) dQ
\cdot
\int_{O(d)}
\exp\left(
- \frac{1}{2t}(PB)^{\top} \frac{\cosh}{\sinch}\left( \frac{t}{2} D_{\Phi} \right) (PB)\right)
dP
.
\]
The first term is a Harish-Chandra integral and the second term can be reduced into the Hermite-Genocchi integral. As $PB \eqdist \lp[2]{B} S$ where $S$ is uniformly distributed on the unit $d-1$-dimensional sphere, i.e., $S \sim \rm{Dirichlet}(1/2, \dots, 1/2)$, the integrand becomes
\[
\exp\left(
- \frac{\lp[2]{B}^2}{2t} S^{\top} \frac{\cosh}{\sinch}\left( \frac{t}{2} D_{\Phi} \right) S\right)
=
\exp\left(
- \frac{\lp[2]{B}^2}{2t}
\set{\sum_{k=1}^{n} h(\varphi_k)(S_{2k-1}^2 + S_{2k}^2) + \Indicator{d = 2n+1} S_d^2}
\right)
.
\]
Aggregating into $Z_k = S_{2k-1}^2 + S_{2k}^2$, $Z$ follows the $\rm{Dirichlet}(1, \dots, 1)$ distribution which is equivalent to being uniform on the $n-1$-dimensional unit simplex. Noting the volume of the simplex is $(n-1)!$ and letting $\lambda = \lp[2]{B}^2 / (2t)$,
\[
I'(B, \Phi)
:=
\int_{O(d)}
\exp\left(
- \frac{1}{2t}(PB)^{\top} \frac{\cosh}{\sinch}\left( \frac{t}{2} D_{\Phi} \right) (PB)\right)
dP
=
(n-1)! \int_{\Delta_{n-1}} \exp\left(
- \lambda \sum_{k=1}^{n} h(\varphi_k) z_k
\right)
dz
.
\]
Applying the Hermite-Genocchi integral to $f(x) = \exp(-x)$ (see Lemma \ref{lemma:Hermite-Genocchi Integral}), we obtain
\begin{align*}
I'(B, \Phi)
=
\frac{(n-1)!\;(-1)^{n-1}}{\lambda^{n-1}} \sum_{k=1}^{n} \frac{\exp(-\lambda \;h(\varphi_{k}))}{\prod_{l \neq k} (h(\varphi_{k}) - h(\varphi_{l}))}
.
\end{align*}
Now we can express the density as
\[
p(B, \Sigma)
=
\frac{2^{n}g(\Sigma)}{(2\pi)^{d^2/2} \; t^{d/2}}
\int_{\R^n}
I(\Sigma, \Phi)
I'(B, \Phi)
g(\Phi)
\prod_{k=1}^n \sinch\left(\frac{t}{2} \varphi_k\right)^{-1}
d \Phi
.
\]
In the even case,
\[
g(\Phi) =
\det[\varphi_{j}^{2(i-1)}]^{2}
\frac{C_d^{-1}}{n!}
(2\pi)^{n(n-1)},\quad
I(\Sigma, \Phi)
=
\frac{C_d (-1)^{n(n-1) / 2}}{\det[\sigma_{j}^{2(i-1)}] \det[\varphi_{j}^{2(i-1)}]}
\det[\cos(\sigma_{i}\varphi_{j} )]
.
\]
Collecting factors yields the result.
\end{proof}

Like in Corollary \ref{corollary:Joint density of L\'evy's area singular values}, we can integrate out the orthogonal matrix of $\bA_t = P J(\Sigma) P^{\top}$ and recover quantities that are independent to $P$. These are the singular values $\sigma_k$ and $u_k = (P_{2k-1}^{\top} B)^2 + (P_{2k}^{\top} B)^2$ which are the squared norm of the projection of $B_t$ onto each singular subspace of $\bA_t$.

\begin{corollary} [Density of the Singular Values and Squared Subspace Norms] The density of the singular values \(\sigma_k\) of the matrix of L\'evy's areas \(\bA_t = P J(\Sigma) P^{\top}\) and the squared subspace norms $u_k = (P_{2k-1}^{\top} B_t)^2 + (P_{2k}^{\top} B_t)^2$ of the Brownian motion $B_t$ is given by
\label{corollary:joint density of sigma, u}
    \[
p(u_{1}, \dots, u_n, \sigma_1, \dots, \sigma_n) =
\frac{g(\Sigma) (2\pi)^n}{(2\pi)^{d^2/2} \; t^{d/2}}
\int_{\R^n}
I(J(\Sigma), \overline{u}, J(\Phi))
g(\Phi)
\prod_{k=1}^n \sinch\left(\frac{t}{2} \varphi_k\right)^{-1}
d \Phi
\]
where $\overline{u} = [\sqrt{u_1}, 0, \dots, \sqrt{u_n}, 0]$ (adding $\sqrt{u_{n+1}} = |P_{2n+1}^{\top} B|$ if $d = 2n+1$). Furthermore, $\sigma_k, u_k$ are independent to $P$ which is Haar-distributed.
\end{corollary}
\begin{proof}
Again perform the change of variables \(\bA_t = P J(\Sigma) P^{\top} \to (P, \Sigma)\) to its orthogonal matrix and singular values $\Sigma$.
Furthermore, the Jacobian is unchanged when we change from $B \to W = P^{\top}B$. Hence $p_{W, \Sigma, P}(W, \Sigma, P) = 2^{n}g(\Sigma)\, p_{B, A}(PW, PJ(\Sigma)P^{\top})$.
Next consider the following change of variables, which incurs a Jacobian of $(1/2)^n$ and $W = R(\theta)\, \overline{u}$ where $R(\theta) = \bigoplus_{k=1}^{n} R(\theta_k) \oplus I_{d-2n}$.
\[
\begin{bmatrix}
    W_{2k-1} \\
    W_{2k}
\end{bmatrix}
=
R(\theta_k)
\begin{bmatrix}
    \sqrt{u_k} \\
    0
\end{bmatrix},\quad
R(\theta_k)
=
\begin{bmatrix}
    \cos(\theta_k) & - \sin(\theta_k) \\
    \sin(\theta_k) & \cos(\theta_k)
\end{bmatrix}
\]
From Lemma \ref{lemma:Joint Density of L\'evy's area and Brownian Motion}, $p_{B, A}(PW, PJ(\Sigma)P^{\top})$ is invariant to $P$ as the dependence is only through
$$
I(J(\Sigma), P^{\top}PW, J(\Phi)) = I(J(\Sigma), W, J(\Phi))
.
$$
In particular, it is invariant to $\theta_k$ (which corresponds to multiplying $P$ by $2\times2$ rotation blocks). Integrating out $\theta_k$ incurs a factor of $(2\pi)^n$, hence 
$p_{u, \Sigma, P}(u, \Sigma, P) = \pi^n p_{W, \Sigma, P}(\overline{u},\Sigma, P)$. As the density is independent of $P$, then $u, \Sigma$ must be independent to $P$ and again $P$ is uniformly distributed on $O(d)$.
Integrating out $P$ yields $p_{u, \Sigma}(u, \Sigma) = g(\Sigma)\, p_{B, A}(P\overline{u}, PJ(\Sigma)P^{\top})$. Overall
$$
p_{u, \Sigma}(u, \Sigma) = (2\pi)^n g(\Sigma)\, p_{B, A}(P\overline{u}, PJ(\Sigma)P^{\top})
.
$$
\end{proof}

\subsection{Extension to other Stochastic Processes}
\label{section:Extension to other stochastic processes}
The factorisation argument holds for other stochastic processes with independent components and orthogonal invariance. This yields a formula for the density of the singular values of their matrix of L\'evy's areas. Such stochastic processes correspond to vector-valued processes whose components are i.i.d centred Gaussian processes (GP), which we assume are regular enough for L\'evy's area to be defined \cite{ferreiro-castillaLevyAreaGaussian2011}. This includes Brownian motion, as well Ornstein-Uhlenbeck processes and fractional Brownian motions with sufficiently high smoothness parameters. 
To see this, let $\widetilde{w}_t(\varphi)$ be the characteristic function of L\'evy's area between two independent copies of the GP then by repeating the proof of Theorem \ref{theorem:new proof of cf}
$$
\Expectation{\exp\left( \frac{i}{2} \mathrm{tr}(\Theta^{\top} \bA_{t}) \right)} = \Expectation{\exp\left( i \sum_{k=1}^{n} \varphi_{k}\bA^{2k-1, 2k} \right)} = \prod_{k=1}^{n} \widetilde{w}_{t}(\varphi_{k})
.
$$
Let $w_t$ be the density of $\bA^{i, j}_t$, to simplify the resulting determinant from Andr\'eief identity
$$
w_{t}(\sigma) = \frac{1}{2\pi}\int_{\R} \cos(\varphi \sigma) \widetilde{w}(\varphi) d \varphi \implies 2\pi\,(-1)^{i-1} \partial_{\sigma}^{2(i-1)} w_{t}(\sigma) = \int_{\R} \varphi^{2(i-1)} \cos(\sigma \varphi) \widetilde{w}(\varphi) d\varphi
.
$$
This again results in a determinantal form for the density of the singular values (see \eqref{equation:C_d} for $C_d$)
$$
p(\sigma_1, \dots, \sigma_n) = \frac{2^{n}}{n!}\, C_{d}^{-1} (-1)^{n(n-1)/2} \det[\sigma_{j}^{2(i-1)}] \det[\partial_{\sigma}^{2(i-1)} w_{t}(\sigma_{j})]
$$
This is in the form of a biorthogonal ensemble \cite{borodin1998biorthogonal}. In the case of Brownian motion, this will allow us to identify the singular values as a determinantal point process in Section \ref{section:Biorthogonal Ensemble}. Upon setting $x_{i} = \sigma_{i}^{2}$, barring the factor of $\prod x_i^{-1/2}$, it is also in the form of a "P\'olya ensemble on $M$" \cite{forsterPolynomialEnsemblesPolya2021}
$$
p(x_{1}, \dots, x_{n}) \propto \det[x_{j}^{i-1}] \det[(x^{1/2} \partial_{x} (x^{1/2} \partial_{x}))^{2(i-1)} w(x_{j}^{1/2})] \prod_{i=1}^{n} \frac{1}{x_{i}^{1/2}}
.
$$
These correspond to eigenvalue distributions which are closed under convolutions of their associated matrices. For L\'evy's areas, a natural analogue of the matrix convolution is the associated SDE which we study in the subsequent section.

\section{The Dynamics of the Singular Values} \label{sec:SDE}
The SDE for $t \mapsto (B_t,\bA_t)$ is classical. In view of Corollary \ref{corollary:joint density of sigma, u}, we identify the time evolution of its orthogonal invariants $(\sigma_1(t),\ldots,\sigma_n(t), u_1(t),\ldots,u_n(t))$, the singular values and squared subspace norms,
as an interacting particle system; see Figure \ref{fig:Singular Value Dynamics}. Using this SDE and the Kolmogorov forward equation, their joint density satisfies a PDE which we spell out, see Proposition \ref{prop:Joint density PDE}.

\begin{figure}[ht]
    \centering
    \includegraphics[width=\linewidth]{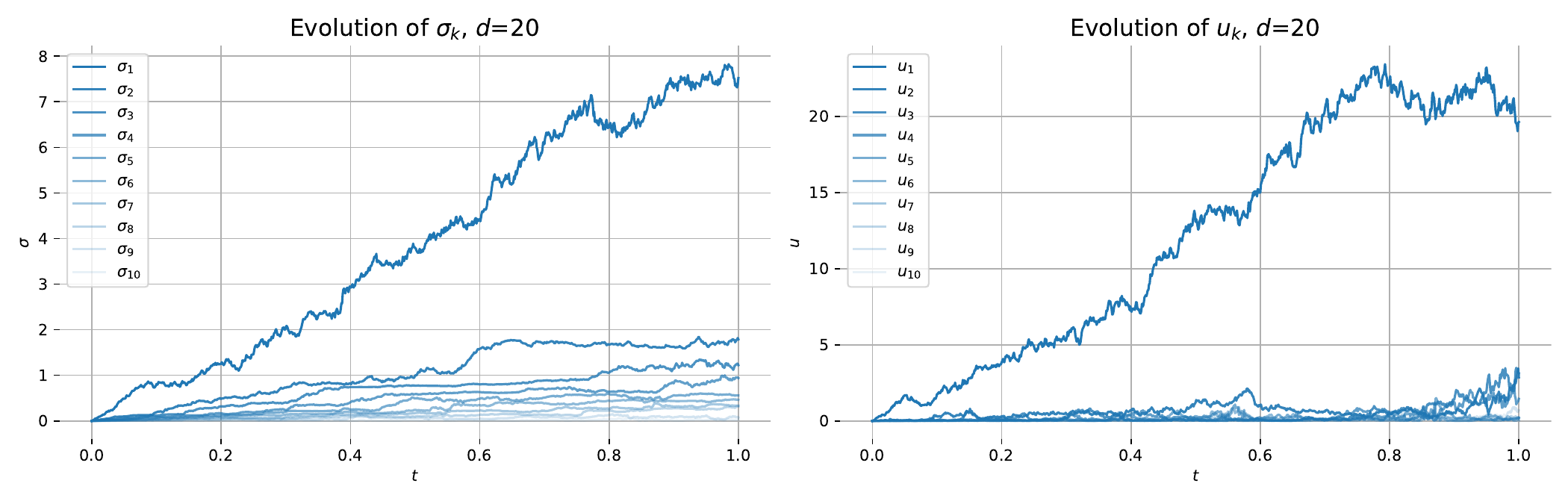}
    \caption{Evolution of the singular values \(\sigma_k\) of L\'evy's area \(\bA_t\) and squared subspace norms \(u_k\) of $B_t$ over \(t \in [0,1]\). The largest singular value $\sigma_{1}$ exhibits linear growth and is an order of magnitude larger than the remaining $\sigma_k$. While the $u_k$ do not preserve an ordering, $u_1$ seems to be generally larger than the remaining $u_k$.}
    \label{fig:Singular Value Dynamics}
\end{figure}

Throughout, we assume that $d$ is even. In the odd case, the only change is that one singular value is always $0$. Furthermore, we instead study the evolution of the signed singular values $\lambda_k(t)$ as they are more natural to analyse and $\sigma_k$ is simply $|\lambda_k|$. This is because they correspond to L\'evy's areas which can be signed. For example, when $d=2$, $\bA^{1, 2}_t = \lambda_1(t)$. We also order $\lambda(t)$ by decreasing absolute value, that is $|\lambda|$ lives on the Weyl chamber, and say
$$\lambda \in \mathcal{W}_n = \set{\lambda \in \R^{n}: |\lambda_1| \ge \cdots \ge |\lambda_n| \ge 0 }.$$
Next recall that the squared subspace norms are $u_k(t) = \lp[2]{P_k^{\top}(t) B_t}^2$, where $P_k(t)$ is the projection onto the singular subspace of $\sigma_k(t)$. If we assume $\sigma_k(t)$ are distinct, that is $\lambda \in \mathcal{W}^\mathrm{o}_n$, and consider the block diagonal form $\bA_t = Q_t J(\Lambda_t) Q_t^{\top}$, then
$$u_k(t) = (Q_{2k-1}^{\top}(t) B_t)^2 + (Q_{2k}^{\top}(t) B_t)^2.$$
In fact this representation will be valid for all $t > 0$ as we will show $\lambda_k(t)$ are always distinct, i.e., $\lambda \in \mathcal{W}_n^{\rm o}$
\footnote{This does not forbid $\lambda_n = 0$. E.g., for $n=1$ $\mathcal{W}_n = \R = \mathcal{W}_n^{\rm o}$.} as they repel (a classic behaviour of the eigenvalues of a random matrix, e.g. Dyson Brownian motion). We make this precise with the following theorem.

\begin{theorem}[An Interacting Particle System]\label{theorem:SDE of Singular values and radii}
For $t>0$ and $d = 2n$, consider the block diagonal form of the matrix of L\'evy's areas $\bA_t = Q_t J(\Lambda_t) Q_t^{\top}$. 
Then the following hold.

\begin{enumerate}
    \item The processes
\begin{equation}
    d\gamma_{k}(t) = \frac{B_t^{\top} P_{k}(t) dB_t}{\sqrt{u_k(t)}},\quad
d\theta_{k}(t) = \frac{B_t^{\top} J_{k}(t) dB_t}{\sqrt{u_k(t)}}
\end{equation}
are independent Brownian motions where $P_k(t)$ and $J_k(t)$ are the projection and rotation on the \(k\)-th singular subspace $\mathrm{span}(Q_{2k-1}, Q_{2k})$. These are given by
\begin{equation}
    P_{k}(t) = Q_{2k-1}(t) Q_{2k-1}^{\top}(t) + Q_{2k}(t) Q_{2k}^{\top}(t),\quad
    J_{k}(t) = Q_{2k-1}(t) Q_{2k}^{\top}(t) - Q_{2k}(t) Q_{2k-1}^{\top}(t).
\end{equation}
\item For $t = 0$, $\lambda_1(0),\ldots,\lambda_n(0), u_1(0), \dots, u_n(0)=0$ and for $t > 0$, $\lambda(t) \in \mathcal{W}^\mathrm{o}_n$ are distinct, $u_k(t) = \lp[2]{P_k^{\top}(t)B_t}^2 > 0$ are positive and $(\lambda_1(t),\ldots,\lambda_n(t),u_1(t),\ldots,u_n(t))$ solves the following SDE. 
\begin{equation}\label{eq:SDE Singular values}
\begin{split}
d \lambda_{k}(t)
&=
\frac{\sqrt{ u_{k} }}{2} d\theta_{k}(t)
+
\frac{\lambda_{k}}{4} \sum_{l \neq k} \frac{u_{k} + u_{l}}{\lambda^{2}_{k} - \lambda^{2}_{l}} dt \\
du_{k}(t) &= 2\sqrt{ u_{k} }\; d\gamma_{k}(t)
+ \sum_{l \neq k } \frac{ u_{l}  \sqrt{u_{k} } \; \lambda_{k}d\theta_{k}(t) +  u_{k} \sqrt{u_{l} } \; \lambda_{l}d\theta_{l}(t)}{\lambda^{2}_{k} - \lambda^{2}_{l}} \\
&+  \left\{  2+
\frac{\lambda^{2}_{k}}{2}
\sum_{\substack{l, r \neq k \\ l \neq r}} \frac{u_{l}(u_{r} + 2u_{k})}{(\lambda^{2}_{k} - \lambda^{2}_{l})(\lambda^{2}_{k} - \lambda^{2}_{r})}
+ \frac{1}{2}\sum_{l \neq k} \frac{\lambda^{2}_{k} u_{l}^{2} - \lambda^{2}_{l} u_{k}^{2}}{(\lambda^{2}_{k} - \lambda^{2}_{l})^2}
\right\} dt
\end{split}
\end{equation}
\end{enumerate}
\end{theorem}

\begin{proof}[Proof of Theorem \ref{theorem:SDE of Singular values and radii}]
The Brownian motions follow by considering their quadratic variation, see Lemma \ref{lemma:Radial and angular Brownian motions}.
Next, let $\mathcal{D}_n$ be the domain of $(\lambda, u)$ such that $\lambda \in \mathcal{W}_n$ and $u_k \ge 0$. Its interior $\mathcal{D}_n^{\rm{o}}$ corresponds to $\lambda \in \mathcal{W}^\mathrm{o}_n$ and $u_k > 0$.
For $t_0 > 0$, as the joint density of $(\lambda, u)(t_0)$ (Lemma \ref{lemma:Joint Density of L\'evy's area and Brownian Motion}) is absolutely continuous with respect to the product Lebesgue measure, $(\lambda, u)(t_0) \in \mathcal{D}_n^{\rm{o}}$ almost surely.
Applying It\^o formula in Lemma \ref{lemma:Derivation of singular value SDE} shows $(\lambda, u)$ satisfies the SDE until it hits the boundary of $\mathcal{D}_n$. By Lemma \ref{lemma:SDE well-posed, interior IC}, the SDE started at $\mathcal{D}_n^{\rm{o}}$ has a unique strong solution and the solution never hits the boundary, that is for all $t >  t_0$, $(\lambda, u)(t) \in \mathcal{D}_n^{\rm{o}}$. By taking a countable sequence of $t_n \downarrow 0$ and extending the unique solution from $t_{n-1}$ up to $t_n$, we deduce for all $t > 0$, $(\lambda, u)(t) \in \mathcal{D}_n^{\rm{o}}$ and satisfies the SDE.
\end{proof}
This agrees with the well-known case of $d=2$, see \cite[Lemma 2.1.1]{baudoinStochasticAreasHorizontal2023}, which we have now extended to $d = 2n$.
Furthermore, $u(t) = \sum_{k=1}^{n} u_k(t) = \lp[2]{B_t}^2$ should be a $2n$-dimensional squared Bessel process.
\[
d u(t) = 2 B_t^{\top} dB_t + 2n\,dt = 2 \sqrt{u(t)} d\gamma(t) + 2n\,dt,\quad
d\gamma(t) = \frac{B_t^{\top} dB_t}{\lp[2]{B_t}}
\]
Indeed, this follows from the SDE and will be important in proving the well-posedness of the SDE.
\begin{lemma}[SDE for $u$]
\label{lemma: SDE for u}
Let $u(t) = \sum_{k=1}^{n} u_k(t)$, then
\[
du(t) = 2 \sqrt{u(t)} \,d \gamma(t) + 2n\, dt
\]
where $\gamma$ is a Brownian motion independent of $\theta_k$. Furthermore, the dynamics of $u$ does not depend on $\lambda_k$.
\end{lemma}
\begin{proof}By symmetry, the following sums are 0.
{\small
$$
\sum_{l, k \text{ distinct}} \frac{ u_{l}  \sqrt{u_{k} } \; \lambda_{k}d\theta_{k}(t) +  u_{k} \sqrt{u_{l} } \; \lambda_{l}d\theta_{l}(t)}{\lambda^{2}_{k} - \lambda^{2}_{l}},
\sum_{l, k, r \text{ distinct}} \frac{u_{l}(u_{r} + 2u_{k})}{(\lambda^{2}_{k} - \lambda^{2}_{l})(\lambda^{2}_{k} - \lambda^{2}_{r})},
\sum_{l, k \text{ distinct}}
\frac{\lambda^{2}_{k} u_{l}^{2} - \lambda^{2}_{l} u_{k}^{2}}{(\lambda^{2}_{k} - \lambda^{2}_{l})^2}
=
0
$$
}
Therefore, $du(t) = 2 \sum_{k=1}^{n} \sqrt{u_k(t)} d\gamma_k(t) + 2n\, dt$. Next $d\gamma = \frac{1}{\sqrt{u}} \sum_k \sqrt{u_k} d \gamma_k$ is a Brownian motion by considering its quadratic variation
$d\gamma d \gamma = \frac{\sum u_k dt}{u} = dt$. Furthermore, it is independent of $\theta_k$ as $\theta_k, \gamma_k$ are independent. Therefore, $u$ is adapted to the filtration generated by $\gamma(t)$ and $u(0)$ which does not depend on $\lambda$.
\end{proof}

Having established that the singular values exhibit repulsion, we now show that they form a determinantal point process (DPP), a canonical modelling framework for repulsive phenomena.
\section{Singular Values as a Determinantal Point Process}
\label{section:Biorthogonal Ensemble}
DPPs are among the most tractable class of interacting point processes and allow for explicit computations and exact sampling. Not surprisingly, they are a key concept in random matrix theory \cite{andersonIntroductionRandomMatrices2009, soshnikovDeterminantalRandomPoint2000, johanssonRandomMatricesDeterminantal2005} where they model the eigenvalues of classical matrix models.  
\begin{definition}[Determinantal Point Process]
\label{definition:Determinantal Point Process}
    Given a collection of \(n\) points on $\mathcal{X} \subseteq \R$ with joint density \(p\), we define its \(k\)-point correlation function for \(k \le n\) as
    \begin{equation}
    \label{equation:Correlation Function}
        \rho_{k}(x_1, \dots, x_k)
        =
        \frac{n!}{(n-k)!}
        \int_{\mathcal{X}^{n-k}} p(x_{1}, \dots, x_{n}) dx_{k+1} \dots dx_{n}
        .
    \end{equation}
    It is said to be a determinantal point process (DPP) with kernel \(K_n(x, y)\) if its \(k\)-point correlation functions can be expressed as the determinant of the Gram matrix of the kernel.
    \begin{equation}
        \rho_{k}(x_1, \dots, x_k) =
        \det [K_n(x_i, x_j)]_{i, j \in [k]}
    \end{equation}
\end{definition}
By Corollary \ref{corollary:Joint density of L\'evy's area singular values}, the density of the (rescaled) singular values $x_i$ is
\begin{equation*}
    p(x_1, \dots, x_n)
    \propto
    \det\left[ x_{j}^{2(i-1)} \right]
    \det\left[ \sech\left(x_{j}\right)^{2(i-1)} \right]
    \prod_{i=1}^{n} w(x_i)
\end{equation*}
where $w(x) = \sech(x)$ when $d=2n$ and $w(x) = x \sech(x) \tanh(x)$ when $d = 2n+1$.
The following lemma identifies $x_i$ as a biorthogonal ensemble. Moreover, provided an associated Gram matrix is invertible, $x_i$ is also a DPP with the associated kernel is given in terms of "biorthogonal functions".
\begin{lemma}[Biorthogonal ensembles]
\label{lemma:Biorthogonal ensemble}
    An ensemble of \(n\) points \(x_i\) on \(\mathcal{X} \subseteq \R\) is a biorthogonal ensemble if there exists functions \(\xi_i, \eta_i, w\) such that
    \begin{equation*}
        p(x_1, \dots, x_n)
        \propto
         \det[\xi_{i}(x_j)]_{i, j \in [n]}
        \det[\eta_{i}(x_j)]_{i, j \in [n]}
        \prod_{i=1}^{n} w(x_i)
        .
    \end{equation*}
    If the associated Gram matrix
    \begin{equation*}
        G_{ij}
        = \ip{\xi_{i}}{\eta_{j}}_w
        = \int_{\mathcal{X}} \xi_{i}(x) \eta_{j}(x) w(x) dx
    \end{equation*}
    is invertible and admits an LU-decomposition $G = LU$ for some lower and upper triangular matrices \(L, U\), then \(x_i\) is a determinantal point process\footnote{Despite the extra factor of $\prod w(x_i)$ in $\rho_k$, these are still a DPP in the literal sense as one can absorb the factor in the determinant, e.g. by scaling the $i$-th row by $w(x_i)$ and consider the kernel $w(x) K_n(x, y)$.}. The correlation function can be expressed as
    \begin{equation*}
        \rho_k(x_1, \dots, x_k) =
        \det\left[ K_n\left( x_i, x_j \right) \right]_{i, j = 1}^{k}
        \prod_{i=1}^{k} w(x_i)
    \end{equation*}
    where $K_{n}(x, y) = \sum_{m=0}^{n-1} p_{m}(x) q_{m}(y)$ and \(p_m(x), q_m(x)\) are biorthogonal functions in the sense \(\ip{p_i}{q_j}_w = \delta_{ij}\). These are specified by the inverses of the triangular matrices \(P = L^{-1}, Q = (U^{-1})^{\top}\)
    \begin{equation*}
        p_{m}(x) = \sum_{k=0}^{m} P_{m+1, k+1} \xi_{k+1},\quad
        q_{m}(x) = \sum_{k=0}^{m} Q_{m+1, k+1} \eta_{k+1}
        .
    \end{equation*}
\end{lemma}

\subsection{The Singular Values are a Determinantal Point Process}
We show the singular values satisfy the requirements of the lemma above and hence form a determinantal point process.

\begin{theorem}[Singular values are a DPP]
\label{theorem: Singular values are a DPP}
    Let $t > 0$ and $n = \lfloor d / 2 \rfloor$. The singular values of the $d \times d$ matrix of L\'evy's areas $\bA_t$ are a determinantal point process on \(\R_+\) with $k$-point correlation function
    \begin{equation}
        \rho_k(\sigma_1, \dots, \sigma_k) = \left(\frac{\pi}{t}\right)^{k} \det\left[ K_{n}\left( \frac{\pi}{t} \sigma_{i}, \frac{\pi}{t} \sigma_{j}\right) \right]_{i, j = 1}^{k},\quad k \le n
    \end{equation}
    and kernel
    \begin{equation}
        K_n(x, y) = \sum_{m=0}^{n-1} p_m(x) q_m(y)
    \end{equation}
   where $p_m, q_m$ are biorthogonal functions in the sense that \(\int_{0}^{\infty} p_i(x) q_j(x) dx = \delta_{ij}\).
\begin{align}
    \label{equation:x^2 Biorthogonal function definition}
        p_m(x)
        &=
        \mathrm{Re} \left[
\begin{cases}
\left( x + i\frac{\pi}{2}\right)^{2m}, & d = 2n\\
\left( x + i \frac{\pi}{2} \right)^{2m + 1}, & d = 2n+1
\end{cases} \right] \\
    \label{equation:sech Biorthogonal function definition}
        q_{m}(x)
&=
\frac{2}{\pi}
\begin{cases}
\frac{1}{(2m)!} \sech^{(2m)}(x), & d = 2n \\
\frac{-1}{(2m + 1)!} \sech^{(2m+1)}(x), & d = 2n + 1
.
\end{cases}
\end{align}
\end{theorem}

\begin{proof}[Proof of Theorem \ref{theorem: Singular values are a DPP}]
To simplify subsequent computations, we first consider the singular values $x_i$ of $\bA_{\pi}$, for which $x_i \eqdist \pi \sigma_i / t$. Following Lemma \ref{lemma:Biorthogonal ensemble}, this is a biorthogonal ensemble over $\mathcal{X} = \R_{+}$ with $\xi_{i}(x) = x^{2(i-1)}, \eta_{i}(x) = \sech(x)^{2(i-1)}$ and weight $w$. 
We proceed in three steps. The results follow from standard generating function arguments and is deferred to Appendix \ref{appendix:Proofs for the Determinantal Point Process}.
\begin{enumerate}[label=(\roman*)]
    \item Section \ref{section:LU decomposition of the Gram matrix} obtains the LU decomposition for the Gram matrix $G = LU$.
    \item Section \ref{section:Computing the biorthogonal functions} inverts $L$ and $U$ and computes the biorthogonal functions $p_m(x)$ and $q_m(x)$.
    \item Section \ref{section:Simplifying the Correlation Function} simplifies the correlation function and absorbs $w$ into the determinant.
\end{enumerate}
This shows that $x_i$ are DPP over $\R_{+}$ with kernel $K_n(x, y)$. Rescaling to $\sigma_i$ completes the proof.
\end{proof}

\subsection{Recurrence Relations}
The biorthogonal functions, up to a factor, are polynomials in $x^2$ and $\sech(x)^2$ respectively and hence satisfy a recurrence relation.
We identify the recurrence coefficients explicitly which will be useful for several purposes. For example evaluating the kernel and establishing bounds for the scaling limits in the next section.
The proof is also a standard computation with generating functions and deferred to Appendix \ref{appendix:Proofs for recurrence relations for the Biorthogonal functions}.

\begin{lemma}[Recurrence Relation]
\label{lemma:Recurrence}
For $d=2n$, $p_m(x)$ is a polynomial in $x^2$ and for $d=2n+1$, $p_m(x) / x$ is also a polynomial in $x^2$. The following recursion holds
\begin{equation}
x^{2} p_{m}(x) = \sum_{k=0}^{m+1} \alpha_{k, m} p_{m+1 - k}(x)
\end{equation}
where $T_{m} = 2 \left( \frac{2}{\pi} \right)^{2m+2}(2m+1)! \,(1 - 2^{-(2m+2)})\zeta(2m+2)$ are the tangent numbers (OEIS A000182), and $\zeta(s) = \sum_{n=1}^{\infty} \frac{1}{n^{s}}$ is the Riemann Zeta function and
\begin{equation}
\begin{split}
\alpha_{0, m} &= 1 \\
\alpha_{1, m} &= \left( \frac{\pi}{2} \right)^{2}
\begin{cases}
4m+1,  & d=2n \\
4m+3, & d=2n+1
\end{cases} \\
\alpha_{k, m} &=  \left( \frac{\pi}{2} \right)^{2k} 2T_{k-1}
\begin{cases}
\binom{2m+1}{2(m+1 - k)}, & d = 2n \\
\binom{2m+2}{2(m+1 - k)+1} & d = 2n+1
\end{cases},\quad 2 \le k \le m+1
.
\end{split}
\end{equation}
Similarly, for $d=2n$, $q_m(x) / \sech(x)$ is a polynomial in $\sech(x)^2$ and for $d=2n+1$, $q_m(x) / (\sech(x) \tanh(x))$ is also a polynomial in $\sech(x)^2$. The following recursion holds
$$
\sech(x)^{2} q_{m}(x) = \sum_{k=0}^{m+1} (-1)^{k-1} \beta_{k, m} q_{m+1-k}(x)
$$
where $0 \le \beta_{k, m} \le 1$ and
\begin{equation}
    \begin{split}
\beta_{0, m} &= 1 \\
\beta_{k, m} &= \frac{2(2k-1)}{\pi^{2k}} \zeta(2k),\quad 1 \le k \le m \\
\beta_{m+1, m} &= \frac{4\zeta(2m+2)}{\pi^{2m+2}}   
\begin{cases}
(2m+1) (1- 2^{-2m-2}), & d = 2n \\
m+1 -2^{-2m-2},& d = 2n+1
\end{cases}
.
\end{split}
\end{equation}
\end{lemma}

The technique can also be extended to get the coefficients of $x^{2N}p_m$. This gives an explicit formula for the expected trace moments $m_{N} = \Expectation{\sum_{i=1}^{n} x_{i}^{2N}} = \frac{1}{2} \Expectation{\tr{(\bA_{\pi}^{\top}\bA_{\pi})^{N}}}$.
\begin{lemma}[Expected Trace Moments]
\label{lemma:Moments of the empirical measure}
Let $d = 2n$.
For $N \in \mathbb{N}$, let $m_N$ denote the $N$-th expected trace moment $m_{N}
= \Expectation{\sum_{i=1}^{n} x_{i}^{2N}}
= \frac{1}{2} \Expectation{\tr{(\bA^{\top}\bA)^{N}}}
$, then
\begin{equation}
m_{N} = \left( \frac{\pi}{2} \right)^{2N} (-1)^{N-1} \left[ -n + \sum_{m=0}^{n-1} \sum_{j=0}^{N-1} \binom{2N}{2j+1} \binom{2(m+j)+1}{2m}(-1)^j T_{j}\right]
.
\end{equation}
\end{lemma}

We conclude this section by remarking standard properties of DPPs that are relevant to us.

\subsection{Sampling L\'evy's Area}\label{section:Sampling from DPP}
By Corollary \ref{corollary:Joint density of L\'evy's area singular values}, the singular values and singular vectors of $\bA_\pi = Q J(X) Q^{\top}$ are independent with $Q$ Haar-distributed. $Q$ can be generated in $\bigO{n^3}$, for example by taking the QR decomposition of a Gaussian matrix with independent entries \cite{stewartEfficientGenerationRandom1980}. The remaining challenge is to sample $x_i$.
Fortunately, $x_i$ are "finite rank projection DPPs", which admit an iterative sampling algorithm as the conditional densities \(p(x_{i+1}|x_{1:i})\) are in closed form \cite[Section 2]{hennigExactSamplingDeterminantal2018}.
\begin{equation}
    p(x_{1}, \dots, x_n)
    = 
    \prod_{k=1}^{n} p(x_k | x_{1:k-1})
    =
    \prod_{k=1}^{n}
    \frac{f_k(x_k)}{n - k +1}
\end{equation}
\begin{equation}
    f_k(x)
    =
    \begin{cases}
        K(x, x), & k = 1 \\
        K(x, x) - K(x, x_{1:k-1}) K(x_{1:k-1}, x_{1:k-1})^{-1} K(x_{1:k-1}, x),
        & 1 < k \le n
    \end{cases}
\end{equation}
where $K(x, x_{1:k-1})_i = K(x, x_i)$ and $K(x_{1:k-1}, x)_i = K(x_i, x)$ are $k-1$-dimensional vectors and \\ $K(x_{1:k-1}, x_{1:k-1})_{i, j} = K(x_i, x_j)$ is the $k-1 \times k-1$ kernel Gram matrix.
This yields an algorithm which runs in \(\bigO{n^3}\):
\begin{enumerate}[label=\roman*.]
    \item Generate an independent Haar-distributed matrix, which has cost $\bigO{n^3}$.
    \item For each conditional density, apply rejection sampling as $f_k(x)$ can be evaluated explicitly. The overall cost is $\bigO{n^3}$ as recurrence relations allow kernel evaluation in $\bigO{n}$ and the inverse of the kernel Gram matrix can be updated sequentially in $\bigO{n^2}$ at each step.
    \item Return $\bA_{\pi} = Q J(X) Q^{\top}$ (and rescale for general $t$).
    \item If we also have the squared subspace norms $u_k$, then we can generate a joint sample of Brownian motion by returning $B_{\pi} = Q \overline{u}$.
\end{enumerate}
Insights from Section \ref{section:Establishing Spectral Decay} can be used to inform our choice of proposal densities. For example $k=1$, should use a Cauchy proposal (see Figure \ref{fig:convergence of densities}) and subsequent conditionals would fare better with an Exponential proposal, as once we have observed large $x_i$ the tails of $p(x_{i+1} | x_{1:i})$ become thinner.

\subsection{Counting Statistics and Fredholm Determinants}
\label{section:Gap probabilities and Fredholm Determinants}
DPPs also provide a standard mechanism for studying singular values through counting statistics and Fredholm determinants. We briefly review this notion, see \cite[Section 11]{keating_RMT_lecture_notes_2025}.
Given a kernel $K_n : \mathcal{X} \times \mathcal{X}$ over a domain $\mathcal{X}$ satisfying $\sup_{x, y \in \mathcal{X}} |K_n(x, y)| < \infty$, for a set $B \subset \mathcal{X}$ we view $K_n: L^{2}(B) \to L^{2}(B)$ as an operator with $K_nf(x) = \int_{B} K_n(x, y) f(y)  \, dy$. Its Fredholm determinant on $B$ is given by
$$
\det[I - K_n]_{B} := \sum_{m=0}^{\infty} \frac{(-1)^{m}}{m!} \int_{B^{m}}  \det[K(x_{i}, x_{j})] dx_{1} \dots dx_{m}
.
$$
Consider the counting measure $\nu_{n} = \sum_{i=1}^{n} \delta_{x_{i}}$, then $\nu_{n}(B) = \# \{ x_{i} \in B \}$ gives the number of points contained in $B$.
For a DPP over $\mathcal{X}$ with kernel $K$ the probability generating function of $\nu_{n}(B)$ is given by the Fredholm determinant over $B$.
$$
G_{B}(t) = \det[I - tK_n]_{B},\quad
\Prob{\nu_{n}(B) = m} = \frac{(-1)^{m}}{m!} \left.\frac{d^m}{dt^m} G_{B}(t)  \right|_{t=1}
$$
For our DPP over $\mathbb{R}_{+}$, the gap probabilities $\Prob{ v_{n}(B) = 0} = G_{B}(1)$ gives the distribution function of the largest and smallest singular values.
$$
\Prob{x_{(1)} \le t} = \det[I - K_n]_{(t, \infty)},\quad
\Prob{x_{(n)} \ge t} = \det[I - K_n]_{(0, t)}
$$

With these results, we now study the behaviour of the singular values of $\bA_t$ as $d \to \infty$.
\section{Scaling Limits}
\label{section:Establishing Spectral Decay}
Classically one considers the limit of global and local statistics. 
Global statistics describe the distribution over all singular values and
local statistics study the fluctuations of singular values around a point in the spectrum.

\subsection{Global Statistics}
Figure \ref{fig:histogram_and_quantiles}, shows that the histogram over the singular values seems to be described well by a Cauchy distribution. This is also consistent with the nearly-linear trajectories of the singular values in high dimensions, see Figure \ref{fig:Dynamics of SV, High Dimension}.
\begin{figure}[h!]
    \centering
    \includegraphics[width=\linewidth]{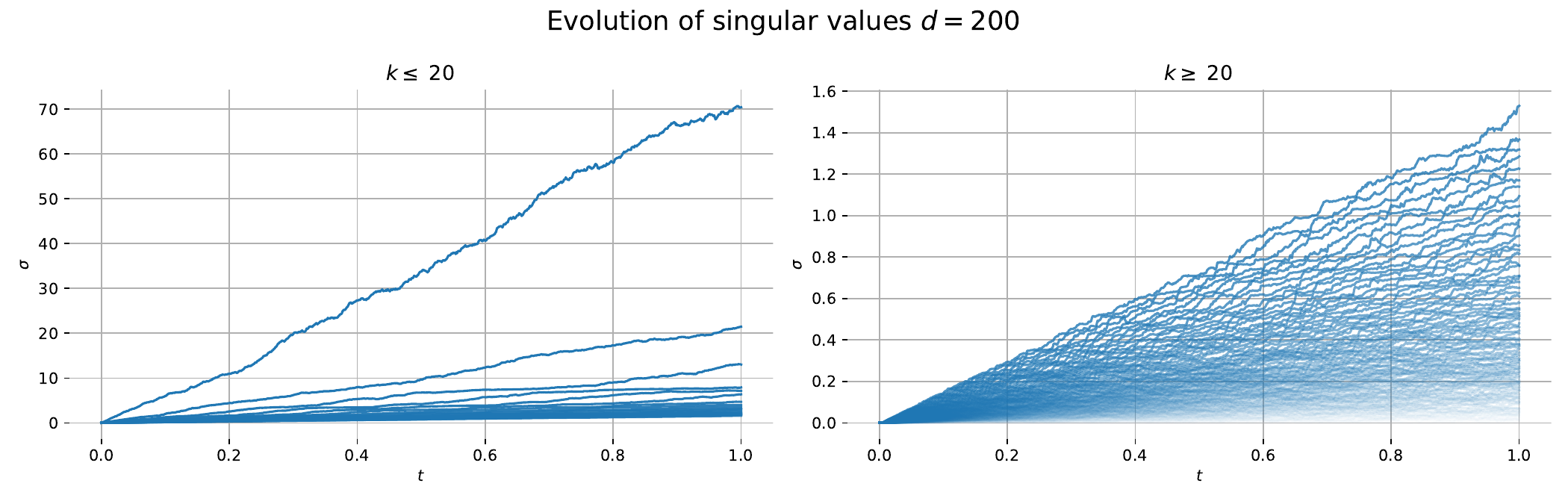}
    \caption{Evolutions of singular values over $[0, 1]$ with $d=200$. Since the largest singular values are an order of magnitude larger, we split $\sigma_{(k)}$ by $k \le 20$ and $k \ge 20$ with the trajectories shaded in progressively lighter shades of blue. The singular values preserve their ordering and all seem to exhibit linear growth.}
    \label{fig:Dynamics of SV, High Dimension}
\end{figure}
Indeed, we prove the following theorem.
\begin{theorem}
    \label{theorem:Global Law}
    The empirical measure of the singular values \(\mu_d(x)\) converges  weakly to the absolute value of a centred Cauchy distribution with scale parameter $t/2$, almost surely. 
    That is,
    \begin{equation}
    \mu_{d} := \frac{1}{n} \sum_{k=1}^{n} \delta_{\sigma_k} \dto \left|\mathrm{Cauchy}\left( 0, \frac{t}{2} \right)\right|\; a.s.
    \end{equation}
    This corresponds to the limit of the empirical measure having the density
    \begin{equation}\label{equation:limit cauchy density}
        \varrho_{t}(\sigma) = \frac{t}{\pi} \left( \sigma^{2} + \left( \frac{t}{2} \right)^{2} \right)^{-1}
    .
    \end{equation}
\end{theorem}
\begin{proof}
    We defer its proof to Section \ref{section:Proof global statistics}.
\end{proof}


\subsection{Local Statistics} 
Due to the above theorem, one expects $n\varrho_{t}(\sigma_{0})$ singular values to be located around a unit interval of $\sigma_{0}$. To get meaningful limits, we rescale such that the number of singular values in a unit interval of $\sigma_0$ is order one. Namely, $\widetilde{\sigma}_{i} = n\varrho_{t}(\sigma_{0})(\sigma_{i} - \sigma_{0})$.  The DPP structure is preserved under this rescaling, reducing the problem to studying another DPP but with a rescaled kernel.
Section \ref{section:Gap probabilities and Fredholm Determinants} shows that gap probabilities can be computed with Fredholm determinants and by \cite[Lemma 3.4.5]{andersonIntroductionRandomMatrices2009} their limits reduce to studying the limit of the rescaled kernels.
\begin{equation}
\sup_{x, y \in B} |K_{n}(x, y) - K(x, y)| \to 0 
\implies
\det[I - K_{n}]_{B} \to \det[I - K]_{B}
.
\end{equation}
Since only gap probabilities are of interest, we may conjugate the kernel, as this leaves the determinant unchanged and may be necessary to obtain a limiting kernel. To simplify the analysis, we restrict to $d=2n$ and we consider kernels corresponding to $t = \pi$ and denote $\varrho = \varrho_{\pi}$ but state probability results in terms of $\sigma_i$ by setting $x_i = \pi \sigma_i / t$.

First, for the smallest and bulk singular values, we consider the rescaled kernel
\begin{equation}\label{equation:Rescaled Kernel}
\widetilde{K}_{n, x_{0}}(u, v) =
\frac{\exp(- 2x_0 (u-v))}{n\varrho(x_{0})}
K_{n}\left(\sigma_{0} + \frac{u}{n \varrho(x_{0})}, \sigma_{0} + \frac{v}{n \varrho(x_{0})} \right)
\end{equation}
whose limit is described by the sine kernel. This is consistent with the universality of the sine kernel as the bulk scaling limit across a broad class of random matrix ensembles \cite{tao2012random}. Moreover, its Fredholm determinant is known explicitly through a solution of a Painlev\'e differential equation.

\begin{lemma}[Sine kernel and its Fredholm Determinant]
\label{lemma:Sine kernel}
See \cite[Theorem 3.1.2]{andersonIntroductionRandomMatrices2009}.
Define the sine kernel as
\begin{equation}
K_{\rm{sine}}(u, v) =
\frac{1}{\pi}
\begin{cases}
\frac{\sin(u-v)}{u-v}, & u \neq v\\
1 & u = v
.
\end{cases}
\end{equation}
Then for $x \ge 0$, its Fredholm determinant over $(-x, x)$ is
\begin{equation}\label{eq:Sine kernel Fredholm Determinant}
\det[I - K_{\mathrm{sine}}]_{(-x, x)}
= 1 - F_{\rm{sine}}(x)
= \exp\left( \int_{0}^{x} \frac{h(s)}{s}  \, ds  \right)
\end{equation}
where $h(s)$ is a solution of the $\sigma$-form of the Painlev\'e V differential equation. I.e.
$$
(s h'')^{2} + 4(sh' - h)(sh' - h + (h')^{2}) = 0
$$
such that
$$
h(s) = -\frac{s}{\pi} - \frac{s^{2}}{\pi^{2}} - \frac{s^{3}}{\pi^{3}} + O(s^{4}) \text{ as } s \downarrow 0
.
$$
Furthermore, $F_{\rm{sine}}$ is increasing with $F_{\rm{sine}}(0) = 0$ and $\lim_{ s \to \infty }F_{\rm{sine}}(s) = 1$. Therefore, $F_{\rm{sine}}$ is a CDF. 
\end{lemma}

Precisely, the smallest singular values are governed by the even sine kernel \cite{ehrhardtDysonsConstantsAsymptotics2007}, whose Fredholm determinant can be related to $F_{\rm{sine}}$.
This kernel arises when considering the limiting gap probabilities in the "hard-edge scaling limit" for the Laguerre or Jacobi random matrix ensembles. In our case, the "hard-edge" corresponds to the singular values being positive.

\begin{theorem}[Smallest Singular Values]
\label{theorem: Smallest singular values}
Let $d=2n$. Consider the rescaled kernel around $0$, then uniformly for $u, v$ in compact subsets of $\R$
\begin{equation}
\widetilde{K}_{n, 0}(u, v)
\to 
\pi K_{\rm{sine}}(\pi u, \pi v) + \pi K_{\rm{sine}}(\pi u, -\pi v)
.
\end{equation}
In particular, the distribution of the smallest singular value is given by
\begin{equation}
\lim_{n \to \infty}
\Prob{\frac{4n}{\pi t} \sigma_{(n)} \ge x} =
\sqrt{ 1 - F_{\mathrm{sine}}(\pi x) } \cdot \exp\left( -\frac{1}{2}
\int_{0}^{\pi x}  \sqrt{  -\frac{d^{2}}{ds^{2}} \log(1 - F_{\mathrm{sine}}(s)) }ds
\right)
.
\end{equation}
\end{theorem}

Around $x_0 > 0$, the limiting kernel is a rescaled sine kernel.

\begin{theorem}[Bulk Singular Values]
\label{theorem: Bulk singular values}
Let $d=2n$. Consider the kernel rescaled around $x_0 > 0$, then uniformly for $u, v$ in compact subsets of $\R$
\begin{equation}
\widetilde{K}_{n, x_{0}}(u, v)
\to
\pi K_{\rm{sine}}(\pi u, \pi v)
.
\end{equation}
Therefore, the probability all rescaled points are bounded away from $x_0$ by at least distance $x$ has limit
\begin{equation}
    \lim_{n \to \infty} \Prob{\forall k \in [n], n\varrho(x_{0}) (\pi \sigma_{k} / t - x_{0}) \notin (-x, x)} = 1 - F_{\mathrm{sine}}(\pi x)
    .
\end{equation}
\end{theorem}

\begin{proof}[Proof of Theorems \ref{theorem: Bulk singular values} and \ref{theorem: Smallest singular values}]
Lemma \ref{lemma:Convergence of the rescaled kernel} shows the kernels converge uniformly on compact subsets of $\R$. Therefore the Fredholm determinants converge and the limit gap probability is given by $\det[I - K^{\star}_{x_{0}}]_{(-x, x)}$. If $x_0 > 0$, as the limit $K^{\star}_{x_{0}}$ is a rescaled sine kernel and
$$
\det[I - K^{\star}_{x_{0}}]_{(-x, x)}
= \det[I - \pi K_{\mathrm{sine}}(\pi \cdot, \pi \cdot )]_{(-x, x)}
= \det[I - K_{\mathrm{sine}}]_{(-\pi x, \pi x)}
.
$$
If $x_0 = 0$, $n\varrho(x_0) = (\pi /2)^{-2}$ and by \cite[Equation 8]{ehrhardtDysonsConstantsAsymptotics2007} the Fredholm determinant is given by
$$
\log \det[I - K^{\star}_{0}]_{(0, x)}
=
\frac{1}{2} \log \det[I - K_{\mathrm{sine}}]_{(-\pi x, \pi x)} - \frac{1}{2} \int_{0}^{\pi x}  \sqrt{ - \frac{d^{2}}{ds^{2}} \log \det[I - K_{\mathrm{sine}}]_{(-s, s)} } ds
.
$$
Set $x_i = \pi \sigma_i / t$ to express the probability statements in terms of $\sigma_i$.
\end{proof}

Finally, we show that the largest singular values are asymptotically Gaussian with the $k$-th largest singular value centred around $\frac{2}{2k-1} \frac{t}{\pi} $. For this we consider the rescaled kernel
\begin{equation}
\widetilde{K}_{n, k}(u, v)
=
\frac{\sqrt{ 2n }}{2k-1} \exp(-\sqrt{2n}(u-v)) K_{n}\left( \frac{2n}{2k-1} + u\frac{\sqrt{ 2n }}{2k-1}, \frac{2n}{2k-1} + v\frac{\sqrt{ 2n }}{2k-1} \right)
.
\end{equation}
and show it converges uniformly to a limit independent of $v$ and equal to the Gaussian density
\begin{equation}
\phi(u) = \frac{1}{\sqrt{ 2\pi }} \exp\left( -\frac{u^{2}}{2} \right)
.
\end{equation}

\begin{theorem}[Largest singular values]
\label{theorem: Largest singular values}
Let $d=2n$.
For each $k$, $\widetilde{K}_{n, k}(u, v)$ converges uniformly over $u, v$ in compact subsets of $\R$ to a limit kernel which is independent of $v$ and equal to $\phi(u)$.
\begin{equation}
\widetilde{K}_{n, k}(u, v)
\to
\phi(u)
.
\end{equation}
In particular, the rescaled largest singular value converges to a normal distribution.
    \begin{equation}
    \sqrt{n}\left(\frac{\sigma_{(k)}}{n} - \frac{2}{2k-1} \frac{t}{\pi}\right)
    \dto
    N\left(0, \frac{1}{2} \left(\frac{2}{2k-1} \frac{t}{\pi}\right)^2\right)
    \end{equation}
\end{theorem}
\begin{proof}[Proof of Theorem \ref{theorem: Largest singular values}]
First set $t=\pi$ and rescale afterwards.
Lemma \ref{lemma:Convergence of the rescaled kernels at infinity} establishes the uniform convergence over compact subsets of $\R$ of $\widetilde{K}_{n, k}$ to $\phi(u)$ which corresponds to the rescaling
$$
x_{i}^{k, n} = (2k-1)\sqrt{ \frac{n}{2} } \left( \frac{x_{i}}{n} - \frac{2}{2k-1} \right)
.
$$
Letting $\nu_{n, k}(B) = \# \{ x_{i}^{k, n} \in B \}$ be the counting measure, this shows that for $B = (z, R)$
$$
\lim_{n \to \infty} \Prob{\nu_{n, k}(z, R) = 0} = \det[I - \phi]_{(z, R)} = \Phi(z) + 1 - \Phi(R)
.
$$
where $\Phi(u) = \int_{-\infty}^{u} \phi(u) du$ is the CDF of the standard normal distribution.

Roughly speaking, for each $k$, the rescaling sends all points larger than $x_{(k)}$ to infinity and all points smaller to $-\infty$.
Therefore, $\Prob{x_{(k)}^{k, n} \le z}$ is the same as asking for 0 points in $(z, R)$ provided $x^{k, n}_{(k)} \le R$. The latter will follow for sufficiently large $R$ by using large deviation bounds on $x^{k, n}_{(k)}$ from Lemma \ref{lemma:General LDB bound}.
Finally, this allows us to  show for each $k$, $x_{(k)}^{k, n}$ converges to a normal distribution as
$$
\lim_{n \to \infty} \Prob{x_{(k)}^{k, n} \le z} = \lim_{R \to \infty} \det[I - \phi]_{(z, R)} = \Phi(z)
.
$$

Formally, we proceed by induction on $x_{k}^{k, n}$ being normally distributed.

\textbf{Part 1: Base case} Decompose the CDF of $x_{(1)}^{1, n}$ as
$$
\Prob{x_{(1)}^{1, n} \le z} = \Prob{\nu_{n, 1}(z, \infty) = 0} = \Prob{\nu_{n, 1}(z, R) = 0} - \Prob{\nu_{n, 1}(z, R) = 0, x_{(1)}^{1, n} \ge R}
.
$$
The first term converges $\Phi(z) + 1 - \Phi(R)$ as $n \to \infty$ by the previous discussion.
Next, the second term is uniformly bounded in $n$ with a large deviation bound in $R$, Lemma \ref{lemma:General LDB bound}.
$$
\left|\Prob{\nu_{n, 1}(z, R) = 0, x_{(1)}^{1, n} \ge R}\right| \le \Prob{x_{(1)}^{1, n} \ge R} \lesssim C_{1} \exp\left( - d_{1} R \right)
$$
where $C_{1}, d_{1}$ are independent to $n$. Taking $R \to \infty$, we deduce $\lim_{ n \to \infty } \Prob{x_{(1)}^{1, n} \le z} = \Phi(z)$.

\textbf{Part 2: Induction} Now assume that $x_{(k-1)}^{k-1, n} \xrightarrow{d} N(0, 1)$ and decompose the CDF of $x_{(k)}^{k, n}$. First,
{\small
\begin{align*}
\Prob{x_{(k)}^{k, n} \le z} &= \Prob{x_{(k)}^{k, n} \le z, x_{(k-1)}^{k, n} > R} + \Prob{x_{(k)}^{k, n} \le z, x_{(k-1)}^{k, n} \le R} \\
\Prob{x_{(k)}^{k, n} \le z, x_{(k-1)}^{k, n} > R} &= \Prob{\nu_{n, k}(z, R) = 0, x_{(k-1)}^{k, n} > R} - \Prob{\nu_{n, k}(z, R) = 0, x_{(k)}^{k, n} > R,  x_{(k-1)}^{k, n} > R} \\
\Prob{\nu_{n, k}(z, R) = 0, x_{(k-1)}^{k, n} > R} &= \Prob{\nu_{n, k}(z, R) = 0} - \Prob{\nu_{n, k}(z, R) = 0, x_{(k-1)}^{k, n} \le R}
.
\end{align*}}
The first line follows from casing on $x_{(k-1)}^{k, n} > R$, the second line from casing on $x_{(k)}^{k, n} > R$ and the third line extracts the gap probability $\Prob{\nu_{n, k}(z, R) = 0}$.
Therefore,
\begin{equation*}
\begin{split}
\Prob{x_{(k)}^{k, n} \le z}
=
\Prob{\nu_{n, k}(z, R) = 0} - \Prob{\nu_{n, k}(z, R) = 0, x_{(k)}^{k, n} > R,  x_{(k-1)}^{k, n} > R} \\
+ \left\{  \Prob{x_{(k)}^{k, n} \le z, x_{(k-1)}^{k, n} \le R} - \Prob{\nu_{n, k}(z, R) = 0, x_{(k-1)}^{k, n} \le R} \right\}
.
\end{split}
\end{equation*}
Again the first term converges to $\Phi(z) + 1 - \Phi(R)$ as $n \to \infty$ and we bound the second term with a large deviation bound \ref{lemma:General LDB bound} to show it vanishes as $R \to \infty$ independent of $n$
$$
\left| \Prob{\nu_{n, k}(z, R) = 0, x_{(k)}^{k, n} > R,  x_{(k-1)}^{k, n} > R} \right| \le \Prob{x_{(k)}^{k, n} > R} \le C_k \exp(- d_k R)
$$
as $C_{k}, d_{k}$ are independent to $n$.
Now, the last two terms are bounded by
$$
\Prob{x_{(k-1)}^{k, n} \le R} = \Prob{x_{(k-1)}^{k-1, n} \le R\frac{2k-3}{2k-1} - 2 \sqrt{\frac{n}{2} }\left( 1 - \frac{2k-3}{2k-1} \right)}
.
$$
By the inductive hypothesis this must vanish in the limit. For any $M \ge 0$, the upper bound is eventually less than $-M$ and we can bound it by $\Phi(-M)$. As $M$ is arbitrary, its limit must be $0$.
$$
\lim_{ n \to \infty } \Prob{x_{(k-1)}^{k-1, n} \le \frac{R}{2k-1} - 2 \sqrt{\frac{n}{2} }\left( 1 - \frac{2k-3}{2k-1} \right)} \le \lim_{ n \to \infty } \Prob{x_{(k-1)}^{k-1, n} \le -M}  = \Phi(-M)
$$
Therefore, $\lim_{ n \to \infty } \Prob{x_{(k)}^{k, n} \le z} = \Phi(z)$ which establishes the inductive step.
\end{proof}

\subsection{Proof of Global statistics}
\label{section:Proof global statistics}
We begin by establishing the pointwise convergence of the normalised one-point correlation function $K_n(x, x) / n$ which is the density of the expected empirical measure $\Expectation{\mu_d(x)}$, see Figure \ref{fig:convergence of densities}. The key observation is $K_n(x, x)$ is a truncated Taylor expansion of $\sech$, however the expansion falls on the boundary of its radius of convergence. Instead we invoke the Hardy-Littlewood Tauberian theorem, Lemma \ref{lemma:Hardy-Littlewood Tauberian Theorem}, to determine the limit. To verify its hypothesis, we require suitable bounds on the derivatives of $\sech$ which we obtain via its pole expansion Lemma \ref{lemma:Sech Pole expansion}. By Scheff\'e's lemma, this shows convergence in expectation of the empirical measure.

\begin{lemma}[Pole Expansion of $\sech$] \label{lemma:Sech Pole expansion} We can express
\begin{equation}
    \sech(x) = i\sum_{k=0}^{\infty} (-1)^{k} \left[ \frac{1}{x+i\pi\left( k+\frac{1}{2} \right)} - \frac{1}{x-i\pi\left( k+\frac{1}{2} \right)} \right]
=
2\sum_{k=0}^{\infty} (-1)^{k-1} \, \mathrm{Im}(x^{-1}_{k})
\end{equation}
where $x_k = x + i \pi (k + \frac{1}{2})$. Furthermore,
\begin{align*}
    \sech ^{(2m)}(x) &= 2 (2m)! \sum_{k=0}^{\infty} (-1)^{k-1} \,\mathrm{Im}(x_{k}^{-2m-1}) \\
    \sech ^{(2m+1)}(x) &= 2 (2m+1)! \sum_{k=0}^{\infty} (-1)^{k} \,\mathrm{Im}(x_{k}^{-2m-2})
    .
\end{align*}
\end{lemma}
\begin{proof}
    Apply Mittag-Leffler's theorem to $\sech(x)$ and differentiate.
\end{proof}

To strengthen to almost sure convergence, we consider an invertible transformation of the singular values onto a compact domain via $z_i = \sech(x_i)^{2}$. $z_i$ is again another biorthogonal ensemble for which $q_m$ is now a polynomial in $z$. Hence $z q_m$ satisfies a recurrence whose coefficients are bounded, Lemma \ref{lemma:Recurrence}. The following lemma then implies the variance of moments of the empirical measure decays like $\bigO{1/n^2}$. By Borel-Cantelli, the moments of the empirical measure converge almost surely. Since the support is compact, the measure is uniquely determined by its moments and hence the empirical measures over $z_i$ converge weakly almost surely. The convergence of the empirical measure over $x_i$ follows due to the invertibility of the transformation.

\begin{figure}[h!]
    \centering
    \includegraphics[width=\linewidth]{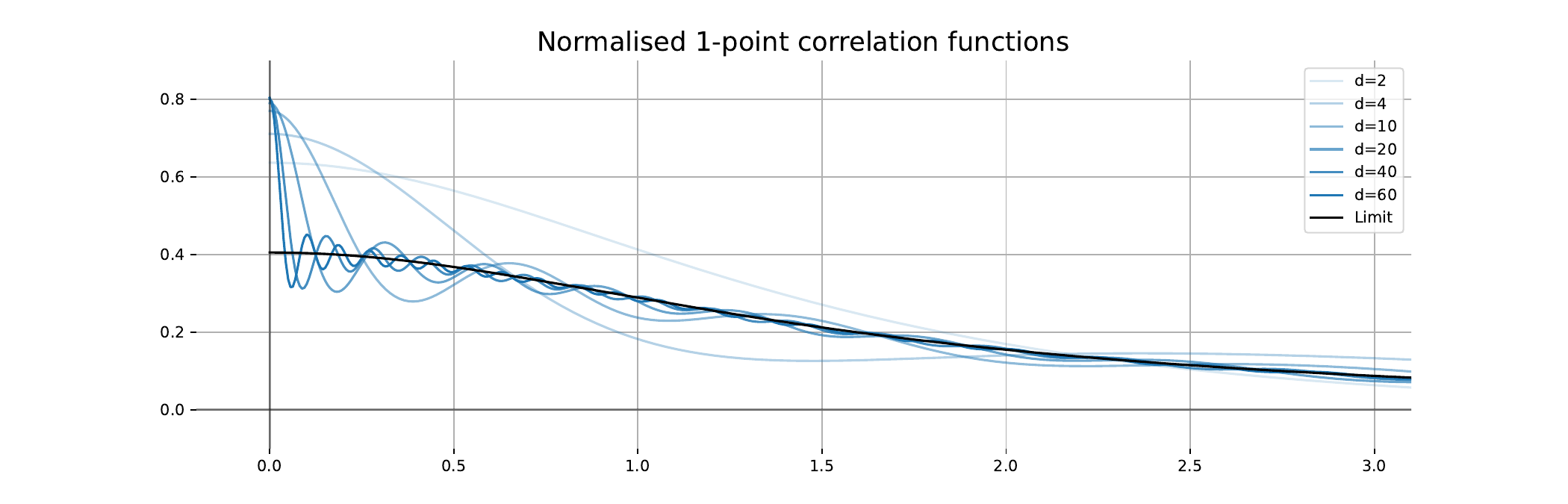}
    \caption{Normalised one-point correlation function $K_n(x, x) / n$ for various $d$. As $d$ increases, the functions converge to the Cauchy density which indicates convergence of the empirical measures over the singular values.}
    \label{fig:convergence of densities}
\end{figure}

\begin{lemma}[Variance bound for biorthogonal ensembles with recurrence]
\label{lemma:BOE Variance bound}
See \cite[Lemma 2.7]{hardyAverageCharacteristicPolynomials2015}.
Let $\{ x_{k} \}_{k \in [n]}$ be a biorthogonal ensemble on $\mathcal{X} \subset \mathbb{R}$ with biorthogonal functions $p_{m}, q_{m}$. I.e.,
$$
\rho_{k}(x_{1}, \dots, x_{k}) = \det[K_{n}(x_{i}, x_{j})]_{i, j \in [k]},\quad
K_{n}(x, y) = \sum_{m=0}^{n-1} p_{m}(x) q_{m}(y)
$$
and denote $\mu_{n}(x)$ as the empirical measure over the points $x_{i}$. Suppose that one of the biorthogonal functions satisfy the recurrence $xp_{m}(x) = \sum_{k=0}^{m+1} \beta_{k, m} \;  p_{k}(x)$, then the variance of moments of $\mu_n$ are given by
$$
\mathrm{Var}\left( \int_{\mathcal{X}} x^{l} d\mu_{n}(x)  \right)
=
\frac{1}{n^{2}} \sum_{k=0}^{n-1} \sum_{\substack{\gamma: |\gamma| = 2l+1 \\ \gamma(0) = k = \gamma(2l) \\ \gamma(l) \ge n}} w(\gamma)
$$
where $\gamma$ is a path on $\mathbb{N}$. The weight of an edge is $w(m \to k) = |\beta_{k, m}|$ and the weight of the path is
$$
w(\gamma) = \prod_{i=0}^{2l-1} w(\gamma(i) \to \gamma(i+1))
.
$$
\end{lemma}

\begin{proof}[Proof of Theorem \ref{theorem:Global Law}]
    Let $t = \pi$ and rescale afterwards to recover general $t$.
    First, we show convergence in expectation which identifies the limit as the Cauchy distribution. Afterwards we show this convergence must be almost sure by using a variance bound on $\mu_d(x)$ in Lemma \ref{lemma:BOE Variance bound}.

    \textbf{Part 1: Convergence in Expectation}
     \begin{equation*}
         \Expectation{\mu_d(x)} \to \left|\mathrm{Cauchy}\left(0, \frac{\pi}{2}\right) \right|
     \end{equation*}
     The expected empirical measure \(\Expectation{\mu_d(x)}\) is absolutely continuous and its density is given by the normalised one-point correlation function \(\varrho_n(x) := K_n(x, x) / n\). By Scheff\'e's lemma it suffices to show that \(\varrho_n(x)\) converges pointwise almost everywhere to the density of \(\left|\mathrm{Cauchy}\left(0, \frac{\pi}{2}\right) \right|\) which is $\varrho(x) = \left(x^2 + \left(\frac{\pi}{2}\right)^2\right)^{-1}$. By Theorem \ref{theorem: Singular values are a DPP}, \(\varrho_n(x)\) is a truncated Taylor series of $\sech(x + (x + i \pi / 2))$
     \begin{equation*}
         \varrho_n(x) = \frac{1}{n} \sum_{m=0}^{n-1} p_m(x) q_m(x)
         .
     \end{equation*}
    However, the radius of convergence of \(\sech(x)\) is limited by its poles at \(\pm i\frac{\pi}{2}\) which limits the validity of \(\sech(x + u)\) expanded around \(x\) to \(|u| < x^2 + (\pi / 2)^2 =: R^2\). Therefore, we consider \(\sech(x + r u)\) and take \(r \to 1\) to obtain the limit.
    Specifically, we show for \(r \in (0, 1)\)
     \begin{equation*}
        \lim_{r\to 1}(1-r) \varrho_r(x) \to \varrho(x),\quad
         \varrho_r(x) := \sum_{m = 0}^{\infty} p_m(x) q_m(x) r^m
     \end{equation*}
     By the Tauberian theorem, Lemma \ref{lemma:Hardy-Littlewood Tauberian Theorem}, this also implies \(\varrho_n(x) \to \varrho(x)\) if \(p_m(x) q_m(x)\) are bounded over \(m\) (for fixed \(x\)).
     We now proceed by casing on the parity of the dimension. Therefore, if both the even and odd subsequences converge to the same limit, convergence holds for the full sequence $\mu_d(x)$.

     \textbf{Even case} Let \(d=2n\), the density is
     \begin{equation*}
         \varrho_n(x) = \frac{1}{n\pi}
\sum_{m=0}^{n-1}
\frac{\sech^{(2m)}(x)}{(2m)!}
\left[ \left( x + i\frac{\pi}{2} \right)^{2m} + \left( x - i\frac{\pi}{2} \right)^{2m} \right]
.
     \end{equation*}

     By considering
     \begin{equation*}
        2\sum_{m \ge 0} \frac{\sech^{(2m)}(x)}{(2m)!} u^{2m}
        = \sech(x + u) + \sech(x-u)
     \end{equation*}
     and letting $u = \sqrt{r}(x \pm i \pi / 2)$ as \(|u|^2 = r |\frac{\pi}{2}i \pm x|^2 < R^2\), we deduce
     \begin{equation*}
         \begin{split}
2\pi \varrho_{r}(x)
&=
\sech\left( x+\sqrt{ r }\left( \frac{\pi}{2}i -x \right) \right)
+
\sech\left( x-\sqrt{ r }\left( \frac{\pi}{2}i - x \right) \right) \\
&+
\sech\left( x+\sqrt{ r }\left( \frac{\pi}{2}i + x \right) \right)
+
\sech\left( x-\sqrt{ r }\left( \frac{\pi}{2}i + x \right) \right) \\
&=
\sech\left( x(1 +  \sqrt{ r } )+ i \sqrt{ r }\frac{\pi}{2} \right)
+
\sech\left( x(1 + \sqrt{ r } )-i\sqrt{ r }\frac{\pi}{2} \right) \\
&+
\sech\left( x(1 - \sqrt{ r } )+i\sqrt{ r }\frac{\pi}{2} \right)
+
\sech\left( x(1 - \sqrt{ r } )-i\sqrt{ r }\frac{\pi}{2} \right)
\end{split}
     \end{equation*}
Using
\begin{equation*}
    \sech(x+iu) + \sech(x-iu) = \frac{2\cosh(x) \cos(u)}{\sinh(x)^{2} + \cos(u)^{2}}
\end{equation*}
\begin{equation*}
    \pi \varrho_{r}(x)
    =
    \frac{\cosh(x(1+\sqrt{ r }))\cos\left( \sqrt{ r } \frac{\pi}{2} \right)}{\sinh(x(1+\sqrt{ r }))^{2} + \cos\left( \sqrt{ r } \frac{\pi}{2} \right)^{2}}
+
\frac{\cosh(x(1-\sqrt{ r }))\cos\left( \sqrt{ r } \frac{\pi}{2} \right)}{\sinh(x(1-\sqrt{ r }))^{2} + \cos\left( \sqrt{ r } \frac{\pi}{2} \right)^{2}}
\end{equation*}
Provided \(x \neq 0\), the first term vanishes as \(r \to 1\) as \(\cos( \pi/ 2) = 0\).
For the second term \(\cosh(x(1-\sqrt{ r })) \to 1\), change variables to \(s = \sqrt{r}\) and consider the asymptotics \(\cos( \pi s/2) \sim \pi (1-s)/2\) and \(\sinh(x(1-s)) \sim x(1-s)\) to deduce
\begin{equation*}
\frac{\cos\left( s \frac{\pi}{2} \right)}{\sinh(x(1-s))^{2} + \cos\left( s \frac{\pi}{2} \right)^{2}}
\sim
\frac{\frac{\pi}{2}(1-s)}{x^{2}(1-s)^{2} + \left( \frac{\pi}{2} \right)^{2}(1-s)^{2}}
=
\frac{1}{(1-s)} \frac{1}{2} \left(x^2 + \left(\frac{\pi}{2}\right)^2\right)^{-1}
.
\end{equation*}
Therefore, for \(x > 0\)
$$
(1-r) \varrho_r(x) \sim \frac{(1 + \sqrt{r})}{2} \varrho(x) \to \varrho(x)
.
$$
To apply the Tauberian theorem, it suffices to show \(p_m(x) q_m(x)\) is bounded over \(m\). By considering the polar form of \(x + i\frac{\pi}{2}  = R(x) \exp(i \theta(x))\), \(R(x)^2 = (\pi/2)^2 + x^2, \theta(x) = \arctan(\pi /(2x))\), we rewrite
\begin{equation*}
    \left|
    \frac{1}{2\pi}
\frac{\sech^{(2m)}(x)}{(2m)!}
\left[ \left( x + i\frac{\pi}{2} \right)^{2m} + \left( x - i\frac{\pi}{2} \right)^{2m} \right]
    \right|
    =
    \frac{1}{\pi}
    \frac{|\sech^{(2m)}(x)|}{(2m)!} R(x)^{2m} |\cos(2m \theta(x))|
    .
\end{equation*}
Cauchy's estimate for \(\sech^{(2m)}\) will not be tight enough.
Instead, as \(\sech(x)\) is meromorphic, we consider its pole expansion with $x_{k} = x + i\pi\left( k+\frac{1}{2} \right)$ (Lemma \ref{lemma:Sech Pole expansion}).
\begin{equation*}
\sech ^{(2m)}(x) = 2 (2m)! \sum_{k=0}^{\infty} (-1)^{k-1} \,\mathrm{Im}(x_{k}^{-2m-1})
.
\end{equation*}
To form our bound, let $x_{k} = R_{k} e^{i \theta_{k}}$ with $R_{k}^{2} = x^{2} + \pi^{2}\left( k+\frac{1}{2} \right)^{2}$ and apply triangle inequality.
\begin{equation*}
    |\sech ^{(2m)}(x)| \le 2(2m)! \sum_{k \in \mathbb{N}_{0}} \frac{1}{R_{k}^{2m+1}}
\end{equation*}
To see that the terms of the series are bounded in \(m\), note $R(x) = R_{0}$ and $R_{k}$ is increasing. Hence \((R_{0}/R_{k})^{2m} < (R_{0} / R_{k})^{2} < 1\) and
\begin{equation*}
|\sech ^{(2m)}(x)| \frac{R_{x}^{2m}}{2(2m)!}
\le
\sum_{k=0}^{\infty} \frac{R_{0}^{2m}}{R_{k}^{2m+1}}
\le
R_{0}^2 \sum_{k=0}^{\infty} \frac{1}{R_{k}^{3}}
\lesssim
\sum_{k=1}^{\infty} \frac{1}{k^{3/2}}
< \infty
.
\end{equation*}
The Tauberian theorem then implies \(\lim_{n \to \infty} \varrho_n(x) \to  \lim_{r\to 1} (1-r) \varrho_r(x)\) which is equal to \(f(x)\) except at 0.
By Scheff\'e's lemma, convergence in distribution follows.

\textbf{Odd case} Let \(d=2n+1\), the density is
     \begin{equation*}
\varrho_n(x) = \frac{1}{n \pi} 
\sum_{m=0}^{n-1}
\frac{\sech^{(2m+1)}(x)}{(2m+1)!}
\left[ \left( x + i\frac{\pi}{2}\right)^{2m+1} + \left( x - i\frac{\pi}{2} \right)^{2m+1} \right]
.
     \end{equation*}
As in the even case, considering
\begin{equation*}
2\sum_{m \ge 0} \frac{\sech^{(2m+1)}(x)}{(2m+1)!} u^{2m+1}
=
\sech(x + u) - \sech(x-u)
\end{equation*}
yields
\begin{equation*}
\begin{split}
2\pi \varrho_{r}(x)
&=
\sech\left( x+\sqrt{ r }\left( \frac{\pi}{2}i -x \right) \right)
-
\sech\left( x-\sqrt{ r }\left( \frac{\pi}{2}i - x \right) \right) \\
&-
\sech\left( x+\sqrt{ r }\left( \frac{\pi}{2}i + x \right) \right)
+
\sech\left( x-\sqrt{ r }\left( \frac{\pi}{2}i + x \right) \right) \\
&=
-\sech\left( x(1 +  \sqrt{ r } )+ i \sqrt{ r }\frac{\pi}{2} \right)
-
\sech\left( x(1 + \sqrt{ r } )-i\sqrt{ r }\frac{\pi}{2} \right) \\
&+
\sech\left( x(1 - \sqrt{ r } )+i\sqrt{ r }\frac{\pi}{2} \right)
+
\sech\left( x(1 - \sqrt{ r } )-i\sqrt{ r }\frac{\pi}{2} \right)
.
\end{split}
\end{equation*}
This is the same as the even case, however the vanishing term has a negative sign which does not change the limit. To deduce the terms are bounded in $m$, a similar analysis shows
\begin{equation*}
    |\sech ^{(2m+1)}(x)| \frac{R_{x}^{2m}}{2(2m+1)!}
    \le
    \sum_{k=0}^{\infty} \frac{R_0^{2m+1}}{R_k^{2m+2}}
    \le
R_{0}^2 \sum_{k=0}^{\infty} \frac{1}{R_{k}^{3}}
\lesssim
\sum_{k=1}^{\infty} \frac{1}{k^{3/2}}
< \infty
    .
\end{equation*}

\textbf{Part 2: Almost sure convergence} 
Having established convergence in expectation, we now deduce almost sure convergence. Consider the biorthogonal ensemble $z_{i} = \sech(x_{i})^{2}$. It is also a biorthogonal ensemble. Let $f(x) = \sech(x)^{2}$ then
$$
\rho_{k}(z_{1}, \dots, z_{k}) = \det[K_{n}(f^{-1}(z_{i}), f^{-1}(z_{j}))]_{i, j = 1}^{k} \prod_{i=1}^{k} \frac{1}{2z_{i} \sqrt{ 1-z_{i} }}
$$
and $z$ is a DPP with kernel $K_{n}'(z, w)$ and denote its empirical measure with $\widetilde{\mu}_{d}$
$$
K_{n}'(z, w) = \sum_{m=0}^{n-1} q_{m}(f^{-1}(z))  \frac{p_{m}(f^{-1}(w))}{2 w \sqrt{ 1-w }}
$$
Its biorthogonal functions satisfy a recurrence relation. Namely, $zq_{m} = \sum_{k=0}^{m+1} (-1)^{k-1} \beta_{k, m} q_{k}$ by Lemma \ref{lemma:Recurrence}. Next Lemma \ref{lemma:BOE Variance bound} yields an expression for the variance of the moments of the empirical measure $\widetilde{\mu}_{d}$ over $z_{i}$. 
To bound the variance, use that the recursion coefficients are in $[0, 1]$ for both $d=2n, 2n+1$, hence the weight of any path is less than 1 and the variance is proportional to the number of paths which start and end at $\gamma(0) \le n -1$ and are greater than $n$ at step $l$.
$$
\mathrm{Var}\left( \int_{\mathcal{X}} x^{l} d\mu_{n}(x)  \right)
\le
\frac{1}{n^{2}}\;
\#\{\gamma: |\gamma| = 2l+1, \gamma(0) = \gamma(2l) \le n-1, \gamma(l) \ge n \}
.
$$
The number of such paths can be bounded independent of $n$.
As $\gamma(l) \ge n$, the path must start at $\gamma(0) \ge n-l$, giving at most $l$ start points. From a state $m$, the path can only move to $\{ 0, \dots, m+1 \}$. As the path has $2l$ steps and starts and returns to $\gamma(0)$, it cannot go higher than $\gamma(0) + 2l$ or lower than $\gamma(0) - 2l$. Therefore the path occupies at most $4l$ states.  Therefore, the number of such paths is then bounded by $(4l)^{2l}$. Finally, as there are $l$ starting points
$$
\mathrm{Var}\left( \int_{\mathcal{X}} x^{l} d\mu_{n}(x)  \right)
\le
\frac{l (4l)^{2l}}{n^{2}}
.
$$
By a Borel-Cantelli argument, this implies almost sure convergence of the moments. As the support of $\widetilde{\mu}_{d}$ is bounded by a compact set $[0, 1]$, convergence of moments imply convergence in distribution. Finally, as $\sech(x)^{2} : [0, \infty) \to (0, 1]$ is a diffeomorphism, almost sure convergence of $\widetilde{\mu}_{d}$ implies almost sure convergence of $\mu_{d}$.
\end{proof}

\subsection{Proofs for Limits of Rescaled Kernels}
\label{appendix:Proofs for the Scaling Limits}
We now show convergence of the rescaled kernels introduced earlier. This follows from using a particular representation of the kernel.
For the smallest and bulk singular values, we use the pole expansion of $\sech$ to reveal a geometric series.

\begin{lemma}[Pole Expansion Version of the Kernel]
\label{lemma:Pole expansion version of the kernel}
    Let $z = x + i \frac{\pi}{2}, y_l = y + i \pi \left( l + \frac{1}{2} \right)$. If $x \neq y$, the kernel can be rewritten as
\begin{equation}
K_{n}(x, y)
=
\frac{2}{\pi}
\sum_{l = 0}^{\infty}
(-1)^{l-1}
\mathrm{Im}\left( \frac{1}{y_{l}} \frac{1 - (z / y_{l})^{2n}}{1 - (z / y_{l})^{2}} \right)
+ \mathrm{Im}\left( \frac{1}{y_{l}} \frac{1 - (\overline{z} / y_{l})^{2n}}{1 - (\overline{z} / y_{l})^{2}} \right)
.
\end{equation}
If $x=y$, replace the first term in the $l=0$ contribution with $\mathrm{Im}\left( n / y_0 \right)$ and $\mathrm{Im}\left( nz / y^2_0 \right)$ for $d=2n, 2n+1$.
\end{lemma}
\begin{proof}
By the pole expansion of Lemma \ref{lemma:Sech Pole expansion},
\begin{align*}
p_{m}(x) &=
\begin{cases}
\mathrm{Re}(z^{2m}), & d=2n \\
\mathrm{Re}(z^{2m+1}), & d=2n+1
\end{cases} \\
q_{m}(y) &=
\frac{4}{\pi}
\sum_{l=0}^{\infty} (-1)^{l-1}
\begin{cases}
\mathrm{Im}(y_{l}^{-2m-1}), & d=2n \\
\mathrm{Im}(y_{l}^{-2m-2}), & d=2n+1
\end{cases}
\end{align*}
In the even case, the kernel can be written as
$$
K_{n}(x, y)
=
\frac{4}{\pi} 
\sum_{l = 0}^{\infty}
(-1)^{l-1}
\sum_{m=0}^{n-1}
\mathrm{Im}(y_{l}^{-2m-1})
\mathrm{Re}(z^{2m})
$$
$$
\mathrm{Im}(y_{l}^{-2m-1})
\mathrm{Re}(z^{2m})
=
\frac{1}{4i} (z^{2m} + \overline{z}^{2m})(y_{k}^{-2m-1} - \overline{y}_{k}^{-2m-1})
.
$$
Combine to reveal a geometric series $\sum_{m=0}^{n-1} a^{2m} = \frac{1 - a^{2n}}{1 - a^{2}}$ if $a \neq \pm 1$ (otherwise the sum is $n$, which occurs if $x = y$ at $l=0$).
Group into conjugate pairs to rewrite each summand as
$$
2i\left[ 
\mathrm{Im}\left( \frac{1}{y_{l}} \frac{1 - (z / y_{l})^{2n}}{1 - (z / y_{l})^{2}} \right)
+ \mathrm{Im}\left( \frac{1}{y_{l}} \frac{1 - (\overline{z} / y_{l})^{2n}}{1 - (\overline{z} / y_{l})^{2}} \right)
 \right]
$$
from which the result follows.
\end{proof}

This allows to identify the limit in terms of the sine kernel.
\begin{lemma}[Convergence of Rescaled Kernels at $x_0$]
\label{lemma:Convergence of the rescaled kernel}
Consider the kernel rescaled around $x_0$.
\begin{equation}
\widetilde{K}_{n, x_{0}}(u, v) = \frac{\exp\left(-2x_0(u-v) \right)}{n \varrho(x_{0})} K_{n}\left( x_{0} + \frac{u}{n \varrho(x_{0})}, x_{0} + \frac{v}{n \varrho(x_{0})} \right)
\end{equation}
As $n \to \infty$, the kernel converges uniformly over compact subsets. I.e., for fixed $M > 0$,
$$
\sup_{u, v \in [-M, M]} \left|
\widetilde{K}_{n, x_0}(u, v)
-
K^{\star}_{x_0}(u, v)
\right| \le C_{n,M} \xrightarrow[n \to \infty]{} 0
.
$$
If $x_0 > 0$, the limiting kernel is a rescaled sine kernel $K^{\star}_{x_0}(u, v) = \pi K_{\mathrm{sine}}(\pi u, \pi v)$ (see Lemma \ref{lemma:Sine kernel}). If $x_0 = 0$, the limit is the even part of the sine kernel $K^{\star}_{0}(u, v) = \pi K_{\rm{sine}}(\pi u, \pi v) + \pi K_{\rm{sine}}(\pi u, -\pi v)$.
\end{lemma}
\begin{proof}
We apply the representation from Lemma \ref{lemma:Pole expansion version of the kernel} and case on $x_0 > 0$ and $x_0 = 0$.
\[
K_{n}(x, y)
=
\frac{2}{\pi}
\sum_{l = 0}^{\infty}
(-1)^{l-1}
\left[
\mathrm{Im}\left( \frac{1}{y_{l}} \frac{1 - (z / y_{l})^{2n}}{1 - (z / y_{l})^{2}} \right)
+ \mathrm{Im}\left( \frac{1}{y_{l}} \frac{1 - (\overline{z} / y_{l})^{2n}}{1 - (\overline{z} / y_{l})^{2}} \right)
 \right]
\]
Let $x = x_{0} + \frac{u}{n}, y = x_{0} + \frac{v}{n}$ and let $z = x + i \frac{\pi}{2}, y_{l} = y + i\pi \left( l+\frac{1}{2} \right)$.

\textbf{Part 1: $x_0 > 0$} Observe
$$
\frac{z}{y_{l}} \to \frac{x_{0} + i\frac{\pi}{2}}{x_{0} + i \pi \left( l+\frac{1}{2} \right)},\quad \left|\frac{x_{0} + i\frac{\pi}{2}}{x_{0} + i \pi \left( l+\frac{1}{2} \right)} \right| \le 1
$$
which is strict if $l \ge 1$. Therefore, the contribution over $l \ge 1$ to the limit kernel vanishes. Consider
$$
\frac{1}{y_{l}} \frac{1 - (z / y_{l})^{2n}}{1 - (z / y_{l})^{2}} = \frac{1}{y_{l}} + \frac{z}{y_{l}^{2}} \frac{1 - \left( \frac{z}{y_{l}} \right)^{2(n-1)}}{1 - \left( \frac{z}{y_{l}} \right)^{2}}
.
$$
Despite the first term growing like $l^{-1}$, its contribution still converges to 0 uniformly by an alternating series test.
First, the terms of the series are decreasing once $2l + 1 \ge |x_{0} + \frac{v}{ n}| / \pi$.
$$
\sum_{l=1}^{\infty} \mathrm{Im}\left( \frac{1}{y_{l}} \right)
=
\sum_{l=1}^{\infty} \frac{(-1)^{l}\frac{\pi}{2}(2l+1)}{\left( x_{0} + \frac{v}{n} \right)^{2} + \pi^{2}(2l+1)^{2}}
$$
This will happen at $l^{\star}$ which is independent of $n$. As $v$ is bounded, $|x_{0} + \frac{v}{n}| \le x_{0} + \frac{|v|}{n} \le x_{0} + \frac{M}{n}$ and for $n \ge Mc$, the terms are decreasing for $l \ge l ^{\star} = \max(0, \lceil \frac{1}{2} \frac{x_{0} + c - \pi}{\pi} \rceil)$. Hence we can bound the tail of the series by the value at $l^{\star}$.
Dividing by $n$ will send the contribution to 0. The other term can be bounded by its absolute value which is order $l^{-3}$. Again as $|u|, |v| \le M$ for $n \ge M / c$
$$
\left| \frac{z}{y_{l}}\right|^{2} = \left| \frac{\overline{z}}{y_{l}}\right|^{2}
\le \frac{x_{0}^{2} + c^{2} + \left( \frac{\pi}{2} \right)^{2} + 2x_{0}c}{x_{0}^{2} + c^{2} + \left( \frac{\pi}{2} \right)^{2} -2x_{0}c + \pi^{2}(l^{2} + l)}
:= C
.
$$
For $c < \min (x_{0} ,\pi^{2} / (2x_{0}))$, $C < 1$. As $|y_{l}| \ge \pi l$ and $|z| \le |x_0 + c + i \pi /2|$
$$
\left|\frac{z}{y_{l}^{2}} \frac{1 - \left( \frac{z}{y_{l}} \right)^{2(n-1)}}{1 - \left( \frac{z}{y_{l}} \right)^{2}}\right|
\le
\frac{1}{l^{3}}
\frac{\sqrt{ (x_{0} + c)^{2} + (\pi/2)^{2} }}{\pi^{3}} \frac{1}{1-C}
$$
and dividing the associated series by $n$ will also send the contribution to 0.

Therefore, the only term to consider is $l = 0$. As $x_0 \neq 0, \frac{\overline{z}}{y_{0}} \not\to \pm 1$ and its contribution vanishes. Whereas $\frac{z}{y_{0}} \to \frac{x_{0} + i \frac{\pi}{2}}{x_{0} + i \frac{\pi}{2}} =1$ which will need care for its limit. Let $w = x_{0} + i \frac{\pi}{2}$,
\begin{align*}
\frac{z}{y_{0}} = 1 + \frac{u-v}{w} \frac{1}{n} + O(n^{-2}) 
&\implies \left( \frac{z}{y_{0}} \right)^{2} = 1 + 2\frac{u-v}{w} \frac{1}{n} + O(n^{-2}) \\
&\implies \left( \frac{z}{y_{0}} \right)^{2n} = \exp\left( 2 \frac{u-v}{w} \right) (1 + O(n^{-1})
.
\end{align*}
Therefore,
$$
\frac{1}{n} K_{n}\left( x_{0} + \frac{u}{n}, x_{0} + \frac{v}{n } \right)
\to
\frac{2}{\pi}
\frac{1}{2(u-v)} \left[
\mathrm{Im}\left(
1 -
\exp\left(\frac{2(u-v)}{|w|^{2}} \overline{w} \right)
 \right)
 \right]
 .
$$
Then use $\mathrm{Im}(\exp(a+ ib)) = \exp(a)\sin(b)$ to simplify the limit as
$$
\frac{1}{\pi(u-v)} \exp\left( \frac{2(u-v)}{|w|^{2}} x_{0} \right) \sin\left( \frac{\pi(u-v)}{|w|^{2}} \right)
.
$$
To get the limit of $\widetilde{K}_{n, x_0}$, we set $u \to u / \varrho(x_{0})$ which simplifies as $|w|^{2} = x^{2} + (\frac{\pi}{2})^{2} = \varrho(x_{0})^{-1}$. Introducing $\exp\left(-2x_0(u-v)\right)$ cancels "$\exp(a)$". If $u = v$, we get $1$ as the geometric sum is just a sum of ones. Overall the limit kernel corresponds to $\pi K_{\rm{sine}}(\pi u, \pi v)$.

\textbf{Part 2: $x_0 = 0$} We proceed as before but this time $\frac{\overline{z}}{y_{0}} \to -1$. Similarly,
\begin{align*}
\frac{\overline{z}}{y_{0}}
= \frac{\frac{u}{n} - w}{\frac{v}{n} + w} = -1 + \frac{(u+v)}{w} \frac{1}{n}+ O(n^{-2})
&\implies \left( \frac{\overline{z}}{y_{0}} \right)^{2} = 1 - \frac{2(u+v)}{w} \frac{1}{n} + O(n^{-2}) \\
&\implies \left( \frac{\overline{z}}{y_{0}} \right)^{2n} \to \exp\left( - \frac{2(u+v)}{w} \right)
.
\end{align*}
This adds an additional term of $\pi K_{\rm{sine}}(\pi u, -\pi v)$. The contribution over $l \ge 1$ still vanishes by a similar argument but we instead use the bound
$$
\left| \frac{z}{y_{l}}\right|^{2} = \left| \frac{\overline{z}}{y_{l}}\right|^{2}
\le \frac{c^{2} + \left( \frac{\pi}{2} \right)^{2} }{\left( \frac{\pi}{2} \right)^{2}+ \pi^{2}(l^{2} + l)}
:= C
$$
which is also less than $1$ for $c < \pi \sqrt{l^2 + l}$.
\end{proof}

For the largest singular values, we use a representation of the kernel in terms of the density of the Gamma distribution with unit scale and shape parameter $2n$, namely $x^{2n - 1} e^{-x} / \Gamma(2n)$. We also consider its survivor function, which is also known as the incomplete gamma function
\begin{equation}
    \widehat{\Gamma}(n, x)  = \frac{1}{\Gamma(n)}\int_{x}^{\infty} t^{n-1} e^{-t} dt
    .
\end{equation}

\begin{lemma}[Gamma Distribution Density Representation of the Kernel]
\label{lemma:Incomplete Gamma Function representation of the kernel}
Let $x_l = (2l-1)x$ then
\begin{equation}
K_{n}(x, y)
=
\frac{4}{\pi}
\sum_{l=1}^{\infty} e^{(2l-1)(x-y)} \frac{e^{-x_{l}}x_{l}^{2n-1}}{\Gamma(2n)} A_l
\end{equation}
where
\begin{equation}
A_l = 
\int_{0}^{(2l-1)\pi / 2} \left( 1 + \frac{t^{2}}{x_{l}^{2}} \right)^{\frac{2n-1}{2}} \sin\left( (2n-1)\arctan\left( \frac{t}{x_{l}} \right) \right) \sin (t) dt
.
\end{equation}
\end{lemma}
\begin{proof}
By considering the geometric series of $\frac{1}{1+x}$ which has a radius of convergence of $|x| < 1$, the expansion below converges absolutely for $x > 0$ as $|e^{-2x}| < 1$. Taking derivatives
\begin{align*}
\sech(x)
&= \frac{2e^{-x}}{e^{-2x} + 1}
= 2\sum_{l=1}^{\infty} (-1)^{l-1} e^{-(2l-1)x}  \\
\sech^{(2m)}(x) 
&=
2\sum_{l=1}^{\infty} (-1)^{l-1} (2l-1)^{2m} e^{-(2l-1)x}
.
\end{align*}
The following double series is absolutely convergent as $\sech ^{(2m)}(x)$ is absolutely convergent and the sum over $m$ is finite. By Fubini's theorem,
\begin{align*}
\sum_{m=0}^{n-1} p_{m}(x) q_{m}(y)
&= \frac{2}{\pi}\sum_{m=0}^{n-1} \frac{\sech^{(2m)}(y)}{(2m)!} \mathrm{Re}\left( \left( x + \frac{i\pi}{2} \right)^{2m} \right) \\
&= \frac{4}{\pi} \sum_{l=1}^{\infty} (-1)^{l-1} e^{-(2l-1) y} \sum_{m=0}^{n-1}  \frac{(2l-1)^{2m}}{(2m)!} \mathrm{Re}\left( \left( x + \frac{i\pi}{2} \right)^{2m} \right) \\
&= \frac{4}{\pi} \sum_{l=1}^{\infty} (-1)^{l-1} e^{-(2l-1) y} \; S_{n, l}(x)
\end{align*}
Let $z_l = (2l-1)\left( x +i\frac{\pi}{2} \right)$,
\begin{align*}
S_{n, l}(x) &= \sum_{m=0}^{n-1} \frac{(2l-1)^{2m}}{(2m)!} \mathrm{Re}\left( \left( x + \frac{i\pi}{2} \right)^{2m} \right)
= \mathrm{Re}\left( \sum_{m=0}^{n-1} \frac{z_l^{2m}}{(2m)!} \right) 
= \frac{1}{2}\mathrm{Re}\left( \sum_{m=0}^{2n-1} \frac{z_l^{m}}{m!} + \frac{(-z_l)^{m}}{m!} \right)
\end{align*}
By Taylor's remainder theorem for $e^{x}$,
$$
\sum_{m=0}^{2n-1} \frac{x^{m}}{m!}
=e^{x} \left[1 - \frac{1}{\Gamma(2n)}\int_{0}^{x} t^{2n-1} e^{-t} dt\right]
 = e^{x} \widehat{\Gamma}(2n, x)
 .
$$
As $e^{x}$ is entire, we can extend to complex $z$ by integrating along any path $\gamma$ which starts at 0 and ends at $z$. 
Next, $\exp\left( (2l-1) \frac{\pi}{2} i \right) = i^{(2l-1)} = (-1)^{l-1}\; i$, hence
\begin{align*}
\mathrm{Re}(e^{\pm z_{l}}\widehat{\Gamma}(2n, \pm z_{l})) &= \pm e^{\pm (2l-1)x} (-1)^{l} \; \mathrm{Im}\,\widehat{\Gamma}(2n, \pm z_l) \\
S_{n, l}(x) &= (-1)^{l-1}\frac{e^{(2l-1)x}}{2 \Gamma(2n)} \mathrm{Im} \left[
\int_{0}^{z_l} t^{2n-1} e^{-t} dt
-
e^{-2(2l-1)x}
\int_{0}^{-z_l} t^{2n-1} e^{-t} dt
\right]
.
\end{align*}
For the path, we choose $\gamma$ which runs along the real axis up to $\mathrm{Re}\, z_l$ and then goes vertically to $z_l$ as along the real axis the integrand is real and its imaginary part is 0. I.e.,
$$
\int_{0}^{z_l} t^{2n-1} e^{-t} dt
=
\int_{0}^{x_{l}} t^{2n-1} e^{-t} dt + i e^{-x_{l}} \int_{0}^{(2l-1) \pi / 2} (x_{l} + it)^{2n-1}e^{- it}  \, dt
$$
where $x_l = (2l-1)x$. Therefore,
\begin{align*}
S_{n, l}(x) &= \frac{(-1)^{l-1}}{\Gamma(2n)} \int_{0}^{(2l-1)\pi / 2} \mathrm{Im} [(x_{l} + it)^{2n-1}]  \, \sin (t) dt \\
&= (-1)^{l-1} e^{(2l-1) x} \frac{e^{-x_{l}}x_{l}^{2n-1}}{\Gamma(2n)} \int_{0}^{(2l-1)\pi / 2} \left( 1 + \frac{t^{2}}{x_{l}^{2}} \right)^{\frac{2n-1}{2}} \sin\left( (2n-1)\arctan\left( \frac{t}{x_{l}} \right) \right) \sin (t) dt
.
\end{align*}

\end{proof}

This allows us to identify the limit kernel as the density of the normal distribution.
\begin{lemma}[Convergence of Rescaled Kernels at Infinity]
\label{lemma:Convergence of the rescaled kernels at infinity}
Consider the rescaled kernel
$$
\widetilde{K}_{n, k}(u, v) = \frac{\sqrt{ 2n }}{2k-1} \exp(-\sqrt{2n}(u-v)) K_{n}\left(\frac{2n}{2k-1} + u\frac{\sqrt{ 2n }}{2k-1}, \frac{2n}{2k-1} + v\frac{\sqrt{ 2n }}{2k-1} \right),\quad
$$
As $n \to \infty$, the kernel converges uniformly over compact subsets. I.e., for fixed $M > 0$,
$$
\sup_{u, v \in [-M, M]} \left|
\widetilde{K}_{n, k}(u, v)
-
\frac{1}{\sqrt{ 2\pi }} \exp\left( -\frac{u^{2}}{2} \right)
\right| \le C_{n,M} \xrightarrow[n \to \infty]{} 0
$$
\end{lemma}
\begin{proof}
Let $x  = \frac{2n}{2k-1} + u\frac{\sqrt{ 2n }}{2k-1}$ and $y  = \frac{2n}{2k-1} + v\frac{\sqrt{ 2n }}{2k-1}$ and $x_l = (2l-1)x$. By Lemma \ref{lemma:Incomplete Gamma Function representation of the kernel},
$$
\widetilde{K}_{n, k}(u, v) \sim
\sum_{l=1}^{\infty} e^{\frac{2 \sqrt{ 2n }}{2k-1} (u-v)(l-k)} \frac{\sqrt{ 2n }}{2k-1}\frac{4}{\pi} \frac{e^{-x_{l}}x_{l}^{2n-1}}{\Gamma(2n)} A_l
.
$$
Let $z = 2n + u \sqrt{2n}$, then $x_l = \alpha_l z$ where $\alpha_l = (2l-1)/(2k-1)$. Let $\phi(u) = \frac{1}{\sqrt{ 2\pi }} \exp\left( -\frac{u^{2}}{2} \right)$ denote the density of the standard normal distribution. By Stirling's approximation $n! \sim \sqrt{2 \pi n} (n / e)^{n}$ and $\log(1 + x) = x - \frac{x^{2}}{2} + O(x^{3})$,
$$
\frac{e^{-x_{l}} x_{l}^{2n-1}}{\Gamma(2n)} = \frac{2n}{\alpha_{l}(2n + u \sqrt{ 2n })}  \left( \frac{e x_{l}}{2n} \right)^{2n} e^{-x_{l}}
\sim \frac{\exp(-2n I(\alpha_{l}))}{\alpha_{l} \sqrt{ 2n }} \phi(u) \exp\left(O(n^{-1/2}) \right)
.
$$
where $I(\alpha) = \alpha - 1 - \log \alpha \ge 0$ is zero at $\alpha=1$, convex and increasing for $\alpha \ge 1$.
Recall
$$
A_l = 
\int_{0}^{(2l-1)\pi / 2} \left( 1 + \frac{t^{2}}{x_{l}^{2}} \right)^{\frac{2n-1}{2}} \sin\left( (2n-1)\arctan\left( \frac{t}{x_{l}} \right) \right) \sin (t) dt
.
$$
As $u$ is bounded, the integrands converge uniformly
$$
1 \le \left( 1 + \frac{t^{2}}{x_{l}^{2}} \right)^{\frac{2n-1}{2}} \le \exp\left( \frac{2n-1}{2} \frac{t^{2}}{\alpha_{l}^{2}(2n + u \sqrt{ 2n })^{2}} \right) \to \exp(0) = 1
,
$$
$$
(2n-1)\arctan\left( \frac{t}{x_{l}} \right)  \to (2n-1) \frac{t}{\alpha_{l}(2n + u \sqrt{ 2n })}  \to \frac{t}{\alpha_{l}}
.
$$
Therefore,
$$
A_{l} \sim \int_{0}^{(2l-1) \pi /2} \sin\left( \frac{t}{\alpha_{l}} \right) \sin(t)   \, dt
\implies
A_k \sim \int_{0}^{(2k-1) \pi /2} \sin(t)^2 dt
= (2k-1) \frac{\pi}{4}
.
$$
For $l = k, \alpha_k = 1, I(\alpha_k) = 0$ and the contribution is
$$
\frac{\sqrt{ 2n }}{2k-1}\frac{4}{\pi} \frac{e^{-x_{k}}x_{k}^{2n-1}}{\Gamma(2n)} A_k
\sim \phi(u)
.
$$
For $l \neq k$, $A_l$ is bounded and the $-2n I(\alpha_{l})$ term will dominate the $\sqrt{ 2n }(u-v)(l-k)$ term as $u,v$ are bounded and only grows like $\sqrt{n}$. Therefore, the contribution over $l \neq k $ is bounded by
$$
\left|\sum_{l \neq k} e^{\frac{2 \sqrt{ 2n }}{2k-1} (u-v)(l-k)} \frac{\sqrt{ 2n }}{2k-1}\frac{4}{\pi}\frac{e^{-x_{l}}x_{l}^{2n-1}}{\Gamma(2n)} A_l\right|
\lesssim
\sum_{l < k} \exp(-2n I(\alpha_{l}))
+
\sum_{l > k} \exp(-2n I(\alpha_{l}))
.
$$
The first term is finite and the second term can be bounded by a geometric. Specifically, as the derivative $I'(\alpha)$ is positive and increasing for $\alpha > 1$ and $\alpha_{l}$ are increasing with constant spacings $\alpha_{l+1} - \alpha_l = 2 / (2k-1)$
\begin{equation*}
\begin{split}
I(\alpha_{k+r+1}) - I(\alpha_{k+r}) = \int_{0}^{2 / (2k-1)} I'(\alpha_{k+r} + \alpha) d\alpha  \ge \int_{0}^{2 / (2k-1)} I'(\alpha_{k+1} + \alpha) d\alpha =  I(\alpha_{k+2}) - I(\alpha_{k+1}) \ge \delta \\
\implies I(\alpha_{k+r+1}) - I(\alpha_{k+1}) = \sum_{l=1}^{r} I(\alpha_{k + l + 1}) - I(\alpha_{k+l}) \ge r \delta
\end{split}
\end{equation*}
Therefore, choosing $\delta = \min (I(\alpha_{k+1}), I(\alpha_{k+2}) - I(\alpha_{k+1})) / 2$, the geometric sum vanishes uniformly.
{\small
$$
\sum_{l > k} \exp(-2n I(\alpha_{l})) \le \exp(-2n I(\alpha_{k+1})) \sum_{r=0}^{\infty} \exp(-2n r \delta) =  \frac{\exp(-2n I(\alpha_{k+1}))}{1 - \exp(-2n \delta)}
\sim \exp(-2n (I(\alpha_{k+1}) - \delta))
$$
}
\end{proof}

\subsection{Large Deviation Bounds for the Largest Singular Values}
In order to prove the convergence of the largest singular values, we show they cannot deviate to be much larger than their limiting values. For the largest singular value, we adapt the argument of \cite[Lemma 3.3.2]{andersonIntroductionRandomMatrices2009} which bounds the tail probability via a moment method.
\begin{lemma}[Large Deviation Bound for the Largest Singular Values]
\label{lemma: LDB of largest SV}
For $R \ge 1$,
\begin{equation}
    \Prob{\sqrt{ \frac{n}{2} }\left( \frac{x_{(1)}}{n} - 2 \right) \ge R} \le \frac{5}{2} \exp\left( -\frac{R}{8} \right)
\end{equation}
\end{lemma}
\begin{proof}
By Markov's inequality,
$$
\Prob{x_{(1)} \ge 2n + R \sqrt{ 2n }}
\le \frac{\Expectation{x_{(1)}^{2N}}}{\left( 2n + R \sqrt{ 2n } \right)^{2N}}
\le \frac{m_{N}(n)}{\left( 2n + R \sqrt{ 2n } \right)^{2N}}
$$
Recall if $a_{j} \ge 0$ are increasing, then $|\sum_{j=0}^{N-1} (-1)^{j} a_{j}| \le a_{N-1}$. By Lemma \ref{lemma:Moments of the empirical measure},
$$
m_{N} = \left( \frac{\pi}{2} \right)^{2N} (-1)^{N-1} \left[ -n + \sum_{m=0}^{n-1} \sum_{j=0}^{N-1} \binom{2N}{2j+1} \binom{2(m+j)+1}{2m}(-1)^{j}T_{j}\right]
$$
and for fixed $m$ define
$$
a_{j} = \binom{2N}{2j+1} \binom{2(m+j)+1}{2m}(-1)^{j} \, 2 \left( \frac{2}{\pi} \right)^{2j+2} (2j+1)!\,Z_{1}(2j+2)
$$
where $(1 - 2^{-2j})\zeta(2j) = Z_1(2j)$. For fixed $k$, $Z_{k}(2j) = \sum_{m=k}^{\infty} \left( \frac{2k-1}{2m-1} \right)^{2j}$ can be bounded by the integral test for $j \ge 1$
$$
1 \le
Z_{k}(2j)
\le 1 + \int_{k}^{\infty} \frac{(2k-1)^{2j}}{(2x-1)^{2j}}  \, dx 
\le 1 + \frac{1}{2}\frac{2k-1}{2j-1}
\le \frac{2k+1}{2}
$$
For $0 \le j \le N-2$, $a_{j+1} / a_j > 1$ which implies $a_j$ is increasing
$$
\frac{a_{j+1}}{a_{j}} = \frac{(2(N-j)-1)!}{(2(N-j)-3)!} \frac{(2j+1)!}{(2j+3)!} \frac{(2(m+j)+3)!}{(2(m+j)+1)!} \left( \frac{2}{\pi} \right)^{2} \frac{Z_{1}(2j+4)}{Z_{1}(2j+2)}
\ge 6 \left( \frac{2}{\pi} \right)^{2} \frac{2}{3} = \left( \frac{4}{\pi} \right)^{2} > 1
.
$$
As the sum is alternating and $a_j$ is increasing, it is dominated by its final term
$$
\left( \frac{\pi}{2} \right)^{2N} (-1)^{N-1}a_{N - 1} = \frac{(2m+2N-1)!}{(2m)!}\, 2 Z_{1}(2N) \le (2n + 2N)^{2N-1} \, 2 Z_{1}(2N)
.
$$
Therefore, as $\left( \frac{\pi}{2} \right)^{2N} n \le (2n + 2N)^{2N}$ and $Z_1(2N) \le \frac{3}{2}$
.
$$
m_{N} \le \left( \frac{\pi}{2} \right)^{2N}(-1)^{N}n + \frac{2n\; Z_{1}(2N)}{2n + 2N} (2n + 2N)^{2N} \le \frac{5}{2} (2n + 2N)^{2N}
$$
Choosing $2N = \lfloor\frac{1}{2} R \sqrt{ 2n }\rfloor$ implies $\frac{2n + 2N}{2n + R \sqrt{ 2n }} \le 1 - \frac{2N}{2n + R \sqrt{ 2n }}$. By using $1-x \le e^{-x}$
$$
\left( \frac{2n + 2N}{2n + R \sqrt{ 2n }} \right)^{2N}
\le
\exp\left( - \frac{4N^{2}}{2n + R \sqrt{ 2n }} \right)
\le
\exp\left( -\frac{R}{8} \right)
.
$$
This yields the exponential bound on the probability
$$
\Prob{\sqrt{ \frac{n}{2} }\left( \frac{x_{(1)}}{n} - 2 \right) \ge R}\le \frac{m_{N}}{(2n + R \sqrt{ 2n })^{2N}} \le \frac{5}{2} \left( \frac{2n + 2N}{2n + R \sqrt{ 2n }} \right)^{2N} \le \frac{5}{2} \exp\left( -\frac{R}{8} \right)
$$
\end{proof}

For the $k$-th largest singular value with $k > 1$, this approach requires controlling $\Expectation{\sum_{i=k}^{n} x_{(i)}^{2N}}$ which we do not currently have in closed form. Instead we use the DPP structure to obtain a tighter bound for general $k$.

\begin{lemma}[Large Deviation Bound for the $k$-th Largest Singular Value]
\label{lemma:General LDB bound}
For $R \ge \frac{(2k-1)\log 2}{2\sqrt{ 2 }}$
\begin{equation}
\Prob{(2k-1)\sqrt{ \frac{n}{2} }\left( \frac{x_{(k)}}{n} - \frac{2}{2k-1}\right) \ge R}
\le  C_k
\exp\left( - d_k R \right)
\end{equation}
where $C_k, d_k$ are independent of $n$
$$
C_k = 3k e^{k} (k+2)^{k}k!\,2^{k} \exp\left( \frac{k(2k-1) \pi^{2}}{8}\right),\quad
d_k = \min\left( \frac{1}{\sqrt{ 2 }}, \frac{2\sqrt{ 2 }}{2k-1} \right)
.
$$
\end{lemma}
\begin{proof}
As $\Indicator{\nu_{n}(R, \infty) \ge k} \le \binom{\nu_{n}(R, \infty)}{k}$
$$
\Prob{x_{(k)} \ge R}
= \Prob{\nu_{n}(R, \infty) \ge k}
\le \Expectation{\binom{\nu_{n}(R, \infty)}{k}}
=
\int_{(R, \infty)^{k}} \det[K_{n}(x_{i}, x_{j})]_{i, j=1}^{k}
dx_{1:k}
.
$$
Hence
$$
\Prob{(2k-1)\sqrt{ \frac{n}{2} }\left( \frac{x_{(k)}}{n} - \frac{2}{2k-1}\right) \ge R}
\le
\int_{\left(\frac{1}{2k-1}(2n + R \sqrt{2n}), \infty\right)^{k}} \det[K_{n}(x_{i}, x_{j})]_{i, j=1}^{k}
dx_{1:k}
.
$$

By the Gamma function representation of the kernel, Lemma \ref{lemma:Incomplete Gamma Function representation of the kernel},
$$
K_{n}(x, y) = \frac{4}{\pi} \sum_{l=1}^{\infty} e^{-(2l-1)y} \frac{(2l-1)^{2n-1}x^{2n-1}}{\Gamma(2n)} A_{l}(x)
$$
where
$$
A_l(x) = 
\int_{0}^{(2l-1)\pi / 2} \left( 1 + \frac{t^{2}}{((2l-1)x)^{2}} \right)^{\frac{2n-1}{2}} \sin\left( (2n-1)\arctan\left( \frac{t}{(2l-1)x} \right) \right) \sin (t) dt
.
$$
Apply Leibniz formula $\det[M_{ij}]_{i,j = 1}^{k} = \sum_{\sigma \in S_{k}} \mathrm{sgn}(\sigma) \prod_{i=1}^{k} M_{\sigma(i) i}$ to show
\begin{align*}
\det[K(x_{i}, x_{j})]_{i,j =1}^{k}
&=
\left( \frac{4}{\pi} \right)^{k}\sum_{\sigma \in S_{k}} \mathrm{sgn}(\sigma) \sum_{l_{1}, \dots, l_{k}=1}^{n} \prod_{i=1}^{k} \frac{(2l_{i}-1)^{2n-1} x_{i}^{2n-1}}{\Gamma(2n)} A_{l_{i}}(x_{i}) e^{-(2l_{i}-1) x_{\sigma(i)}} \\
&=
\left( \frac{4}{\pi} \right)^{k}
\sum_{l_{1}, \dots, l_{k}=1}^{\infty}
\prod_{i=1}^{k} \frac{(2l_{i}-1)^{2n-1} x_{i}^{2n-1}}{\Gamma(2n)} A_{l_{i}}(x_{i})
\sum_{\sigma \in S_{k}} \mathrm{sgn}(\sigma)   \prod_{i=1}^{k}e^{-(2l_{i}-1) x_{\sigma(i)}}
\end{align*}
The second term is a determinant
$$
\sum_{\sigma \in S_{k}} \mathrm{sgn}(\sigma)   \prod_{i=1}^{k}e^{-(2l_{i}-1) x_{\sigma(i)}}
=
\det[e^{-(2l_{i}-1)x_{j}}]_{i, j=1}^{k}
$$
and if two $l_{i}$ are equal, two rows are equal and the determinant is 0. Therefore, we can restrict the sum to $l_{1}, \dots, l_{k}$ distinct and we can rewrite $\Expectation{\binom{\nu_{n}(R, \infty)}{k}}$ as
$$
k!
\left( \frac{4}{\pi} \right)^{k}
\sum_{\sigma \in S_{k}} \mathrm{sgn}(\sigma)
\sum_{1 \le l_{1} < \dots < l_{k} < \infty}
\int_{\left( \frac{1}{2k-1}(2n + R \sqrt{2n }), \infty \right)^{k}}  \prod_{i=1}^{k} \frac{(2l_{i}-1)^{2n-1}x_{i}^{2n-1}}{\Gamma(2n)} A_{l_{i}}\left( x_{i} \right) e^{-(2l_{i} - 1)x_{\sigma(i)}} dx_{1:k}
.
$$
Next use, $1 + x \le e^x, |\sin(x)| \le 1$ to bound
$$
\sup_{x \ge \frac{1}{2k-1}(2n + R \sqrt{ 2n })}|A_{l}(x)|
\le \frac{(2l-1) \pi}{2} \exp\left( \frac{2n-1}{2} \frac{(\pi / 2)^{2} (2k-1)^{2}}{(2n + R \sqrt{ 2n })^{2}} \right)
\le \frac{(2l-1) \pi}{2} \exp\left( \frac{\pi^{2} (2k-1)^{2}}{8} \right)
.
$$
Now we apply triangle inequality: first use the bound for $A_l(x)$, then change variables to $y_{i} = (2l_{\sigma^{-1}(i)}-1) x_{i}$ to reveal normalised incomplete gamma functions with $\alpha_l = (2l - 1) / (2k -1)$ and sum over all permutation to incur an additional $k!$ factor. This yields the upper bound
$$
(k!)^{2} 2^{k}
\exp\left( \frac{\pi^{2}\, k(2k-1)^{2}}{8} \right)
\sum_{1 \le l_{1} < \dots < l_{k} < \infty} 
\prod_{i=1}^{k} \widehat{\Gamma}(2n, \alpha_{l_{i}}(2n + R \sqrt{ 2n }))
.
$$
The sum over $l_i$ is an elementary symmetric polynomial in $\widehat{\Gamma}_{l} := \widehat{\Gamma}(2n, \alpha_{l}(2n + R \sqrt{ 2n })) \ge 0$. Provided the sums converge, which we show next, the following bound follows by properties of elementary symmetric polynomials: $e_{k}(\widehat{\Gamma}_{l}) = \sum_{i=k}^{\infty} \widehat{\Gamma}_{i} e_{k-1}(\widehat{\Gamma}_{l} \setminus \widehat{\Gamma}_{i})$ and $k! e_{k}(\widehat{\Gamma}_{l}) \le \left( \sum_{l} \widehat{\Gamma}_{l} \right)^{k}$.
$$
\sum_{1 \le l_{1} < \dots < l_{k} < \infty} 
\prod_{i=1}^{k} \widehat{\Gamma}(2n, \alpha_{l_{i}}(2n + R \sqrt{ 2n }))=
e_{k}(\widehat{\Gamma}_{l})
\le e_{k-1}(\widehat{\Gamma}_{l}) \sum_{l \ge k} \widehat{\Gamma}_{l}
\le \frac{1}{(k-1)!} \left( \sum_{l=1}^{\infty} \widehat{\Gamma}_{l} \right)^{k-1} \sum_{l \ge k} \widehat{\Gamma}_{l}
$$

Recall that $\widehat{\Gamma}(2n, x) = \Prob{X_{2n} \ge x}$ where $X_{2n}$ is a Gamma distribution with rate 1 and shape $2n$. By Markov inequality and the MGF of the Gamma distribution, for $t \in(0, 1)$
$$
\widehat{\Gamma}(2n, x) \le e^{-tx} \Expectation{e^{tX}} = e^{-tx} (1-t)^{-2n} = \exp(-tx - 2n \log(1-t))
$$
Minimise over $t$, to get $t^{\star} = 1 - \frac{2n}{x}$ which is valid if $x > 2n$. Hence for $\alpha \geq 1$, by $\log(1 + x) \le x$
$$
\widehat{\Gamma}(2n, \alpha(2n + R \sqrt{ 2n })) \le \exp\left( -2n I(\alpha) -\alpha R \sqrt{ 2n } + 2n \log\left( 1 + \frac{R}{\sqrt{ 2n }} \right) \right) \le \exp(-(\alpha-1)\sqrt{ 2 } R)
.
$$
For $\alpha=1$, the bound above is trivial and we use a tighter bound $\log(1+x) \le x - \frac{x^{2}}{2(1 + x)}$
$$
\widehat{\Gamma}(2n, 2n + R \sqrt{ 2n }) \le \exp\left( -\frac{R^{2}}{2\left( 1 + \frac{R}{\sqrt{ 2n }} \right)} \right) \le \exp\left( 1 - \frac{R}{\sqrt{ 2 }} \right)
.
$$
Overall we have, (including a redundant factor of $e > 1$ for $l \neq k$ to simplify the final bound)
$$
\widehat{\Gamma}_l
\le
e
\begin{cases}
1, & l < k \\
\exp\left( - \frac{R}{\sqrt{ 2 }} \right), &  l = k \\
\exp\left( -\frac{2(l-k)}{2k-1} R\sqrt{ 2 }\right), & l > k
.
\end{cases}
$$

For $l < k$,  $\sum_{l < k} \widehat{\Gamma}_l\le (k-1)e$. Next let $c_{k} = 2 \sqrt{ 2 } / (2k-1)$ and for $R \ge \log 2 / c_{k}$
$$
\sum_{l \ge k} \widehat{\Gamma}_l
\le e\left( \exp\left( -\frac{R}{\sqrt{ 2 }} \right) + \frac{e^{-c_{k}R}}{1- e^{-c_{k}R}} \right)
\le 3e \exp\left( - \min\left( \frac{1}{\sqrt{ 2 }}, c_{k} \right)R \right)
.
$$
In particular, this is less than $3e$ and $\sum_{l} \widehat{\Gamma}_l\le e(k+2)$.  Therefore,
$$
\Expectation{\binom{\nu_{n,k}(R, \infty)}{k}} \le  3k e^{k} (k+2)^{k}k!\,2^{k} \exp\left( \frac{k(2k-1) \pi^{2}}{8}\right)
\exp\left( - \min\left( \frac{1}{\sqrt{ 2 }}, \frac{2\sqrt{ 2 }}{2k-1} \right)R \right)
.
$$
\end{proof}

\section{Conclusion, Geometry and Future Directions}
\label{section:Conclusion}
Somewhat surprisingly, there is little literature about the singular values of L\'evy's area matrix despite it being a fundamental object of stochastic analysis for over 70 years. 
In this article, we derived explicit formulas for its distribution, obtained an interacting-particle description, established a determinantal point process representation, and analysed both global and local asymptotic regimes.
We hope the results of this paper convince the reader that the singular spectrum of L\'evy's area possesses a rich and surprisingly tractable structure.
Below we give a geometric interpretation of the singular values as L\'evy's areas of a transformed path.
Finally, we conclude with future directions.

\subsection{Geometric Interpretation}\label{sec:geometry}
Because $\bA_t$ is skew-symmetric, it is the orthogonal direct sum of planar rotations.
Formally, there exists an orthogonal decomposition 
\[
\mathbb{R}^d = \bigoplus_{i=1}^n V_i \text{ if }d\text{ is even, and }\mathbb{R}^d = \bigoplus_{i=1}^n V_i\oplus \overline{V}_{0} \text{ if }d\text{ is odd}
\]
such that:
\begin{itemize}
    \item each $V_i$ is a $2$-dimensional subspace and $\overline{V}_0$ is a one-dimensional subspace
    \item $V_i \perp V_j$ for $i \neq j$ are orthogonal,
    \item each $V_i$ is invariant under $\bA_t$, i.e.\ $\bA_t(v) \in V_i$ for $v \in V_i$,
    \item and there exists an orthonormal basis of $V_i$ in which the restriction of $\bA_t$ takes the form
    \[
    \bA_t|_{V_i} =
    \begin{pmatrix}
    0 & \sigma_i(t) \\
    -\sigma_i(t) & 0
    \end{pmatrix}, \quad \sigma_i(t) \ge 0.
    \]
\end{itemize}
Spelled out in matrix notation, this is the block diagonal form of $\bA_t = Q_t J(\Sigma_t) Q_t^{\top}$.
The above representation combined with the fact that the coordinates of $\bA_t$ are the areas of the Brownian trajectories $[0,t] \ni s\mapsto B_s$, implies that
\begin{itemize}
\item \(\hat B_s = Q_t^{\top}B_s\) for $s \in [0,t]$ is a (path dependent) rotation of the realised Brownian path,
    \item the matrix of L\'evy's areas of $\hat B_t$ is block diagonal with only $2n$ non-zero entries $\pm \sigma_k$
    \item these $\sigma_k$ are the successive extremal areas of orthogonal planar projections of $B_t$, e.g.~$\sigma_1 = \max_{u\perp v, |u|=|v|=1} u^\top \bA_t v$ corresponds to the maximal L\'evy's area between $u^{\top} B_t$ and $v^{\top} B_t$.
\end{itemize}
In Figure \ref{fig:transformed BM L\'evy's area matrix}, we visualise trajectories of $\hat B_t$ and its L\'evy's areas:
as \(\sigma_{(1)}\) is large, \(s \mapsto(\hat B_1(s), \hat B_2 (s))\) must form a coherent trajectory enclosing a significant amount of area.
In contrast, small singular values correspond to incoherent trajectories with negligible enclosed areas.
The mechanism underlying this behaviour appears to be that \(\hat B_{2k}(t) \sim \cos(\pi(2k-1)t)\) and \(\hat B_{2k-1}(t) \sim \sin(\pi(2k-1) t)\) behave like scaled and phase-shifted sinusoids.
Hence for pairs of the same frequency, \((\hat B_{2k-1}, \hat B_{2k})\) form circles which enclose area of the same sign.
In contrast, pairs with differing frequencies form Lissajous curves which yield zero signed area.
We sketch in Remark \ref{rem:lissajous} how these Lissajous figures and the asymptotics of the leading singular values emerge from optimal series expansions.
\begin{figure}
    \centering
    \includegraphics[width=\linewidth]{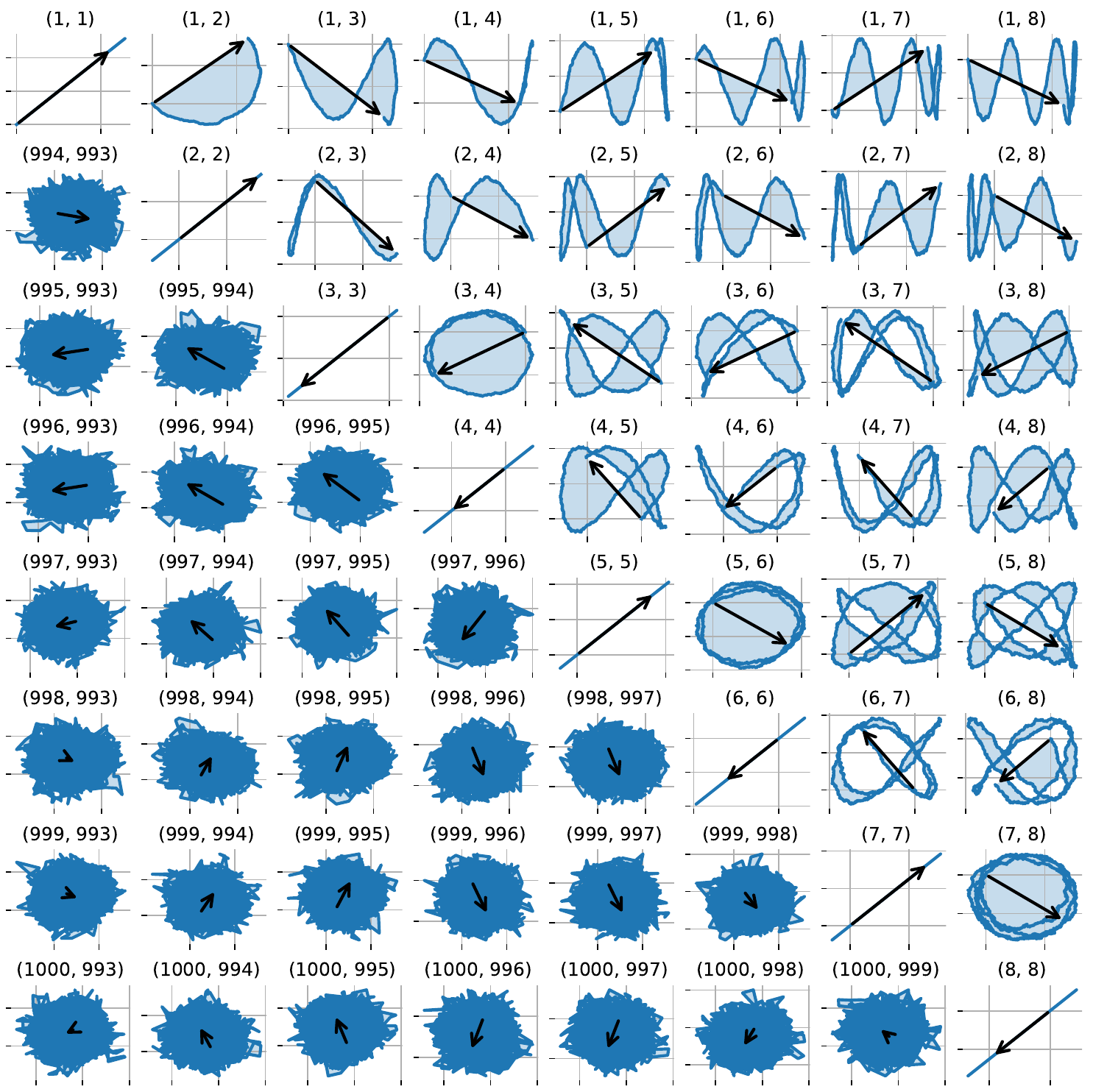}
    \caption{We visualise the singular values of the matrix of L\'evy's area \(\bA_t = Q_t J(\Sigma_t) Q_t^{\top}\) for \(d=1000\)-dimensional Brownian motion \(B_t\). The singular values correspond to the L\'evy's area of \(\hat B_t = Q_1^{\top} B_t\). Like in Figure \ref{fig:Planar Brownian Motion}, we plot the trajectory and enclosed area corresponding to pairs for the first and last eight coordinates of $\hat B_t$.}
    \label{fig:transformed BM L\'evy's area matrix}
\end{figure}

\begin{remark}[SVD vs Optimal approximation with respect to the Gaussian Hilbert Space]\label{rem:lissajous}
Recently, \cite{dickinsonDiagonalizationSimulationLevy2025a} provide a series expansion of L\'evy's area which at each truncation level $n$ is the optimal approximation with respect to functions of $4n$ elements of the Gaussian Hilbert space of \(B\).
\begin{align*}
\bA_t
&=
\sum_{k=1}^{\infty}
\frac{t}{\pi(2k-1)} [X_{k} \otimes Y_{k} - Y_{k} \otimes  X_{k}] \\
B_{t} &=
\sum_{k=1}^{\infty}
\frac{\sqrt{ 2 }}{\pi(2k-1)}
[\sin((2k-1)\pi t)X_{k} + (1- \cos((2k-1) \pi t))Y_{k}]
\end{align*}
where $X_{k}, Y_{k}$ are independent \(d\)-dimensional standard Gaussian vectors
\begin{equation*}
    X_{k} = \sqrt{ 2 } \int_{0}^{1} \cos((2k-1) \pi t) dB_{t},\quad
Y_{k} = \sqrt{ 2 } \int_{0}^{1} \sin((2k-1) \pi t) dB_{t}
.
\end{equation*}
On the other hand, the block diagonal form corresponds to the truncated singular value decomposition which yields the optimal approximation over matrices of a given rank. These are different forms of optimality, but as $d \to \infty$, they seem to agree at least in terms of the behaviour in Figure \ref{fig:transformed BM L\'evy's area matrix}.
To see this, note independent vectors are approximately orthogonal which makes the series expansion a natural candidate for the SVD with $Q_{2k-1} =\frac{1}{\sqrt{2n}} X_k, Q_{2k} = \frac{1}{\sqrt{2n}} Y_k$. As $\hat B_s = Q_t^{\top} B_s$,
\begin{equation*}
\hat B_{2k}(s) \sim \frac{2\sqrt{ n}}{\pi(2k-1)}
(1- \cos((2k-1) \pi s)),\quad
\hat B_{2k-1}(s) \sim \frac{2\sqrt{n}}{\pi(2k-1)}
\sin((2k-1)\pi s)
\end{equation*}
which agrees with the sinusoidal behaviour. Furthermore, the L\'evy's area of $\hat B_{2k}(s), \hat B_{2k-1}(s)$ is approximately $\frac{t}{\pi} \frac{2n}{2k-1}$ which agrees with the limit of $\sigma_{(k)}$ and the Gaussian fluctuations follows as it is approximately a sample mean.
\begin{equation*}
    \sqrt{ d } \left( \frac{\sigma_{(k)}}{d} - \frac{t}{\pi(2k-1)} \right)
\sim
\frac{\sqrt{ d }}{2 \pi (2k-1)}
\left( 
\frac{1}{d}
\sum [(X_{k}^{i})^{2} + (Y_{k}^{i})^{2}]
-2
 \right)
\end{equation*}
It is intriguing that these two notions of optimality appear to coincide in the limit $d \to \infty$, and understanding the mechanism behind this phenomenon is another interesting open question.
\end{remark}

\subsection{Future Directions}

\begin{description}
    \item[High Dimensions, Low Rank Approximations, and Signatures]
    The decay of the singular spectrum, Figure \ref{fig:histogram_and_quantiles}, shows that low-rank approximations to $\bA_t$ will work well, thus one can avoid the $O(d^2)$ complexity as $d \to \infty$ by smart approximations.
    In fact,  one of our original motivations for this article was to better understand the empirical success of the low-rank approximation of signature tensors $\int dB^{\otimes m}$ developed in \cite{kiraly2019kernels} and \cite{tothseq2tens}. 
    From this perspective, our results justify the Brownian case with signature level $m=2$. 
    For general processes, one cannot hope for an explicit formula but an interesting question is already the behaviour of the singular spectrum of Brownian iterated integrals for degree higher than $2$. 
    The situation here is also complicated by the case that there are several different notions of singular values resp.~eigenvalues for higher order tensors which do not share all the nice properties of the spectrum of matrices. 
    \item[Beyond Brownian Motion]
    The low-rank approximations to higher iterated integrals have shown very robust empirical performance over a wide range of different processes so potentially there are other, maybe more algebraic, reasons that allow one to hope for low-rank structure of iterated integrals in high dimensions, see Figure~\ref{fig:Singular Spectrum of GPs}.
    We have seen the singular values of $\bA_t$ for Gaussian processes remain a biorthogonal ensemble, suggesting that many of the techniques developed here may be extended.
    A natural next step is to understand how the singular values of a Gaussian process are determined by its covariance function.
    \item[Numerics] We derived an explicit formula for the density of the singular values and for a DPP with explicit kernel.
    Although numerics are beyond the scope of this work, we again draw attention to the fact that such DPPs famously come with algorithms that allow to sample exactly and in polynomial time (unlike MCMC methods where one waits for the chain to converge).     
    To make it useful for strong numerical schemes of SDEs, one would need to extend such a DPP representation to sample jointly from the squared subspace norms of Brownian motion and the singular values; see Section \ref{section:Sampling from DPP}. 
    Unfortunately, for the derivation of the joint density, the Harish-Chandra form breaks and we have an additional term in the integral over the orthogonal group, see Section \ref{section:Density of L\'evy's areas and Brownian increments}.
    We leave this as an open problem and invite readers to investigate.
\end{description}

\begin{figure}[h!]
    \centering
    \includegraphics[width=\linewidth]{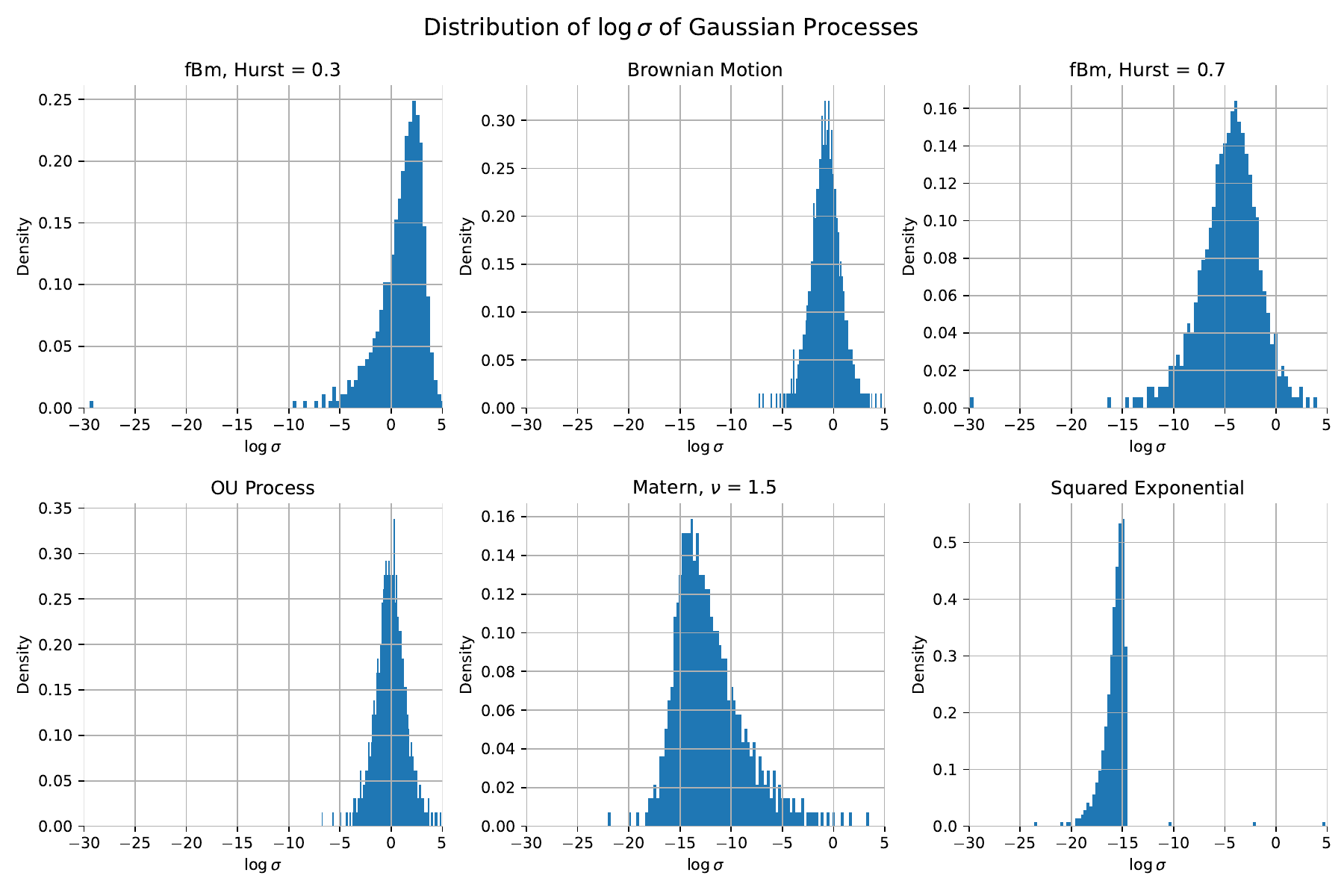}
    \caption{Histograms of the empirical measure of the log singular values for various Gaussian processes, analogous to Figure \ref{fig:histogram_and_quantiles}. The first row shows fractional Brownian motions with increasing Hurst parameter, and the second row shows Mat\'ern Gaussian processes with increasing smoothness parameter. In general, smoother paths appear to exhibit faster spectral decay.}
    \label{fig:Singular Spectrum of GPs}
\end{figure}

\textbf{Acknowledgments}
Danilo Jr Dela Cruz was supported by the EPSRC Centre for Doctoral Training in Mathematics of Random
Systems: Analysis, Modelling and Simulation (EP/S023925/1) and Harald Oberhauser was supported by the EPSRC grant "Mathematical Foundations of Intelligence: An Erlangen Programme for AI" [EP/Y028872/1].
\bibliographystyle{amsalpha}
\bibliography{references}       

\appendix
\section*{Appendix}
Appendix \ref{appendix:Auxiliary Results} contains standard results that we use throughout the article.
Appendix \ref{appendix:Proofs for the Dynamics of the Singular Values} contains the proof of the SDE dynamics of the singular values from Section \ref{sec:SDE}.
Appendix \ref{appendix:Proofs for the Determinantal Point Process} contains part of the proof needed for the DPP, Section \ref{section:Biorthogonal Ensemble}.

\section{Auxiliary Results}
\label{appendix:Auxiliary Results}

We collect several standard identities and lemmas used throughout the paper.

\subsection{Matrix Identities}\label{appendix:Matrix identities}

\begin{lemma}[Block diagonal representation of skew-symmetric matrices]
\label{lemma:Spectral Theorem for Skew-symmetric Matrices}
    For $A \in\so(d)$,  there exists an orthogonal matrix $Q \in O(d)$ and \(n = \lfloor d /2 \rfloor\) "signed singular values" \(\lambda_k \in \mathbb{R}\) such that
    \begin{equation}
        A = Q J(\Lambda) Q^{\top},\quad
        J(\Lambda) = \bigoplus_{k=1}^{n} J(\lambda_k) \oplus 0_{d-2n},\quad
        J(\lambda) = \begin{bmatrix}
            0 & \lambda \\
            -\lambda & 0
        \end{bmatrix}
        .
    \end{equation}
    If $\lambda$ are distinct, this decomposition is unique up to the sign and permutation of $\lambda$ and a rotation of the basis of each singular subspace \(\mathrm{span}(Q_{2k-1}, Q_{2k})\). The singular value decomposition (SVD) \(A = Q D_{\Sigma} \bar{Q}^{\top}\) is obtained by setting $\sigma_k = |\lambda_k|$ \((\bar{Q}_{2k-1} , \bar{Q}_{2k}) = (Q_{2k}, Q_{2k-1})\) and forming the diagonal matrix \(D_{\Sigma} = \bigoplus_{k=1}^{n} D_{\sigma_k} \oplus 0_{d-2n}, D_{\sigma} = \sigma I_2\). See \cite[Section 4.1]{ferrerCorrelationFunctionsHarishChandra2007}.
\end{lemma}

\begin{lemma}[Jacobian for block diagonalisation of skew-symmetric matrices]
\label{lemma:Jacobian for Block Diagonalisation of Skew-symmetric Matrices}
    See \cite[Section 4.1]{ferrerCorrelationFunctionsHarishChandra2007}.
    Let $A = Q J(\Lambda) Q^{\top} \in \so(d)$ and \(n = \lfloor d/2\rfloor\), then
    \begin{equation}
        dA = g(\Lambda) \prod_{k=1}^{n} d \lambda_{k} \;  dQ
    \end{equation}
    where \(d\lambda_k\) is the Lebesgue measure over \(\mathbb{R}\) and \(dQ\) is the Haar-measure over \(O(d)\) with $\int_{O(d)} dQ = 1$.
    The Jacobian is
    \begin{equation}\label{equation:Skew-symmetric block diagonalisation Jacobian}
        g(\Lambda) =
        \widetilde{\Delta}_{d}(\Lambda)^2
        \frac{C_d^{-1}}{n!}
        \cdot
        \begin{cases}
        (2\pi)^{n(n-1)},& d = 2n \\
        (2\pi)^{n^2},& d = 2n+1
        \end{cases}
        .
    \end{equation}
    where \(C_d\)
    \begin{equation} \label{equation:C_d appendix}
        C_d = \begin{cases}
        \prod_{p=1}^{n-1} (2p)!, & d = 2n \\
        \prod_{p=1}^{n-1} (2p+1)!, & d = 2n+ 1
        \end{cases}
    \end{equation}
    and \(\widetilde{\Delta}(\sigma)\) is the generalised Vandermonde determinant which depends on the parity of \(d\).
    \begin{equation}\label{equation:Generalised Vandermonde determinant}
        \widetilde{\Delta}_{d}(\Lambda)
        =
        \prod_{1 \le i<j\le n} (\lambda_{i}^{2} - \lambda_{j}^{2})
        \begin{cases}
        1, & d = 2n \\
        \prod_{i} \lambda_{i},& d = 2n+1
        \end{cases}
    \end{equation}
\end{lemma}

\subsection{Integral Identities}
\label{appendix:Integral identities}

\begin{lemma}[Harish-Chandra integral over \(O(d)\)]
\label{lemma:Harish-Chandra Integrals over the Orthogonal Group}
    See \cite[Equation 1.3, 1.4]{forrester2019orthogonal}.
    For \(A, B \in \so(d)\) with singular values \(\set{a_i}_{i \in [n]}, \set{b_i}_{i \in [n]}\) respectively, 
    \begin{equation}
        \int_{O(d)} \exp\left(\frac{1}{2}\tr(A Q B Q^{\top})\right) dQ
        =
        \frac{C_d}{\widetilde{\Delta}_d(a) \widetilde{\Delta}_d(b)}
        \begin{cases}
        \det[\cosh(a_{i} b_{j} )]_{i, j = 1}^{n}, & d = 2n \\
        \det[\sinh(a_{i} b_{j} )]_{i, j = 1}^{n}, & d = 2n+ 1
        \end{cases}
        .
    \end{equation}
    Where \(C_d\) is the same as in equation \ref{equation:C_d appendix}. Due to the invariance of the Haar measure to multiplication by orthogonal matrices, \(A, B\) are not required to be in block diagonal form. Once in block diagonal form, by analytic continuation, this extends to $A = J(X), B = J(Y)$ where $x_i, y_i$ are complex valued.
\end{lemma}

\begin{lemma}[Andr\'eief identity]
\label{lemma:Andr\'eief Identity}
    See \cite[Equation 1.7]{forresterMeetAndreiefBordeaux2018}.
    Given integrable functions \(f_i, g_i, h\)
    \begin{equation}
    \begin{split}
        \int_{\R^{n}}
        \det[f_{j}(\varphi_{i})]
        \det[g_{j}(\varphi_{i})]
        \prod_{i=1}^{n} h(\varphi_{i}) d\varphi_{i}
        = 
        n!
        \det
        \left[
        \int_{\R} f_{i}(\varphi) g_{j}(\varphi) h(\varphi) d\varphi
        \right]
    \end{split}
    .
    \end{equation}
\end{lemma}

\begin{lemma}[Hermite-Genocchi Integral over the simplex]
\label{lemma:Hermite-Genocchi Integral}
See \cite[Theorem 4.2]{baxter2011functionals}.
Let $f$ be an $n-1$-times differentiable function and $x_1, \dots, x_n$ be real numbers. Consider the $n-1$-dimensional simplex
\begin{equation*}
    \Delta_{n-1} = \set{t \in \R_{+}^n: \sum_{k=2}^{n} t_k \le 1, t_1 = 1 - \sum_{k=2}^{n} t_k}
\end{equation*}
then the following integral over the simplex is equal to the divided difference
\begin{equation}
    \int_{\Delta_{n-1}} f^{(n-1)}\left( \sum_{k=1}^{n} t_k x_k\right) dt
    =
    f[x_1, \dots, x_n]
    =
    \sum_{k=1}^{n} \frac{f(x_k)}{\prod_{l \neq k} (x_k - x_l)}
\end{equation}
\end{lemma}

\subsection{Analytic Preliminaries}



\begin{lemma}[Fourier transform of $\sech$]
\label{lemma:Fourier Transform of sech}
    For $a, b \in \mathbb{R}, a \neq 0$
$$
\int_{\mathbb{R}} e^{itx} \sech(ax) dx
=
\frac{\pi}{|a|} {\sech\left( \frac{\pi t}{2a} \right)}
.
$$
\end{lemma}
\begin{proof}
    This follows from a standard contour integral argument.
\end{proof}

\begin{lemma}[Laplace transform of powers of $\sech$]
\label{lemma:Laplace Transform of powers of sech}
$$
\varphi_{p}(t) = \int_{\mathbb{R}} e^{tx} \sech(x)^{p} \, dx
=
2^{p-1} \frac{\Gamma\left(\frac{p + t}{2}\right) \Gamma\left(\frac{p - t}{2}\right)}{\Gamma(p)}
$$
\end{lemma}
\begin{proof}
First, substitute $u = e^{2x}, du = 2u dx$ to show
$$
\varphi_{p}(t) = 2^{p-1} \int_{\mathbb{R}_{+}} u^{\frac{p+t}{2} - 1} (1+u)^{-p}\, du 
.
$$
Then $u = \frac{v}{1-v}, du  = (1-v)^{-2}dv$, to relate it to the Beta function and Gamma function.
$$
\varphi_{p}(t) = 2^{p-1} \int_{(0, 1)}  v^{\frac{p+t}{2}  - 1} (1 - v)^{\frac{p-t}{2}- 1} dv
$$
\end{proof}

\begin{lemma}[Derivatives of \(\sech x\)]
\label{lemma:Derivatives of sech}
    See \cite[Equation 20, 21]{wintucky1971formulas}
    \begin{align}
        \mathrm{sech}^{(2m)}(x)
        &=
        \sum_{n=0}^{m} (-1)^{n} W_{2n + 1, m-n} (2n)! \;\mathrm{sech}(x)^{2n+ 1} \\
        \mathrm{sech}^{(2m+1)}(x)
        &=
        \tanh(x)
        \sum_{n=0}^{m} (-1)^{n+1}W_{2n + 1, m-n} (2n + 1)! \;\mathrm{sech}(x)^{2n   + 1}
    \end{align}
where the coefficients $W$ are defined recursively
\begin{equation}
\begin{split}
        W_{2n+1, 0} = 1,\quad
        W_{2n+1, 1} = \sum_{m=0}^{n} (2m+1)^{2},\quad \\
        W_{2n+1, k} = (2n+1)^{2} W_{2n+1, k-1} + W_{2n-1, k}
        .
\end{split}
\end{equation}
\end{lemma}

\begin{lemma}[Hardy-Littlewood Tauberian Theorem]
\label{lemma:Hardy-Littlewood Tauberian Theorem}
    See \cite[pg. 184]{duren2012invitation}.
    Given a sequence of \(a_n \ge -C\) bounded below, define over \(\rho \in (0, 1)\)
    \begin{equation}
        A(r) = \sum_{m=0}^{\infty} a_m r^m  
    \end{equation}
    If \(\lim_{r \to 1-} (1-r)A(r)\) exists, then the following limit also exists and is equal to it.
    \begin{equation}
        \lim_{n \to \infty} \frac{1}{n} \sum_{m=0}^{n-1} a_m = \lim_{r \to 1-} (1-r)A(r)
    \end{equation}
\end{lemma}

\section{Proofs for the Dynamics of the Singular Values}
\label{appendix:Proofs for the Dynamics of the Singular Values}

This appendix develops the SDE governing the signed singular values and associated subspace norms. We begin by establishing perturbation formulas for the signed singular vectors and signed singular values of a skew-symmetric matrix (Appendix \ref{appendix:Perturbation analysis for Skew-symmetric matrices}). These formulas are used to obtain the SDE satisfied by the signed singular values $\lambda_k$ and squared subspace norms $u_k$ via It\^o's formula (Appendix \ref{appendix:Derivation of the Singular Value SDE}). We then prove that this SDE is well-posed when started in the interior of the state space (Appendix \ref{appendix:Well-posedness of the SDE}) and derive the corresponding Fokker-Planck equation (Appendix \ref{appendix:Derivation of the PDE}).

\subsection{Perturbation Analysis for Skew-symmetric matrices}
\label{appendix:Perturbation analysis for Skew-symmetric matrices}

\begin{lemma}[Derivatives of singular vector and values of skew-symmetric matrices]\label{lemma:derivatives of singular vectors and values}
Let \(A = Q J(\Lambda) Q^{\top}\) be a skew-symmetric with distinct singular values \(\lambda_k\) and singular vectors \(q_k\). For \(\alpha = (\alpha_1, \alpha_2)\), denote \(\partial_{\alpha} \) as the partial derivative with respect to \(A_{\alpha}\) then
\begin{equation} \label{equation: sval derivatives}
\begin{split}
        \partial_{\alpha} \lambda_{k} &= [J_{k}]_{\alpha} \\
    \partial_{\alpha} \partial_{\beta} \lambda_{k}
&=
\sum_{a} [S_{a, 2k}]_{\alpha} [J_{2k-1, a}]_{\beta} - [S_{a, 2k-1}]_{\alpha} [J_{2k, a}]_{\beta}
\end{split}
\end{equation}
and
\begin{equation} \label{eq: svector derivatives}
    \begin{split} 
    \partial_{\alpha} q_{a}
    &= \sum_{b} q_b [S_{b, a}]_{\alpha} \\
    \partial_{\alpha} \partial_{\beta} q_{a}
    &= \sum_{b} q_b [T_{b, a}]_{\alpha, \beta}
\end{split}
\end{equation}
where \(E_{a, b} := q_a q_b^{\top}\), \(J_{a, b} := E_{a, b} - E_{b, a}\), \(S_{a, b}:= q_{a}^{\top} \partial q_b\), \([T_{a, b}]_{\alpha, \beta} := q_{a}^{\top} \partial_{\alpha} \partial_{\beta} q_b\). As shorthand, \(J_k := J_{2k-1, 2k}, S_{k} := S_{2k-1, 2k}\).
Furthermore, \(J_{a, b}, S_{a, b}\) are skew as matrices \(J_{a, b}^{\top} = - J_{a, b}, S_{a, b}^{\top} = -S_{a, b}\) and skew in their index \(J_{a, b} = -J_{b, a}, S_{a, b} = - S_{b,a}\). In particular, \(S_{a, a} = J_{a, a} = 0\).
Similarly, \([T_{a, b}]_{\alpha, \beta} = [T_{a, b}]_{\beta, \alpha}\) and
\begin{equation} \label{eq: Taa expansion}
    [T_{a, a}]_{\alpha, \beta}
    =
    \sum_{c} [S_{a, c}]_{\alpha} [S_{c, a}]_{\beta}
    .
\end{equation}

For \(l \neq k\), define \(D_{kl} = \lambda_k^2 - \lambda_l^2\)
\begin{equation} \label{eq: S_a,b identity}
\begin{split}
    D_{kl} \; S_{2l-1, 2k-1} &= \lambda_{k} J_{2l-1, 2k} - \lambda_{l}J_{2l, 2k-1} \\
D_{kl} \; S_{2l, 2k-1} &= \lambda_{k}J_{2l, 2k} + \lambda_{l}J_{2l-1, 2k-1} \\
D_{kl} \; S_{2l-1, 2k} &= -\lambda_{l}J_{2l,2k} - \lambda_{k}J_{2l-1, 2k-1} \\
D_{kl} \; S_{2l, 2k} &= \lambda_{l} J_{2l-1,2k} - \lambda_{k}J_{2l, 2k-1}
\end{split}
\end{equation}
Multiplication of \(J\) yields
\begin{equation}\label{eq: J multiplication}
    J_{a b}^{\top} J_{c d} = J_{ab} J_{c, d}^{\top} =
    \delta_{ac} E_{bd} + \delta_{bd} E_{ac}
    - \delta_{ad} E_{bc} - \delta_{bc} E_{ad}
    .
\end{equation}
This shows
\begin{equation} \label{eq: JS multiplication}
\begin{split}
D_{kl} J_{2k-1, 2l-1}^{\top} S_{2l-1, 2k-1} &= -\lambda_{k} E_{2k-1, 2k} + \lambda_{l} E_{2l-1, 2l} \\
D_{kl} J_{2k-1, 2l}^{\top}S_{2l, 2k-1} &= -\lambda_{k} E_{2k-1, 2k} -\lambda_{l} E_{2l, 2l-1} \\
D_{kl} J_{2k, 2l-1}^{\top} S_{2l-1, 2k} &= \lambda_{l} E_{2l-1, 2l} + \lambda_{k} E_{2k , 2k-1} \\
D_{kl} J_{2k, 2l}^{\top} S_{2l, 2k} &= -\lambda_{l} E_{2l, 2l-1} + \lambda_{k} E_{2k, 2k-1}
,
\end{split}
\end{equation}
\begin{equation} \label{eq: JkS multiplication}
\begin{split}
D_{kl} J_{k}^{\top} S_{2l-1, 2k-1} = \lambda_{k} E_{2k-1, 2l-1} + \lambda_{l} E_{2k, 2l}\\
D_{kl} J_{k}^{\top} S_{2l, 2k-1} = \lambda_{k}E_{2k-1, 2l} - \lambda_{l}E_{2k, 2l-1}\\
D_{kl} J_{k}^{\top} S_{2l-1, 2k} = -\lambda_{l} E_{2k-1, 2l} + \lambda_{k} E_{2k, 2l-1}\\
D_{kl} J_{k}^{\top} S_{2l, 2k} = \lambda_{l} E_{2k-1, 2l-1} + \lambda_{k} E_{2k, 2l}
\end{split}
\end{equation}
and
\begin{equation} \label{eq: S_a,b squared}
\begin{split}
    \begin{split}
D_{kl}^{2} S_{2l-1, 2k-1}^{2} &= -\lambda_{k}^{2}(E_{2l-1, 2l-1} + E_{2k, 2k}) - \lambda_{l}^{2}(E_{2l, 2l} + E_{2k-1, 2k-1}) \\
 D_{kl}^{2} S_{2l, 2k-1}^{2} &= -\lambda_{k}^{2}(E_{2l, 2l} + E_{2k, 2k}) - \lambda_{l}^{2}(E_{2l-1, 2l-1} + E_{2k-1, 2k-1}) \\
D_{kl}^{2} S_{2l-1, 2k}^{2} &= -\lambda_{l}^{2}(E_{2l, 2l} + E_{2k, 2k}) - \lambda_{k}^{2}(E_{2l-1, 2l-1} + E_{2k-1, 2k-1})\\
D_{kl}^{2} S_{2l, 2k}^{2} &= -\lambda_{l}^{2}(E_{2l-1, 2l-1} + E_{2k, 2k}) - \lambda_{k}^{2}(E_{2l, 2l} + E_{2k-1, 2k-1})
\end{split}
\end{split}
\end{equation}

Define \([U_{a, b}]_{\alpha, \beta} := q_{a}^{\top} \partial_{\alpha} A\; \partial_{\beta} q_{b}\). By expanding, \([U_{a, b}]_{\alpha, \beta}
    = \sum_{c} [J_{a, c}]_{\alpha} [S_{c, b}]_{\beta}\).
These can be used to obtain equations for $T$. $T_{k,k}$ is already given by \eqref{eq: Taa expansion}. For $T_{2k-1, 2k}$,
\begin{equation}\label{eq: Symmetrised Tk expansion}
2 \lambda_{k} ([T_{2k-1, 2k}]_{\alpha, \beta} + [T_{2k, 2k-1}]_{\alpha, \beta})
=
[U_{2k, 2k}]_{\beta, \alpha}
+ [U_{2k, 2k}]_{\alpha, \beta}
- [U_{2k-1, 2k-1}]_{\beta, \alpha} 
- [U_{2k-1, 2k-1}]_{\alpha, \beta}
.
\end{equation}
For \(l \neq k\),
\begin{equation}\label{eq: T_a,b identity}
\begin{split}
    \lambda_{k} [T_{2l-1, 2k}]_{\alpha, \beta}
+ \lambda_{l} [T_{2l, 2k-1}]_{\alpha, \beta}
&=
- [J_{k}]_{\beta} [S_{2l-1, 2k}]_{\alpha}
- [J_{k}]_{\alpha} [S_{2l-1, 2k}]_{\beta} 
- [U_{2l-1, 2k-1}]_{\beta, \alpha} 
- [U_{2l-1, 2k-1}]_{\alpha, \beta}\\
- \lambda_{l} [T_{2l-1, 2k}]_{\alpha, \beta}
- \lambda_{k}[T_{2l, 2k-1}]_{\alpha, \beta}
&=
[J_{k}]_{\beta}[S_{2l, 2k-1}]_{\alpha}
+ [J_{k}]_{\alpha} [S_{2l, 2k-1}]_{\beta}
- [U_{2l, 2k}]_{\beta, \alpha} 
- [U_{2l, 2k}]_{\alpha, \beta} \\ \\
\lambda_{k} [T_{2l, 2k}]_{\alpha, \beta}
- \lambda_{l} [T_{2l-1, 2k-1}]_{\alpha, \beta}
&=
- [J_{k}]_{\beta} [S_{2l, 2k}]_{\alpha}
- [J_{k}]_{\alpha} [S_{2l, 2k}]_{\beta}
- [U_{2l, 2k-1}]_{\beta, \alpha} 
- [U_{2l, 2k-1}]_{\alpha, \beta}\\
\lambda_{l} [T_{2l, 2k}]_{\alpha, \beta}
- \lambda_{k} [T_{2l-1, 2k-1}]_{\alpha, \beta}
&=
[J_{k}]_{\beta} [S_{2l-1, 2k-1}]_{\alpha}
+ [J_{k}]_{\alpha} [S_{2l-1, 2k-1}]_{\beta}
- [U_{2l-1, 2k}]_{\beta, \alpha} 
- [U_{2l-1, 2k}]_{\alpha, \beta}
.
\end{split}
\raisetag{50pt}
\end{equation}
These linear system can be inverted via
\begin{equation}\label{eq: ST system matrix}
    \begin{bmatrix}
\lambda_{k}  & \lambda_{l} \\
-\lambda_{l} & -\lambda_{k}
\end{bmatrix}^{-1}
=
\frac{1}{\lambda_{k}^{2} - \lambda_{l}^{2}}
\begin{bmatrix}
\lambda_{k}  & \lambda_{l} \\
-\lambda_{l}  & -\lambda_{k}
\end{bmatrix}
,\quad
\begin{bmatrix}
\lambda_{k}  & \lambda_{l} \\
-\lambda_{l} & -\lambda_{k}
\end{bmatrix}^{-\top}
=
\frac{1}{\lambda_{k}^{2} - \lambda_{l}^{2}}
\begin{bmatrix}
\lambda_{k}  & -\lambda_{l} \\
\lambda_{l}  & -\lambda_{k}
\end{bmatrix}
.
\end{equation}
\end{lemma}
\begin{proof}
Skewness of $J_{a, b}$ as a matrix and its index follows from its definition. By the chain rule \(\partial_{ij} = - \partial_{ji}\) as \(A\) is skew-symmetric which implies \(S_{a, b}^{\top} = -S_{a, b}\). \([T_{a, b}]_{\alpha, \beta} = [T_{a, b}]_{\beta, \alpha}\) follows as partial derivatives commute \(\partial_{\alpha} \partial_{\beta} = \partial_{\beta} \partial_{\alpha}\).

\eqref{eq: svector derivatives} follows as $q$ is an orthogonal basis and for a vector $v = \sum_{a} q_a \; (q_{a}^{\top}v)$. To obtain expressions for the coefficients $S, T$, consider the orthogonality and singular value equations.

\textbf{Orthogonality} As \(q\) are an orthogonal basis $\delta_{a, b} = q_{a}^{\top} q_{b}$. By taking derivatives,
\begin{align*}
0 &= \partial_{\beta} q_{a}^{\top}q_{b} + q_{a}^{\top} \partial_{\beta} q_{b} \\
0 &= \partial_{\alpha} \partial_{\beta} q_{a}^{\top} q_{b} + \partial_{\beta}q_{a}^{\top} \partial_{\alpha}q_{b} + \partial_{\alpha} q_{a}^{\top} \partial_{\beta}q_{b} + q_{a}^{\top} \partial_{\alpha} \partial_{\beta}q_{b}
.
\end{align*}
In our notation,
\begin{align*}
0 &= [S_{b, a}]_{\beta} + [S_{a, b}]_{\beta} \\
0 &= [T_{b, a} + T_{a, b}]_{\alpha, \beta} + \partial_{\beta}q_{a}^{\top} \partial_{\alpha}q_{b} + \partial_{\alpha} q_{a}^{\top} \partial_{\beta}q_{b}
.
\end{align*}
The first equation establishes $S_{a, b} = - S_{b, a}$. For \eqref{eq: Taa expansion}, develop the second equation by expanding \(\partial_{\alpha} q\) and taking \(a = b\).
\begin{align*}
\partial_{\beta}q_{a}^{\top} \partial_{\alpha}q_{b}
&= \sum_{c} [S_{c, a}]_{\beta}  \; q_{c}^{\top} \sum_{d} [S_{d, b}]_{\alpha}  \; q_{d}
= \sum_{c} [S_{c, b}]_{\alpha} [S_{c, a}]_{\beta} \\
\partial_{\alpha}q_{b}^{\top} \partial_{\beta}q_{b}
&= \sum_{c} [S_{c, a}]_{\alpha} [S_{c, b}]_{\beta}
\end{align*}

\textbf{Singular equations}
As we are differentiating with respect to $A_{\alpha}$,
$
\partial_{\alpha} A = e_{\alpha_{1}} e_{\alpha_{2}}^{\top} - e_{\alpha_{2}} e_{\alpha_{1}}^{\top}
.
$
Hence $\partial_{\alpha} \partial_{\beta} A = 0$ and
$
q_{a}^{\top}\partial_{\alpha} A q_{b} = q_{a}^{\alpha_{1}} q_{b}^{\alpha_{2}} - q_{b}^{\alpha_{1}} q_{a}^{\alpha_{2}} = [J_{a, b}]_{\alpha}
.
$
Now take derivatives of the following equation.
\begin{align*}
Aq_{2k-1} &= - \lambda_{k} q_{2k} \\
\partial_{\alpha} A \; q_{2k-1} + A \partial q_{2k-1} &= - \partial_{\alpha} \lambda_{k} q_{2k} - \lambda_{k} \partial_{\alpha} q_{2k} \\
\partial_{\alpha} \partial_{\beta} A q_{2k-1}
+ \partial_{\beta} A \; \partial_{\alpha} q_{2k-1} 
+ \partial_{\alpha} A\; \partial_{\beta} q_{2k-1}
+ A \partial_{\alpha} \partial_{\beta}q_{2k-1}
&=
- \partial_{\alpha} \partial_{\beta} \lambda_{k} \; q_{2k}
- \partial_{\beta} \lambda_{k} \partial_{\alpha} q_{2k} \\
&- \partial_{\alpha} \lambda_{k} \partial_{\beta} q_{2k}
- \lambda_{k} \partial_{\alpha} \partial_{\beta} q_{2k} 
\end{align*}
\begin{align*}
Aq_{2k} &= \lambda_{k} q_{2k-1} \\
\partial_{\alpha} A \; q_{2k} + A \partial q_{2k} &= \partial_{\alpha} \lambda_{k} q_{2k-1} + \lambda_{k} \partial_{\alpha} q_{2k-1} \\
\partial_{\alpha} \partial_{\beta} A q_{2k}
+ \partial_{\beta} A \; \partial_{\alpha} q_{2k} 
+ \partial_{\alpha} A\; \partial_{\beta} q_{2k}
+ A \partial_{\alpha} \partial_{\beta}q_{2k}
&=
\partial_{\alpha} \partial_{\beta} \lambda_{k} \; q_{2k-1}
+ \partial_{\beta} \lambda_{k} \partial_{\alpha} q_{2k-1} \\
&+ \partial_{\alpha} \lambda_{k} \partial_{\beta} q_{2k-1}
+ \lambda_{k} \partial_{\alpha} \partial_{\beta} q_{2k-1}
\end{align*}
First focus on the first level derivatives and take its inner product with \(q\). To simplify use orthogonality and that \(x^\top \partial_{\alpha} A x = 0\). For \(l \neq k\),
\begin{align*}
\lambda_{k}\; q_{2k}^{\top} \partial_{\alpha} q_{2k-1} &= - \lambda_{k} q_{2k-1}^{\top} \partial_{\alpha} q_{2k} \\
q_{2k}^{\top} \partial_{\alpha} A \; q_{2k-1}  &= - \partial_{\alpha} \lambda_{k} \\
q_{2l-1}^{\top}\partial_{\alpha} A \; q_{2k-1} + \lambda_{l}\; q_{2l}^{\top} \partial_{\alpha} q_{2k-1} &=  - \lambda_{k} q_{2l-1}^{\top} \partial_{\alpha} q_{2k} \\
q_{2l}^{\top}\partial_{\alpha} A \; q_{2k-1} - \lambda_{l}\; q_{2l-1}^{\top} \partial_{\alpha} q_{2k-1} &= - \lambda_{k} q_{2l}^{\top} \partial_{\alpha} q_{2k}
\end{align*}
\begin{align*}
q_{2k-1}^{\top}\partial_{\alpha} A \; q_{2k}  &= \partial_{\alpha} \lambda_{k} \\
- \lambda_{k} q_{2k-1}^{\top} \partial_{\alpha} q_{2k} &= \lambda_{k} q_{2k}^{\top}\partial_{\alpha} q_{2k-1} \\
q_{2l-1}^{\top}\partial_{\alpha} A \; q_{2k} + \lambda_{l} q_{2l}^{\top} \partial_{\alpha} q_{2k} &= \lambda_{k} q_{2l-1}^{\top} \partial_{\alpha} q_{2k-1} \\
q_{2l}^{\top}\partial_{\alpha} A \; q_{2k} -\lambda_{l} q_{2l-1}^{\top} \partial_{\alpha} q_{2k} &= \lambda_{k} q_{2l}^{\top} \partial_{\alpha} q_{2k-1}
\end{align*}
These can be grouped into pairs. The first pair is a trivial equality which places no restriction \(S_{k}\) \footnote{This corresponds to the non-uniqueness of the block diagonalisation. I.e., each singular subspace has dimension 2 and the basis \(q_{2k-1}, q_{2k}\) can be rotated.}. The second pair shows \(\partial_{\alpha} \lambda_k = [J_{k}]_{\alpha}\). Differentiating again completes the proof of \eqref{equation: sval derivatives}.
\begin{align*}
\partial_{\beta} \lambda_{k} &= q_{2k-1}^{\top} \partial_{\beta} A q_{2k} \\
\partial_{\alpha}\partial_{\beta} \lambda_{k} &= \partial_{\alpha}q_{2k-1}^{\top} \partial_{\beta} A q_{2k} + q_{2k-1}^{\top} \partial_{\beta} A \partial_{\alpha}q_{2k} \\
&= q_{2k-1}^{\top} \partial_{\beta} A \partial_{\alpha} q_{2k} - q_{2k}^{\top} \partial_{\beta} A \partial_{\alpha} q_{2k -1} \\
&= \sum_{a} [J_{2k-1, a}]_{\beta} [S_{a, 2k}]_{\alpha} - [J_{2k, a}]_{\beta} [S_{a, 2k-1}]_{\alpha} \\
\end{align*}
The remaining pairs yield linear systems described by the matrices in \eqref{eq: ST system matrix}. Inverting yields \eqref{eq: S_a,b identity}. Apply the same technique to the second level derivatives to obtain \eqref{eq: Symmetrised Tk expansion} and \eqref{eq: T_a,b identity}.
\begin{align*}
[U_{2k-1, 2k-1}]_{\beta, \alpha} 
+ [U_{2k-1, 2k-1}]_{\alpha, \beta}
+ \lambda_{k} [T_{2k, 2k-1}]_{\alpha, \beta}
&=
- [J_{k}]_{\beta} [S_{k}]_{\alpha}
- [J_{k}]_{\alpha} [S_{k}]_{\beta}
- \lambda_{k}  [T_{2k-1, 2k}]_{\alpha, \beta} \\
[U_{2k, 2k-1}]_{\beta, \alpha}
+ [U_{2k, 2k-1}]_{\alpha, \beta}
- \lambda_{k}[T_{2k-1, 2k-1}]_{\alpha, \beta}
&=
- \partial_{\alpha} \partial_{\beta} \lambda_{k}
- \lambda_{k} [T_{2k, 2k}]_{\alpha, \beta} \\
[U_{2l-1, 2k-1}]_{\beta, \alpha} 
+ [U_{2l-1, 2k-1}]_{\alpha, \beta}
+ \lambda_{l} [T_{2l, 2k-1}]_{\alpha, \beta}
&=
- [J_{k}]_{\beta} [S_{2l-1, 2k}]_{\alpha}
- [J_{k}]_{\alpha} [S_{2l-1, 2k}]_{\beta} 
- \lambda_{k} [T_{2l-1, 2k}]_{\alpha, \beta} \\
[U_{2l, 2k-1}]_{\beta, \alpha} 
+ [U_{2l, 2k-1}]_{\alpha, \beta}
- \lambda_{l} [T_{2l-1, 2k-1}]_{\alpha, \beta}
&=
- [J_{k}]_{\beta} [S_{2l, 2k}]_{\alpha}
- [J_{k}]_{\alpha} [S_{2l, 2k}]_{\beta}
- \lambda_{k} [T_{2l, 2k}]_{\alpha, \beta}
\end{align*}
\begin{align*}
[U_{2k-1, 2k}]_{\beta, \alpha} 
+ [U_{2k-1, 2k}]_{\alpha, \beta}
+ \lambda_{k} [T_{2k, 2k}]_{\alpha, \beta}
&=
\partial_{\alpha} \partial_{\beta} \lambda_{k}
+ \lambda_{k} [T_{2k-1, 2k-1}]_{\alpha, \beta} \\
[U_{2k, 2k}]_{\beta, \alpha}
+ [U_{2k, 2k}]_{\alpha, \beta}
- \lambda_{k}[T_{2k-1, 2k}]_{\alpha, \beta}
&=
- [J_{k}]_{\beta} [S_{k}]_{\alpha}
- [J_{k}]_{\alpha} [S_{k}]_{\beta}
+ \lambda_{k} [T_{2k, 2k-1}]_{\alpha, \beta} \\
[U_{2l-1, 2k}]_{\beta, \alpha} 
+ [U_{2l-1, 2k}]_{\alpha, \beta}
+ \lambda_{l} [T_{2l, 2k}]_{\alpha, \beta}
&=
[J_{k}]_{\beta} [S_{2l-1, 2k-1}]_{\alpha}
+ [J_{k}]_{\alpha} [S_{2l-1, 2k-1}]_{\beta}
+ \lambda_{k} [T_{2l-1, 2k-1}]_{\alpha, \beta} \\
[U_{2l, 2k}]_{\beta, \alpha} 
+ [U_{2l, 2k}]_{\alpha, \beta}
- \lambda_{l} [T_{2l-1, 2k}]_{\alpha, \beta}
&=
[J_{k}]_{\beta}[S_{2l, 2k-1}]_{\alpha}
+ [J_{k}]_{\alpha} [S_{2l, 2k-1}]_{\beta}
+ \lambda_{k}[T_{2l, 2k-1}]_{\alpha, \beta}
\end{align*}
Equations \eqref{eq: J multiplication}, \eqref{eq: JS multiplication}, \eqref{eq: JkS multiplication}, \eqref{eq: S_a,b squared} follow from orthogonality.
\end{proof}

\subsection{Derivation of the Singular Value SDE}
\label{appendix:Derivation of the Singular Value SDE}

To get the SDE, we apply It\^o formula. The following results will be useful when computing the sums that appear in It\^o formula. Throughout we let $R_k = \sqrt{u_k}$.
\begin{lemma}[L\'evy's Area contraction identities] \label{lemma:L\'evy's Area contraction}
Let \(\bA_t\) be the matrix of L\'evy's areas and let \(f: \mathbb{N}^2 \to \mathbb{R}\) be a skew-symmetric function \(f(\alpha') = -f(\alpha)\) where $\alpha'=(\alpha_2, \alpha_1)$ is the transpose, then
    \begin{equation}
        \sum_{\alpha} f(\alpha) d\bA_{a} = \frac{1}{2}\sum_{i, j} f(i, j) B_{i} dB_{j}
    \end{equation}
    For \(g: \mathbb{N}^2 \times  \to \mathbb{R}\) skew in both variables separately \(g(\alpha', \beta) = g(\alpha, \beta') = - g(\alpha, \beta)\) then
    \begin{equation}
        \sum_{\alpha, \beta} g(\alpha, \beta) d \ip{\bA_{\alpha}}{\bA_{\beta}}
= \frac{1}{4}
\sum_{i, j, k} g((i, j), (k, j)) B_{i} B_{k}
dt
.
    \end{equation}
\end{lemma}
\begin{proof}
\(d\bA_{\alpha} = \frac{1}{2} (B_{\alpha_1} dB_{\alpha_2} - B_{\alpha_2} dB_{\alpha_1})\). Then
\begin{equation*}
    \sum_{i < j} f(i, j) (B_{i} dB_{j} - B_{j} dB_{i})
    =
    \sum_{i < j} f(i, j) B_{i} dB_{j} + f(j, i) B_{j} dB_{i}
    =
    \sum_{i < j} f(i, j) B_{i} dB_{j}
    + \sum_{j < i} f(i, j) B_{i} dB_{j}
\end{equation*}
This yields the full sum over \(i, j\) as \(i = j\) has zero contribution. Similarly,
\begin{equation*}
d\ip{\bA_{\alpha}}{\bA_{\beta}}
= \frac{1}{4}
(B_{\alpha_{1}} B_{\beta_{1}}\, \delta_{{\alpha_{2}}{\beta_{2}}}
- B_{\alpha_{1}} B_{\beta_{2}}\, \delta_{{\alpha_{2}}{\beta_{1}}}
- B_{\alpha_{2}} B_{\beta_{1}}\, \delta_{{\alpha_{1}}{\beta_{2}}}
+ B_{\alpha_{2}} B_{\beta_{2}}\, \delta_{{\alpha_{1}}{\beta_{1}}})dt
\end{equation*}
and the sum weighted by \(g\) over \(\alpha = (i < j), \beta = (k < l)\) can be combined to recover a sum over \(i, j, k, l\). Finally, \(\delta_{lj}\) forces \(l = j\).
\end{proof}

Next, provided the singular values are distinct the following are Brownian motions.
\begin{lemma}[Radial and angular Brownian motions]
\label{lemma:Radial and angular Brownian motions}
    Let \(\bA_t\) be the matrix of L\'evy's areas which we block diagonalise into \(\bA_t = Q J(\Lambda) Q^{\top}\), denoting the columns of \(Q\) with \(q\). Consider the rotation and projection on the \(k\)-th singular subspace
    \begin{equation*}
        J_{k} = q_{2k-1} q_{2k}^{\top} - q_{2k} q_{2k-1}^{\top},\quad
        P_{k} = q_{2k-1} q_{2k-1}^{\top} + q_{2k} q_{2k}^{\top}
        .
    \end{equation*}
    Let \(R_k = \lp{J_k B} = \lp{P_k B}\) be the norm of Brownian motion on the \(k\)-th singular subspace.
    Then, the following are independent Brownian motions which we define as the radial and angular Brownian motions respectively.
    \begin{equation*}
        d\gamma_{k} = \frac{B^{\top} P_{k} dB}{R_{k}},\quad
d\theta_{k} = \frac{B^{\top} J_{k} dB}{R_{k}}
.
    \end{equation*}
\end{lemma}
\begin{proof}
By L\'evy's characterisation of Brownian motion, it suffices to check their quadratic variation. As \(P_l^{\top} P_k = J_l^{\top} J_l = \delta_{lk} P_k\)
\begin{align*}
d\gamma_{k} d\gamma_{l} &= \frac{B^{\top} P_{k} dB dB^{\top} P_{l}^{\top}B}{R_{k}^{2}} 
= \frac{B^{\top} P_{k} B}{R_{k}^{2}} \delta_{kl} dt
= \delta_{kl} dt \\
d\theta_{k} d\theta_{l} &= \frac{B^{\top} J_{k} dB dB^{\top} J_{l}^{\top}B}{R_{k}^{2}}
= \frac{B^{\top} P_{k} B}{R_{k}^{2}} \delta_{kl} \; dt
= \delta_{kl} dt
.
\end{align*}
Finally, $P_k J_l^{\top} = \delta_{kl} J_k^{\top}$ is skew-symmetric and $P_l$ is symmetric, hence $P_{k} J_{l}$ is skew-symmetric and
\begin{align*}
d\gamma_{k} d\theta_{l} &= \frac{B^{\top} P_{k} dB dB^{\top} J_{l}^{\top}B}{R_{k}^{2}}
= \frac{B^{\top} J_{k}^{\top} B}{R_{k}^{2}} \delta_{kl} dt
= 0
.
\end{align*}
\end{proof}

\begin{lemma} \label{lemma:Derivation of singular value SDE}
    Let $d = 2n$ be even. The singular values \(\lambda_k\) of the matrix of L\'evy's areas and squared norm \(R_k^2\) of Brownian motion on each singular subspace satisfy the SDE given in Theorem \ref{theorem:SDE of Singular values and radii}.
\end{lemma}
\begin{proof}
The matrix of L\'evy's areas and Brownian motion satisfy the following SDE.
\begin{align*}
    d\bA^{i, j}_t &= \frac{1}{2}(B^i_t dB^j_t - B^j_t dB^i_t)\\
    dB^i_t &= dB^i_t
\end{align*}
The singular values and vectors are functions of \(\bA\). For \(\alpha_1 < \alpha_2\), let \(\partial_{\alpha} := \frac{\partial}{\partial A_{\alpha}}\). By It\^o Formula,
\begin{align*}
d\lambda_{k} &= \sum_{\alpha} \partial_{\alpha} \lambda_{k} d\bA_{\alpha} + \frac{1}{2} \sum_{\alpha, \beta} \partial_{\alpha} \partial_{\beta} \lambda_{k} \;d \ip{\bA_{\alpha}}{\bA_{\beta}} \\
dq_{a} &= \sum_{\alpha} \partial_{\alpha} q_{a} d\bA_{\alpha} + \frac{1}{2} \sum_{\alpha, \beta} \partial_{\alpha} \partial_{\beta}  q_{a} \;d \ip{\bA_{\alpha}}{\bA_{\beta}}
\end{align*}
The squared norms are given by \(R_k^2 = B^{\top} P_k B\) and we obtain $dR_k^2$ via product rule once we have \(dP_k\).
Formulas for the derivatives with respect to $\bA_{\alpha}$ are given in Lemma \ref{lemma:derivatives of singular vectors and values} and we simplify the summation over \(\alpha, \beta\) via Lemma \ref{lemma:L\'evy's Area contraction}. As in Lemma \ref{lemma:Radial and angular Brownian motions}, we consider the Brownian motions \(\gamma_k, \theta_k\)
\begin{equation*}
d\gamma_{k} = \frac{B^{\top} P_{k} dB}{R_{k}},\quad
d\theta_{k} = \frac{B^{\top} J_{k} dB}{R_{k}}
.
\end{equation*}

\textbf{1. Singular values}
Using \eqref{equation: sval derivatives}, the diffusion term is
\begin{equation*}
    \sum_{\alpha} \partial_{\alpha} \lambda_{k} dA_{\alpha}
= \frac{1}{2} \sum_{i, j} \partial_{ij} \lambda_{k} B_{i} dB_{j}
= \frac{1}{2} \sum_{i, j} [J_{k}]_{ij} B_{i} dB_{j}
= \frac{1}{2} B^{\top} J_{k} dB
= \frac{1}{2} R_k d \theta_k
\end{equation*}
and the drift term is
$$
\frac{1}{2} \sum_{\alpha, \beta} \partial_{\alpha} \partial_{\beta} \lambda_{k} \;d \ip{\bA_{\alpha}}{\bA_{\beta}}
=
\frac{1}{8} \sum_{i, j, x} B_{i} B_{x}
\sum_{a} [S_{a, 2k}]_{ij} [J_{a, 2k-1}]_{xj}  - [S_{a, 2k-1}]_{\alpha} [J_{a, 2k}]_{xj}
.
$$
Use symmetry to swap $i, x$ and sum over $j$ to show the inner term is
$$
[J_{2k-1, a}^{\top} S_{a, 2k} - J_{2k, a}^{\top} S_{a, 2k-1}]_{ij} 
.
$$
The contribution of $a \in \{ 2k-1, 2k \}$ is 0 as $S_{a,a} = 0 = J_{a, a}$. For $l \neq k$, use \eqref{eq: S_a,b identity} and \ref{eq: J multiplication} to show the contribution of $a \in \set{2l-1, 2l}$ is
\begin{equation*}
    \begin{split}
D_{kl}(J_{2k-1, 2l-1}^{\top} S_{2l-1, 2k} - J_{2k, 2l-1}^{\top} S_{2l-1, 2k-1}
+ J_{2k-1, 2l}^{\top} S_{2l, 2k} - J_{2k, 2l}^{\top} S_{2l, 2k-1})
\\ =
\lambda_k (E_{2l-1, 2l-1} + E_{2k-1, 2k-1}) + \lambda_k(E_{2k, 2k} + E_{2l-1, 2l-1}) \\
+ \lambda_k(E_{2l, 2l} + E_{2k-1, 2k-1}) +  \lambda_k( E_{2l, 2l} +E_{2k, 2k}) \\
= 2 \lambda_k (P_l + P_k)
\end{split}
\end{equation*}
Summing over \(i, j\) yields
$$
\frac{1}{2} \sum_{\alpha, \beta} \partial_{\alpha} \partial_{\beta} \lambda_{k} \;d \ip{\bA_{\alpha}}{\bA_{\beta}}
=
\frac{\lambda_{k}}{4} \sum_{l \neq k}
\frac{R_{k}^{2} + R_{l}^{2}}{\lambda_{k}^{2} - \lambda_{l}^{2}}
.
$$

\textbf{2. Singular radii} By It\^o product rule and symmetry of $P_{k}$
$$
dR_{k}^{2} = 2 B^{\top} P_{k} dB + B^{\top} dP_{k} B + \{2 B^{\top}dP_{k} dB + dB^{\top} P_{k} dB\}
$$
The first term corresponds to the \(k\)-th radial Brownian motion $2 B^{\top} P_{k} dB \to 2 R_{k} d\gamma_{k}$. Also
$$
dB^{\top} P_{k} dB = \sum_{i, j} (P_{k})_{ij} \delta_{ij} dt = \mathrm{tr}(P_{k})dt = 2dt
$$
which recovers the usual SDE for a squared Bessel process. It remains to compute \(dP_k\).

If we decompose $q$ into its diffusion and drift term
\begin{align*}
dq_{2k-1} &= dg_{2k-1} + dh_{2k-1} \\
dq_{2k} &= dg_{2k} + dh_{2k}
\end{align*}
then the diffusion and drift term of $dP_{k} = dC_{k} + dD_{k}$ is
\begin{align*}
dC_{k} &= q_{2k-1} dg_{2k-1}^{\top} + dg_{2k-1} q_{2k-1}^{\top} + q_{2k} dg_{2k}^{\top} + dg_{2k} q_{2k}^{\top} \\
dD_{k} &= q_{2k-1} dh_{2k-1}^{\top} + dh_{2k-1} q_{2k-1}^{\top} + q_{2k} dh_{2k}^{\top} + dh_{2k} q_{2k}^{\top} \\
&+ dg_{2k-1} dg_{2k-1}^{\top} + dg_{2k} dg_{2k}^{\top}
.
\end{align*}
We now decompose $dR_{k}^{2}$ as
$$
2 R_{k} d\gamma_{k} + B^{\top} dC_{k} B
+ \set{B^{\top} dD_{k} B + 2 B^{\top} dC_{k} dB + 2 dt}
.
$$
Using the expressions below we get the diffusion term as
$$
2 R_k d\gamma_k + R_{k}\sum_{l \neq k } \frac{\lambda_{k} R_{l}^2 d\theta_{k}
+ \lambda_{l} R_{k} R_{l} d\theta_{l}}{\lambda_{k}^{2} - \lambda_{l}^{2}}
$$
and drift term
$$
2
+
\frac{\lambda_{k}^{2}}{2}
\sum_{l, r \neq k, l \neq r} \frac{R_{l}^{2}(R_r^2 + 2R_{k}^{2})}{(\lambda_{k}^{2} - \lambda_{l}^{2})(\lambda_{k}^{2} - \lambda_{r}^{2})}
+
\frac{1}{2}\sum_{l \neq k} \frac{\lambda_k^2 R_l^4 - \lambda_l^2 R_k^4}{(\lambda_k^2 -\lambda_l^2)^2}
dt
$$

\textbf{2.1 Contribution of the projection's diffusion term} We show
\begin{equation*}
B^{\top} dC_{k} B
=
R_{k}\sum_{l \neq k } \frac{\lambda_{k} R_{l}^2 d\theta_{k}
+ \lambda_{l} R_{k}R_{l} d\theta_{l}}{\lambda_{k}^{2} - \lambda_{l}^{2}},\quad
B^{\top} dC_{k} dB = 0
\end{equation*}
First consider the diffusion term of the singular vectors.
\begin{align*}
dg_{2k-1}
&=
\sum_{\alpha} \partial_{a} q_{2k-1} d\bA_{\alpha}
=
\frac{1}{2} \sum_{i, j} B_{i} dB_{j}
\sum_{a} q_{a} [S_{a, 2k-1}]_{ij} \\
dg_{2k}
&=
\sum_{\alpha} \partial_{a} q_{2k} d\bA_{\alpha}
=
\frac{1}{2} \sum_{i, j} B_{i} dB_{j}
\sum_{a} q_{a} [S_{a, 2k}]_{ij}
.
\end{align*}
Hence,
\begin{equation*}
    dC_k
    = \frac{1}{2} \sum_{i, j} B_{i} dB_{j}
    \sum_{a}
(q_{2k-1}q_{a}^{\top} + q_{a} q_{2k-1}^{\top}) [S_{a, 2k-1}]_{ij}
+ (q_{2k}q_{a}^{\top} + q_{a} q_{2k}^{\top}) [S_{a, 2k}]_{ij}
.
\end{equation*}
The contribution of $a \in  \{ 2k-1, 2k \}$ is 0 as the coefficients will match and $S_{a, b} = -S_{b, a}$.
For $l \neq k$, consider the \((x, y)\)-th component of the contribution of \(a = 2l-1, 2l\). Expand \(S\) using \eqref{eq: S_a,b identity} and withhold the \((\lambda_k^2 - \lambda_l^2)\) factor to show
\begin{equation*}
\begin{split}
(q_{2k-1}^{x}q_{2l-1}^{y} + q_{2l-1}^{x} q_{2k-1}^{y}) [\lambda_{k}(q_{2l-1}^{i} q_{2k}^{j} - q_{2k}^{i}q_{2l-1}^{j}) - \lambda_{l}(q_{2l}^{i} q_{2k-1}^{j} - q_{2k-1}^{i}q_{2l}^{j})] \\
+ (q_{2k}^{x}q_{2l-1}^{y} + q_{2l-1}^{x} q_{2k}^{y}) [ -\lambda_{l}(q_{2l}^{i} q_{2k}^{j} - q_{2k}^{i}q_{2l}^{j}) - \lambda_{k}(q_{2l-1}^{i} q_{2k-1}^{j} - q_{2k-1}^{i}q_{2l-1}^{j})] \\
+ (q_{2k-1}^{x}q_{2l}^{y} + q_{2l}^{x} q_{2k-1}^{y}) [\lambda_{k}(q_{2l}^{i} q_{2k}^{j} - q_{2k}^{i}q_{2l}^{j}) + \lambda_{l} (q_{2l-1}^{i} q_{2k-1}^{j} - q_{2k-1}^{i}q_{2l-1}^{j})] \\
+ (q_{2k}^{x}q_{2l}^{y} + q_{2l}^{x} q_{2k}^{y}) [\lambda_{l}(q_{2l-1}^{i} q_{2k}^{j} - q_{2k}^{i}q_{2l-1}^{j}) - \lambda_{k}(q_{2l}^{i} q_{2k-1}^{j} - q_{2k-1}^{i}q_{2l}^{j})]
.
\end{split}
\end{equation*}
The coefficient of $\lambda_{k}$ is given below and can be regrouped into $P_{l}^{yi} J_{k}^{xj} + P_{l}^{xi} J_{k}^{yj} + P_{l}^{yj} J_{k}^{ix} + P_{l}^{xj} J_{k}^{iy}$.
\begin{equation*}
\begin{split}
(q_{2k-1}^{x}q_{2l-1}^{y} + q_{2l-1}^{x} q_{2k-1}^{y})(q_{2l-1}^{i} q_{2k}^{j} - q_{2k}^{i}q_{2l-1}^{j})  \\
- (q_{2k}^{x}q_{2l-1}^{y} + q_{2l-1}^{x} q_{2k}^{y})(q_{2l-1}^{i} q_{2k-1}^{j} - q_{2k-1}^{i}q_{2l-1}^{j}) \\
+ (q_{2k-1}^{x}q_{2l}^{y} + q_{2l}^{x} q_{2k-1}^{y})(q_{2l}^{i} q_{2k}^{j} - q_{2k}^{i}q_{2l}^{j}) \\
- (q_{2k}^{x}q_{2l}^{y} + q_{2l}^{x} q_{2k}^{y})(q_{2l}^{i} q_{2k-1}^{j} - q_{2k-1}^{i}q_{2l}^{j})
.
\end{split}
\end{equation*}
Similarly, the coefficient of $\lambda_{l}$ can be regrouped into $P_{k}^{xi} J_{l}^{yj} + P_{k}^{yi} J_{l}^{xj} + P_{k}^{xj}J_{l}^{iy} + P_{k}^{yj} J_{l}^{ix}$.
\begin{equation*}
\begin{split}
-(q_{2k-1}^{x}q_{2l-1}^{y} + q_{2l-1}^{x} q_{2k-1}^{y})(q_{2l}^{i} q_{2k-1}^{j} - q_{2k-1}^{i}q_{2l}^{j}) \\
- (q_{2k}^{x}q_{2l-1}^{y} + q_{2l-1}^{x} q_{2k}^{y})(q_{2l}^{i} q_{2k}^{j} - q_{2k}^{i}q_{2l}^{j}) \\
+ (q_{2k-1}^{x}q_{2l}^{y} + q_{2l}^{x} q_{2k-1}^{y}) (q_{2l-1}^{i} q_{2k-1}^{j} - q_{2k-1}^{i}q_{2l-1}^{j})\\
+ (q_{2k}^{x}q_{2l}^{y} + q_{2l}^{x} q_{2k}^{y})(q_{2l-1}^{i} q_{2k}^{j} - q_{2k}^{i}q_{2l-1}^{j})
.
\end{split}
\end{equation*}

To compute $B^{\top} dC_{k} B$, we weight the sum over \(i,j,x,y\) by $B_{i} dB_{j} B_{x} B_{y}$. For both $\lambda_{k}, \lambda_{l}$, the last two terms vanish as $B^{\top}J B = 0$ and the first two terms are equal to $B^{\top} P_l B \; B^{\top} J_k dB = R_l^2 R_k d \theta_k$ and $B^{\top} P_k B \; B^{\top} J_l dB = R_k^2 R_l d\theta_l$ respectively. Taking note of the factors yields $B^{\top} dC_{k} B$ as given.

For $B^{\top} dC_{k} dB$, the sum is weighted by $B_{i} B_{x}\delta_{jy} dt$. The terms vanish as either they contain $J_{k}^{jj} = 0$ or are of the form $B^{\top} P_{k} J_{l} B = 0$.

\textbf{Contribution of the projection's drift term} We show
\begin{equation*}
    B^{\top} d D_k  B
    =
    \frac{\lambda_{k}^{2}}{2}
    \sum_{l, r \neq k, l \neq r} \frac{R_{l}^{2}(R_r^2 + 2R_{k}^{2})}{(\lambda_{k}^{2} - \lambda_{l}^{2})(\lambda_{k}^{2} - \lambda_{r}^{2})}
    dt
    +
    \frac{1}{2}\sum_{l \neq k} \frac{\lambda_k^2 R_l^4 - \lambda_l^2 R_k^4}{(\lambda_k^2 -\lambda_l^2)^2}
    dt
\end{equation*}
To proceed,  decompose $dD_{k} = dG_k + dH_k$ where
\begin{equation*}
\begin{split}
dG_k &= dg_{2k-1} dg_{2k-1}^{\top} + dg_{2k} dg_{2k}^{\top} \\
dH_k &= q_{2k-1} dh_{2k-1}^{\top} + dh_{2k-1} q_{2k-1}^{\top} + q_{2k} dh_{2k}^{\top} + dh_{2k} q_{2k}^{\top}
.
\end{split}
\end{equation*}
We first show the contribution of terms involving \(g\) is (omitting $dt$ for notational clarity)
\begin{equation*}
\begin{split}
-4 B^{\top}dG_k B
=
R_k^2 B^\top S_k^2 B
+
\sum_{l \neq k} \frac{2\lambda_k (R_l^2 B^{\top} S_k P_k B - R_k^2B^{\top}S_k P_l B)}{(\lambda_{k}^{2} - \lambda_{l}^{2})^2} \\
-
\sum_{l, r \neq k, l \neq r} \frac{2 \lambda_{k}^{2} R_{l}^{2} R_{r}^{2}}{(\lambda_{k}^{2} - \lambda_{l}^{2})(\lambda_{k}^{2} - \lambda_{r}^{2})}
-
\sum_{l \neq k} \frac{\lambda_{k}^{2}R_{l}^{2}(R_{k}^{2} + 2R_{l}^{2}) + \lambda_{l}^{2}R_{k}^{2}R_{l}^{2}}{(\lambda_{k}^{2} - \lambda_{l}^{2})^{2}}
.
\end{split}
\end{equation*}
Next the terms including $h$, contribute
\begin{equation*}
\begin{split}
4B^{\top} dH_k B
=
R_{k}^{2}B^{\top} S_{k}^{2} B + \sum_{l \neq k}
\frac{- (\lambda_{k}^{2} + \lambda_{l}^{2}) R_{k}^{2} R_{l}^{2}  - 2\lambda_{l}^{2} R_{k}^{4}}{(\lambda_{k}^{2} - \lambda_{l}^{2})} \\
+ 
\sum_{l, r \neq k, l \neq r} \frac{4\lambda_{k}^{2}R_{l}^{2}R_{k}^{2}}{(\lambda_{k}^{2} - \lambda_{l}^{2})(\lambda_{k}^{2} - \lambda_{r}^{2})}
+ \sum_{l \neq k}\frac{2\lambda_{k}(R_{k}^{2} B^{\top} P_{l}S_{k} B  - R_{l}^{2} B^{\top}P_{k} S_{k} B)}{(\lambda_{k}^{2} - \lambda_{l}^{2})^2}
.
\end{split}
\end{equation*}
Summing these together, terms including $S_k$ vanish. And the remaining terms can be grouped into a sum over \(k, l, r\) and \(l, k\).

\textbf{2.2 Contribution of the projection's drift term, Part 1}
Use the expansion
\[
dg_{2k-1} dg_{2k-1}^{\top} + dg_{2k} dg_{2k}^{\top}
=
\frac{1}{4} \sum_{i,j,x,y} B_{i} dB_{j} B_{x} dB_{y}
\sum_{a, b} q_{a} q_{b}^{\top} (
[S_{a, 2k-1}]_{ij} [S_{b, 2k-1}]_{xy}
+ [S_{a, 2k}]_{ij} [S_{b, 2k}]_{xy}
)
.
\]
As $dB_{j}dB_{y} = \delta_{jy} dt$, the sum over $j = y$ simplifies into (note $S$ is skew-symmetric)
$$
\frac{-1}{4}\sum_{i, j} B_{i} B_{j}
\sum_{a, b} q_{a} q_{b}
[S_{a, 2k-1}S_{b, 2k-1} + S_{a, 2k}S_{b, 2k}^{\top}]_{ij}
.
$$
Therefore, the contribution to $B^{\top} dD_{k} B$ is, where $W = Q^{\top} B$,
$$
\frac{-1}{4}\sum_{i, j,x, y} B_{i} B_{j}
\sum_{a, b} W_{a} W_{b}
[S_{a, 2k-1}S_{b, 2k-1} + S_{a, 2k}S_{b, 2k}]_{ij}
.
$$

We break into the following cases: \begin{enumerate*}
    \item \(a, b \in \set{2k-1, 2k}\)
    \item \(a\in \set{2k-1, 2k} \) with \(b \notin \set{2k-1, 2k}\) (and at the same time \(b\in \set{2k-1, 2k}\) with \(a \notin \set{2k-1, 2k}\))
    \item \(a, b \notin \set{2k-1, 2k}\)
    .
\end{enumerate*}
In each case, the stated contribution excludes the overall factor of $-1/4$.

\textbf{Case 1}
For $a, b \in \set{2k-1, 2k}$, the contribution is
$$
R_k^2 B^\top S_k^2 B
.
$$
As $S_{a, a} = 0$, the corresponding sum is only $R_{k}^{2}
[S_{k}^{2}]_{ij}$. Summing over $i, j$ yields the result.

\textbf{Case 2}
The contribution is
$$
2
\sum_{l \neq k} \frac{\lambda_k (R_l^2 B^{\top} S_k P_k B - R_k^2B^{\top}S_k P_l B)}{D_{kl}^2}
.
$$
This is the sum of \((a, c), (c, a)\) for $a \in \{ 2k -1, 2k \}$, which simplifies into
\begin{equation*}
2
\sum_{c \notin \{ 2k-1, 2k \}}
W_{2k-1} W_{c}
[S_{k}S_{c, 2k}]_{ij}
- W_{2k} W_{c}
[S_k S_{c, 2k-1}]_{ij}
.
\end{equation*}
Consider each sum for \(c \in \set{2l-1, 2l}\).
Factor out \(S_k\) by expanding the product (i.e., $[S_k S_{a, b}]_{ij} = \sum_{r} [S_k]_{ir} [S_{a, b}]_{rj}$) and consider the \((r, j)\)-th component.
\begin{equation*}
\begin{split}
W_{2k-1} W_{2l-1} ( -\lambda_{l}J_{2l,2k} - \lambda_{k}J_{2l-1, 2k-1})_{rj} 
+ W_{2k-1} W_{2l} (\lambda_{l} J_{2l-1,2k} - \lambda_{k}J_{2l, 2k-1})_{rj}
\\- W_{2k} W_{2l-1} (\lambda_{k} J_{2l-1, 2k} - \lambda_{l}J_{2l, 2k-1})_{rj}
- W_{2k} W_{2l} (\lambda_{k}J_{2l, 2k} + \lambda_{l}J_{2l-1, 2k-1})_{rj}
\end{split}
\end{equation*}
Expand $S$, which yields a factor of $D_{kl}^2$. Sum over \(j\), using that $J_{a, b}B = q_a W_b - q_b W_a$.
\begin{equation*}
\begin{split}
W_{2k-1} W_{2l-1} ( -\lambda_{l}(q_{2l} W_{2k} - q_{2k}W_{2l}) - \lambda_{k}(q_{2l-1}W_{2k-1}-q_{2k-1}W_{2l-1}))_{r} \\
+ W_{2k-1} W_{2l} (\lambda_{l} (q_{2l-1} W_{2k} - q_{2k}W_{2l-1}) - \lambda_{k}(q_{2l}W_{2k-1} - q_{2k-1} W_{2l}))_{r} \\
- W_{2k} W_{2l-1} (\lambda_{k} (q_{2l-1} W_{2k} - q_{2k}W_{2l-1}) - \lambda_{l}(q_{2l}W_{2k-1} - q_{2k-1} W_{2l}))_{r} \\
- W_{2k} W_{2l} (\lambda_{k}(q_{2l} W_{2k} - q_{2k}W_{2l}) + \lambda_{l}(q_{2l-1}W_{2k-1}-q_{2k-1}W_{2l-1}))_{r}
\end{split}
\end{equation*}
For each $l$, consider the coefficients of $\lambda$.
\begin{align*}
\lambda_{k} \to \{ -W_{2k-1} W_{2l-1}(q_{2l-1}W_{2k-1}-q_{2k-1}W_{2l-1}) \\
-W_{2k-1} W_{2l} (q_{2l}W_{2k-1} - q_{2k-1} W_{2l}) \\
- W_{2k} W_{2l-1} (q_{2l-1} W_{2k} - q_{2k}W_{2l-1}) \\
- W_{2k} W_{2l} (q_{2l} W_{2k} - q_{2k}W_{2l})
\} \\
=  R_{l}^{2}(q_{2k-1} W_{2k-1} + q_{2k} W_{2l}) - R_{k}^{2} (q_{2l-1}W_{2l-1} + q_{2l} W_{2l}) \\
\lambda_{l} \to \{ 
-W_{2k-1} W_{2l-1} (q_{2l} W_{2k} - q_{2k}W_{2l}) \\
W_{2k-1} W_{2l} (q_{2l-1} W_{2k} - q_{2k}W_{2l-1}) \\
W_{2k} W_{2l-1}(q_{2l}W_{2k-1} - q_{2k-1} W_{2l}) \\
- W_{2k} W_{2l}(q_{2l-1}W_{2k-1}-q_{2k-1}W_{2l-1})
\}
= 0
\end{align*}
This shows it is equivalent to summing over \( R_l^2[S_k P_k]_{ij} - R_k^2[S_k P_l]_{ij}\).

\textbf{Case 3}
For $l, r \neq k$, the corresponding sum is
\begin{equation*}
\begin{split}
\sum_{l, r \neq k}
W_{2l-1} W_{2r-1} [S_{2l-1, 2k-1}S_{2r-1, 2k-1} + S_{2l-1, 2k}S_{2r-1, 2k}]_{ij} \\
W_{2l-1} W_{2r} [S_{2l-1, 2k-1}S_{2r, 2k-1}+ S_{2l-1, 2k}S_{2r, 2k}]_{ij} \\
W_{2l} W_{2r-1} [S_{2l, 2k-1}S_{2r-1, 2k-1} + S_{2l, 2k}S_{2r-1, 2k}]_{ij} \\
W_{2l} W_{2r} [S_{2l, 2k-1}S_{2r, 2k-1}+ S_{2l, 2k}S_{2r, 2k}]_{ij}
\end{split}
\end{equation*}
Expand and multiply the $S$ matrices, simplifying via \eqref{eq: J multiplication}. Case on $l \neq r, l = r$. respectively, The contributions are 
$$
-\sum_{l, r \neq k, l \neq r} \frac{2 \lambda_{k}^{2} R_{l}^{2} R_{r}^{2}}{(\lambda_{k}^{2} - \lambda_{l}^{2})(\lambda_{k}^{2} - \lambda_{r}^{2})},\quad
-\sum_{l \neq k} \frac{\lambda_{k}^{2}R_{l}^{2}(R_{k}^{2} + 2R_{l}^{2}) + \lambda_{l}^{2}R_{k}^{2}R_{l}^{2}}{(\lambda_{k}^{2} - \lambda_{l}^{2})^{2}}
.
$$

\textbf{Case 3.1}
If $l \neq r$,
\begin{align*}
S_{2l-1, 2k-1} S_{2r-1, 2k-1} =
(\lambda_{k} J_{2l-1, 2k} - \lambda_{l}J_{2l, 2k-1})
(\lambda_{k} J_{2r-1, 2k} - \lambda_{r}J_{2r, 2k-1})
\\
= -\lambda_{k}^{2} E_{2l-1, 2r-1} - \lambda_{l} \lambda_{r} E_{2l, 2r}\\
S_{2l-1, 2k} S_{2r-1, 2k} =
(-\lambda_{l}J_{2l,2k} - \lambda_{k}J_{2l-1, 2k-1})
(-\lambda_{r}J_{2r,2k} - \lambda_{k}J_{2r-1, 2k-1})
\\
= - \lambda_{l} \lambda_{r} E_{2l, 2r} -\lambda_{k}^{2} E_{2l-1, 2r-1}\\
S_{2l-1, 2k-1} S_{2r, 2k-1} =
(\lambda_{k} J_{2l-1, 2k} - \lambda_{l}J_{2l, 2k-1})
(\lambda_{k}J_{2r, 2k} + \lambda_{r}J_{2r-1, 2k-1})
\\
= -\lambda_{k}^{2} E_{2l-1, 2r} + \lambda_{l} \lambda_{r} E_{2l, 2r-1}\\
S_{2l-1, 2k} S_{2r, 2k} =
(-\lambda_{l}J_{2l,2k} - \lambda_{k}J_{2l-1, 2k-1})
(\lambda_{r} J_{2r-1,2k} - \lambda_{k}J_{2r, 2k-1})
\\
= \lambda_{l} \lambda_{r} E_{2l, 2r-1} - \lambda_{k}^{2} E_{2l-1, 2r}\\
S_{2l, 2k-1} S_{2r-1, 2k-1} =
(\lambda_{k}J_{2l, 2k} + \lambda_{l}J_{2l-1, 2k-1})
(\lambda_{k} J_{2r-1, 2k} - \lambda_{r}J_{2r, 2k-1})
\\
= -\lambda_{k}^{2} E_{2l, 2r-1} + \lambda_{l} \lambda_{r} E_{2l-1, 2r}\\
S_{2l, 2k} S_{2r-1, 2k} =
(\lambda_{l} J_{2l-1,2k} - \lambda_{k}J_{2l, 2k-1})
(-\lambda_{r}J_{2r,2k} - \lambda_{k}J_{2r-1, 2k-1})
\\
= \lambda_{l} \lambda_{r} E_{2l-1, 2r} -\lambda_{k}^{2} E_{2l, 2r-1}\\
S_{2l, 2k-1} S_{2r, 2k-1} =
(\lambda_{k}J_{2l, 2k} + \lambda_{l}J_{2l-1, 2k-1})
(\lambda_{k}J_{2r, 2k} + \lambda_{r}J_{2r-1, 2k-1})
\\
= -\lambda_{k}^{2} E_{2l, 2r} - \lambda_{l} \lambda_{r}E_{2l-1, 2r-1}\\
S_{2l, 2k} S_{2r, 2k} =
(\lambda_{l} J_{2l-1,2k} - \lambda_{k}J_{2l, 2k-1})
(\lambda_{r} J_{2r-1,2k} - \lambda_{k}J_{2r, 2k-1})
\\
= -\lambda_{l}\lambda_{r} E_{2l-1, 2r-1} - \lambda_{k}^{2}E_{2l, 2r}
.
\end{align*}
Sum over $i, j$ and collect by coefficients of $\lambda$
\begin{equation*}
\begin{split}
-
\sum_{\substack{l, r \neq k, l \neq r}}
W_{2l-1}W_{2r-1} [
\lambda_{k}^{2} E_{2l-1, 2r-1} + \lambda_{l} \lambda_{r} E_{2l, 2r}
+ \lambda_{l} \lambda_{r} E_{2l, 2r} +\lambda_{k}^{2} E_{2l-1, 2r-1}
] \\
+W_{2l-1}W_{2r} [
\lambda_{k}^{2} E_{2l-1, 2r} - \lambda_{l} \lambda_{r} E_{2l, 2r-1}
-\lambda_{l} \lambda_{r} E_{2l, 2r-1} + \lambda_{k}^{2} E_{2l-1, 2r}
] \\
+W_{2l} W_{2r-1} [
\lambda_{k}^{2} E_{2l, 2r-1} - \lambda_{l} \lambda_{r} E_{2l-1, 2r}
- \lambda_{l} \lambda_{r} E_{2l-1, 2r} +\lambda_{k}^{2} E_{2l, 2r-1}
] \\
+W_{2l} W_{2r} [
\lambda_{k}^{2} E_{2l, 2r} + \lambda_{l} \lambda_{r}E_{2l-1, 2r-1}
+\lambda_{l}\lambda_{r} E_{2l-1, 2r-1} + \lambda_{k}^{2}E_{2l, 2r}
]
.
\end{split}
\end{equation*}
Collect the coefficients of $\lambda$
\begin{equation*}
\begin{split}
\lambda_{k}^{2} \to
\{ 2W_{2l-1}^{2}W_{2r-1}^{2} + 2W_{2l-1}^{2}W_{2r}^{2} +  2W_{2l}^{2} W_{2r-1}^{2} + 2 W_{2l}^{2} W_{2r}^{2}  \}
= 2 R_{l}^{2} R_{r}^{2} \\
\lambda_{l} \lambda_{r} \to W_{2l-1}W_{2r-1}W_{2l}W_{2r} \{ 2 -2 - 2 + 2  \}
= 0
\end{split}
\end{equation*}
and considering all the factors yields
\[
-
\sum_{l \neq r \neq k} \frac{2 \lambda_{k}^{2} R_{l}^{2} R_{r}^{2}}{(\lambda_{k}^{2} - \lambda_{l}^{2})(\lambda_{k}^{2} - \lambda_{r}^{2})}
.
\]

\textbf{Case 3.2}
If $l = r$,
\begin{align*}
S_{2l-1, 2k-1} S_{2r-1, 2k-1} =
(\lambda_{k} J_{2l-1, 2k} - \lambda_{l}J_{2l, 2k-1})
(\lambda_{k} J_{2r-1, 2k} - \lambda_{r}J_{2r, 2k-1})
\\
=  -\lambda_{k}^{2}(E_{2l-1} + E_{2k}) -\lambda_{l}^{2}(E_{2l} + E_{2k-1})\\
S_{2l-1, 2k} S_{2r-1, 2k} =
(-\lambda_{l}J_{2l,2k} - \lambda_{k}J_{2l-1, 2k-1})
(-\lambda_{r}J_{2r,2k} - \lambda_{k}J_{2r-1, 2k-1})
\\
= -\lambda_{l}^2(E_{2l} + E_{2k}) -\lambda_{k}^{2}(E_{2l-1} + E_{2k-1})\\
S_{2l-1, 2k-1} S_{2r, 2k-1} =
(\lambda_{k} J_{2l-1, 2k} - \lambda_{l}J_{2l, 2k-1})
(\lambda_{k}J_{2r, 2k} + \lambda_{r}J_{2r-1, 2k-1})
\\
= -\lambda_{k}^{2} E_{2l-1, 2l} -\lambda_{k} \lambda_{l} E_{2k, 2k-1} + \lambda_{k}\lambda_{l} E_{2k-1, 2k} + \lambda_{l}^2 E_{2l, 2l-1}\\
S_{2l-1, 2k} S_{2r, 2k} =
(-\lambda_{l}J_{2l,2k} - \lambda_{k}J_{2l-1, 2k-1})
(\lambda_{r} J_{2r-1,2k} - \lambda_{k}J_{2r, 2k-1})
\\
= \lambda_{l}^2 E_{2l, 2l-1} - \lambda_{l} \lambda_{k} E_{2k, 2k-1} + \lambda_{k}\lambda_{l} E_{2k-1, 2k} - \lambda_{k}^{2} E_{2l-1, 2l}\\
S_{2l, 2k-1} S_{2r-1, 2k-1} =
(\lambda_{k}J_{2l, 2k} + \lambda_{l}J_{2l-1, 2k-1})
\lambda_{k} J_{2r-1, 2k} - \lambda_{r}J_{2r, 2k-1}
\\
= -\lambda_{k}^{2} E_{2l, 2l-1} + \lambda_{k}\lambda_{l} E_{2k, 2k-1} - \lambda_{l}\lambda_{k}E_{2k-1, 2k} + \lambda_{l}^2E_{2l-1, 2l}\\
S_{2l, 2k} S_{2r-1, 2k} =
(\lambda_{l} J_{2l-1,2k} - \lambda_{k}J_{2l, 2k-1})
(-\lambda_{r}J_{2r,2k} - \lambda_{k}J_{2r-1, 2k-1})
\\
=  \lambda_{l}^2E_{2l-1, 2l} + \lambda_{l} \lambda_{k} E_{2k, 2k-1} - \lambda_{k}\lambda_{l} E_{2k-1, 2k} - \lambda_{k}^{2}E_{2l, 2l-1}\\
S_{2l, 2k-1} S_{2r, 2k-1} =
(\lambda_{k}J_{2l, 2k} + \lambda_{l}J_{2l-1, 2k-1})
(\lambda_{k}J_{2r, 2k} + \lambda_{r}J_{2r-1, 2k-1})
\\
= -\lambda_{k}^{2}(E_{2k} + E_{2l}) - \lambda_{l}^2 (E_{2k-1} + E_{2l-1})\\
S_{2l, 2k} S_{2r, 2k} =
(\lambda_{l} J_{2l-1,2k} - \lambda_{k}J_{2l, 2k-1})
(\lambda_{r} J_{2r-1,2k} - \lambda_{k}J_{2r, 2k-1})
\\
= -\lambda_{l}^2(E_{2k} + E_{2l-1}) -\lambda_{k}^{2}(E_{2l} + E_{2k-1})
.
\end{align*}
This simplifies the sum into
\begin{equation*}
\begin{split}
-
\sum_{l \neq k}
W_{2l-1}^2 [
\lambda_{k}^{2}(E_{2l-1} + E_{2k}) + \lambda_{l}^{2}(E_{2l} + E_{2k-1})
+ \lambda_{l}^2(E_{2l} + E_{2k}) + \lambda_{k}^{2}(E_{2l-1} + E_{2k-1})
] \\
2W_{2l-1}W_{2l} [
\lambda_{k}^{2} E_{2l-1, 2l} + \lambda_{k} \lambda_{l} E_{2k, 2k-1} - \lambda_{k}\lambda_{l} E_{2k-1, 2k} - \lambda_{l}^2 E_{2l, 2l-1}
] \\
2W_{2l} W_{2l-1} [
\lambda_{k}^{2} E_{2l, 2l-1} - \lambda_{k}\lambda_{l} E_{2k, 2k-1} + \lambda_{l}\lambda_{k}E_{2k-1, 2k} - \lambda_{l}^2E_{2l-1, 2l}
] \\
W_{2l}^2  [
\lambda_{k}^{2}(E_{2k} + E_{2l}) + \lambda_{l}^2 (E_{2k-1} + E_{2l-1})
+\lambda_{l}^2(E_{2k} + E_{2l-1}) +\lambda_{k}^{2}(E_{2l} + E_{2k-1})
].
\end{split}
\end{equation*}
Collect coefficients of \(\lambda\),
\begin{equation*}
\begin{split}
\lambda_{k}^{2} \to \{ 
W_{2l-1}^{2} (2 W_{2l-1}^{2} + R_{k}^{2})
+ 2W_{2l-1}^{2} W_{2l}^{2}
+ 2W_{2l}^{2}W_{2l-1}^{2}
+ W_{2l}^{2} (R_{k}^{2} + 2 W_{2l}^{2})
\}
= R_{k}^{2} R_{l}^{2} + 2R_{l}^{4} \\
\lambda_{l}^{2} \to \{ 
W_{2l-1}^{2} (2W_{2l}^{2} + R_{k}^{2})
- 2 W_{2l-1}^{2}W_{2l}^{2}
- 2 W_{2l-1}^{2}W_{2l}^{2}
+ W_{2l}^{2} (R_{k}^{2} + 2 W_{2l-1}^{2})
\}
= R_{k}^{2} R_{l}^{2} + 0 \\
\lambda_{l} \lambda_{k} \to \{ 0 \}
.
\end{split}
\end{equation*}
The contribution is
\[
-
\sum_{l \neq k} \frac{\lambda_{k}^{2}R_{l}^{2}(R_{k}^{2} + 2R_{l}^{2}) + \lambda_{l}^{2}R_{k}^{2}R_{l}^{2}}{(\lambda_{k}^{2} - \lambda_{l}^{2})^{2}}
.
\]

\textbf{2.3 Contribution of the projection's drift term, Part 2}
By It\^o formula,
\begin{align*}
dh_{2k-1}
&=
\frac{1}{2}\sum_{\alpha, \beta} \partial_{\alpha, \beta} q_{2k-1} d\ip{\bA_{\alpha}}{\bA_{\beta}}
=
\frac{dt}{8} \sum_{i, j, x} B_{i} B_{j}
\sum_{a} q_{a} [T_{a, 2k-1}]_{ir, jr} \\
dh_{2k}
&=
\frac{1}{2}
\sum_{\alpha, \beta} \partial_{\alpha, \beta} q_{2k} d\ip{\bA_{\alpha}}{\bA_{\beta}}
=
\frac{dt}{8} \sum_{i, j, r} B_{i} B_{j}
\sum_{a} q_{a} [T_{a, 2k}]_{ir, jr}
.
\end{align*}
We are interested in $B^{\top} dH_k B$ which simplifies into
$$
4 B^{\top} dH_k B
=
\sum_{i,j,r} B_{i} B_{j} 
\sum_{a}
W_{2k-1}W_{a} [T_{a, 2k-1}]_{ir, jr}
+ W_{2k}W_{a} [T_{a, 2k}]_{ir, jr}
$$
We show
\begin{equation*}
\begin{split}
4B^{\top} dH_k B
=
R_{k}^{2}B^{\top} S_{k}^{2} B + \sum_{l \neq k}
\frac{- (\lambda_{k}^{2} + \lambda_{l}^{2}) R_{k}^{2} R_{l}^{2}  - 2\lambda_{l}^{2} R_{k}^{4}}{D_{kl}^{2}} \\
+ 
2\sum_{l \neq k} \frac{2\lambda_{k}^{2}R_{l}^{2}R_{k}^{2}}{D_{kl}} \left(\sum_{r\neq l,k}\frac{1}{D_{kr}}\right)
+ \frac{\lambda_{k}(R_{k}^{2} B^{\top} P_{l}S_{k} B  - R_{l}^{2} B^{\top}P_{k} S_{k} B)}{D_{kl}^2}
.
\end{split}
\end{equation*}

Split into two cases: $a \in \{ 2k-1, 2k \}$ and $a \in \set{2l-1, 2l}$ for \(l \neq k\). The contributions stated exclude the factor of $1/4$.

\textbf{Case 1}
For $a \in \{ 2k-1, 2k \}$, the contribution is
$$
R_{k}^{2}B^{\top} S_{k}^{2} B + \sum_{l \neq k}
\frac{- (\lambda_{k}^{2} + \lambda_{l}^{2}) R_{k}^{2} R_{l}^{2}  - 2\lambda_{l}^{2} R_{k}^{4}}{D_{kl}^{2}}
.
$$
To show this, we split into two terms.
\begin{equation*}
\begin{split}
W_{2k-1}^{2} [T_{2k-1, 2k-1}]_{ir, jr}
+ W_{2k}W_{2k-1} [T_{2k-1, 2k}]_{ir, jr} \\
W_{2k-1}W_{2k} [T_{2k, 2k-1}]_{ir, jr}
+ W_{2k}^{2} [T_{2k, 2k}]_{ir, jr} \\
=
W_{2k-1}^{2} [T_{2k-1, 2k-1}]_{ir, jr}
+ W_{2k}^{2} [T_{2k, 2k}]_{ir, jr} \\
+ W_{2k-1} W_{2k} ([T_{2k-1, 2k}]_{ir, jr} + [T_{2k, 2k-1}]_{ir, jr})
\end{split}
\end{equation*}

\textbf{Case 1.1}
For the first term, apply \eqref{eq: Taa expansion}. The sum over \(r\) yields matrix multiplication. To simplify notation, define \([T]_{ij} := \sum_{r} [T]_{ir, jr}\) and interpret $T$ as a matrix. Hence,
$$
T_{2k-1, 2k-1} = S_{k}^{2} + \sum_{l \neq k} S_{2l-1, 2k-1}^{2} + S_{2l, 2k-1}^{2},\quad
T_{2k, 2k} = S_{k}^{2} + \sum_{l \neq k} S_{2l-1, 2k}^{2} + S_{2l, 2k}^{2}
.
$$
First $S_k \to (W_{2k-1}^2 + W_{2k}^2) B^{\top} S_k^2 B$. For \(l \neq k\), expand $S^2$ via \eqref{eq: S_a,b squared} and sum over $(i, j)$.
\begin{equation*}
\begin{split}
D_{kl}^2 B^{\top} (W_{2k-1}^2(S_{2l-1, 2k-1}^{2} + S_{2l, 2k-1}^{2}) + W_{2k}^2 (S_{2l-1, 2k}^{2} + S_{2l, 2k}^{2})) B \\
= - (\lambda_{k}^{2} + \lambda_{l}^{2}) R_{k}^{2} R_{l}^{2} - 2\lambda_{l}^{2}W_{2k-1}^{4} - 2\lambda_{l}^{2} W_{k}^{4} - 4\lambda_{k}^{2} W_{2k-1}^{2} W_{2k}^{2} \\
=
W_{2k-1}^{2} \{ 
-\lambda_{k}^{2}(W_{2l-1}^{2} + W_{2k}^{2}) - \lambda_{l}^{2}(W_{2l}^{2} + W_{2k-1}^{2}) \\
-\lambda_{k}^{2}(W_{2l}^{2} + W_{2k}^{2}) - \lambda_{l}^{2}(W_{2l-1}^{2} + W_{2k-1}^{2})
\} \\
+ W_{2k}^{2} \{ 
-\lambda_{l}^{2}(W_{2l}^{2} + W_{2k}^{2}) - \lambda_{k}^{2}(W_{2l-1}^{2} + W_{2k-1}^{2}) \\
-\lambda_{l}^{2}(W_{2l-1}^{2} + W_{2k}^{2}) - \lambda_{k}^{2}(W_{2l}^{2} + W_{2k-1}^{2})
\}
.
\end{split}
\end{equation*}
Therefore, the contribution is
\begin{equation*}
\begin{split}
R_{k}^{2}B^{\top} S_{k}^{2} B 
+ \sum_{l \neq k}
\frac{- (\lambda_{k}^{2} + \lambda_{l}^{2}) R_{k}^{2} R_{l}^{2} - 2\lambda_{l}^{2}R_k^4 - 4(\lambda_{k}^{2} - \lambda_l^2)W_{2k-1}^{2} W_{2k}^{2}}{D_{kl}^{2}}
.
\end{split}
\end{equation*}

\textbf{Case 1.2} By \eqref{eq: Symmetrised Tk expansion},
$$
2 \lambda_{k} ([T_{2k-1, 2k}]_{\alpha, \beta} + [T_{2k, 2k-1}]_{\alpha, \beta})
=
[U_{2k, 2k}]_{\beta, \alpha}
+ [U_{2k, 2k}]_{\alpha, \beta}
- [U_{2k-1, 2k-1}]_{\beta, \alpha} 
- [U_{2k-1, 2k-1}]_{\alpha, \beta}
.
$$
As we are summing over equal weights ($B_i, B_j$) with respect to $(i, j)$, \(\alpha = (i, r)\) and \(\beta = (j, r)\) can be swapped, which means its equivalent to consider
$$
\lambda_{k} ([T_{2k-1, 2k}]_{\alpha, \beta} + [T_{2k, 2k-1}]_{\alpha, \beta})
=
 [U_{2k, 2k}]_{\alpha, \beta}
- [U_{2k-1, 2k-1}]_{\alpha, \beta}
.
$$
Take the sum over $r$ and interpret $[U]_{ij} := \sum_{r} [U]_{ir, jr}$ as a matrix.
\begin{equation*}
\begin{split}
U_{2k-1, 2k-1} &= -J_{k}^{\top} S_{k} + \sum_{l \neq k} J_{2k-1, 2l-1}^{\top} S_{2l-1, 2k-1} + J_{2k-1, 2l}^{\top}S_{2l, 2k-1} \\
U_{2k, 2k} &= -J_{k}^{\top} S_{k} + \sum_{l \neq k} J_{2k, 2l-1}^{\top} S_{2l-1, 2k} + J_{2k, 2l}^{\top} S_{2l, 2k}
\end{split}
\end{equation*}
When considering their difference the $S_{k}$ term vanishes. For \(l \neq k\), use \eqref{eq: JS multiplication} to show the contribution is
\begin{equation*}
\begin{split}
D_{kl} B^{\top} (J_{2k, 2l-1}^{\top} S_{2l-1, 2k} + J_{2k, 2l}^{\top} S_{2l, 2k}
- J_{2k-1, 2l-1}^{\top} S_{2l-1, 2k-1} - J_{2k-1, 2l}^{\top}S_{2l, 2k-1})B
\\= 4 \lambda_{k} W_{2k-1} W_{2k} = 
\{ \lambda_{l} W_{2l-1} W_{2l} + \lambda_{k} W_{2k}W_{2k-1}  
-\lambda_{l} W_{2l-1}W_{2l} + \lambda_{k} W_{2k} W_{2k-1}
\} \\
- \{  -\lambda_{k} W_{2k-1}W_{2k} + \lambda_{l} W_{2l-1} W_{2l}
-\lambda_{k} W_{2k-1}W_{2k} -\lambda_{l} W_{2l}W_{2l-1}
\}
.
\end{split}
\end{equation*}
Therefore, the contribution is
\begin{equation*}
W_{2k-1}W_{2k} \sum_{i,j,r} B_i B_j \lambda_{k} [T_{2k-1, 2k} + T_{2k, 2k-1}]_{ir, jr}
=
\sum_{l \neq k} \frac{4W_{2k-1}^{2}W_{2k}^{2}}{D_{kl}}
\end{equation*}
which cancels with the final term of the previous case.

\textbf{Case 2} For \(l \neq k \), we show
\begin{equation*}
\begin{split}
\sum_{ijr} B_i B_j(W_{2k-1}W_{2l-1} [T_{2l-1, 2k-1}]_{ir, jr}
+ W_{2k}W_{2l-1} [T_{2l-1, 2k}]_{ir, jr} \\
+ W_{2k-1}W_{2l} [T_{2l, 2k-1}]_{ir, jr}
+ W_{2k}W_{2l} [T_{2l, 2k}]_{ir, jr}) \\
=
\sum_{l \neq k} \frac{4\lambda_{k}^{2}R_{l}^{2}R_{k}^{2}}{D_{kl}} \left(\sum_{r\neq l,k}\frac{1}{D_{kr}}\right)
+ \frac{2\lambda_{k}(R_{k}^{2} B^{\top} P_{l}S_{k} B  - R_{l}^{2} B^{\top}P_{k} S_{k} B)}{D_{kl}^2}
\end{split}
\end{equation*}
The \(T\) terms can be expressed through \(U, J, S\) via \eqref{eq: T_a,b identity}. By symmetry swap $\alpha = (i, r), \beta =(j, r)$ to simplify the linear systems. E.g.
\begin{align*}
\lambda_{k} [T_{2l-1, 2k}]_{\alpha, \beta}
+ \lambda_{l} [T_{2l, 2k-1}]_{\alpha, \beta}
&=
- 2[J_{k}]_{\alpha} [S_{2l-1, 2k}]_{\beta} 
- 2[U_{2l-1, 2k-1}]_{\alpha, \beta}\\
- \lambda_{l} [T_{2l-1, 2k}]_{\alpha, \beta}
- \lambda_{k}[T_{2l, 2k-1}]_{\alpha, \beta}
&=
2[J_{k}]_{\alpha}[S_{2l, 2k-1}]_{\beta}
- 2[U_{2l, 2k}]_{\alpha, \beta}
.
\end{align*}
Inverting yields
\begin{align*}
\frac{D_{kl}}{2} [T_{2l-1, 2k-1}]_{\alpha, \beta} = 
- \lambda_{l} ([J_{k}]_{\alpha} [S_{2l, 2k}]_{\beta}
+ [U_{2l, 2k-1}]_{\alpha, \beta}) \\
- \lambda_{k} ([J_{k}]_{\alpha} [S_{2l-1, 2k-1}]_{\beta}
- [U_{2l-1, 2k}]_{\alpha, \beta}) \\
\frac{D_{kl}}{2} [T_{2l, 2k-1}]_{\alpha, \beta} =
\lambda_{l} ([J_{k}]_{\alpha} [S_{2l-1, 2k}]_{\beta}
+ [U_{2l-1, 2k-1}]_{\alpha, \beta}) \\
-\lambda_{k} ([J_{k}]_{\alpha}[S_{2l, 2k-1}]_{\beta}
- [U_{2l, 2k}]_{\alpha, \beta}) \\
\frac{D_{kl}}{2} [T_{2l-1, 2k}]_{\alpha, \beta} =
-\lambda_{k} ([J_{k}]_{\alpha} [S_{2l-1, 2k}]_{\beta}
+ [U_{2l-1, 2k-1}]_{\alpha, \beta}) \\
+ \lambda_{l} ([J_{k}]_{\alpha}[S_{2l, 2k-1}]_{\beta}
- [U_{2l, 2k}]_{\alpha, \beta}) \\
\frac{D_{kl}}{2} [T_{2l, 2k}]_{\alpha, \beta} =
-\lambda_{k} ([J_{k}]_{\alpha} [S_{2l, 2k}]_{\beta}
+ [U_{2l, 2k-1}]_{\alpha, \beta}) \\
- \lambda_{l}([J_{k}]_{\alpha} [S_{2l-1, 2k-1}]_{\beta}
- [U_{2l-1, 2k}]_{\alpha, \beta})
.
\end{align*}
Hence, with a factor of \(D_{kl} / 2\), the sum of \(T\) weighted by \(W\) is
\begin{equation}
\begin{split}
([J_{k}]_{\alpha} [S_{2l-1, 2k-1}]_{\beta}) \to \{ 
-\lambda_{k} W_{2k-1}W_{2l-1}
-\lambda_{l} W_{2k}W_{2l}
\}
\\
([J_{k}]_{\alpha} [S_{2l, 2k-1}]_{\beta}) \to \{ 
\lambda_{l} W_{2k}W_{2l-1}
-\lambda_{k} W_{2k-1}W_{2l}
\}\\
([J_{k}]_{\alpha} [S_{2l-1, 2k}]_{\beta}) \to \{ 
-\lambda_{k} W_{2k}W_{2l-1}
+ \lambda_{l} W_{2k-1}W_{2l}
\}\\
([J_{k}]_{\alpha} [S_{2l, 2k}]_{\beta}) \to \{ 
-\lambda_{l} W_{2k-1}W_{2l-1}
-\lambda_{k} W_{2k}W_{2l}
\}\\
\\
(U_{2l-1, 2k-1}) \to \{ 
-\lambda_{k} W_{2k}W_{2l-1}
+\lambda_{l} W_{2k-1}W_{2l}
\}\\
(U_{2l, 2k-1}) \to \{ 
-\lambda_{l} W_{2k-1}W_{2l-1}
-\lambda_{k} W_{2k}W_{2l}
\}\\
(U_{2l-1, 2k}) \to \{ 
+\lambda_{k} W_{2k-1}W_{2l-1}
+ \lambda_{l} W_{2k}W_{2l}
\}\\
(U_{2l, 2k}) \to  \{ 
-\lambda_{l} W_{2k}W_{2l-1}
+ \lambda_{k} W_{2k-1}W_{2l}
\}
.
\end{split}
\end{equation}
Take the sum over $r$ and interpret $U$ as matrix as in the previous case.
\begin{equation*}
\begin{split}
[U_{2l-1,2k-1}]  
&  
- J_{2l-1, 2k}^{\top}S_{k} + J_{l}^{\top} S_{2l, 2k-1}
+ \sum_{r \neq l, k}
J_{2l-1, 2r-1}^{\top} S_{2r-1, 2k-1} 
+ J_{2l-1, 2r}^{\top} S_{2r, 2k-1} \\
[U_{2l,2k-1}]  
&=  
- J_{2l, 2k}^{\top}S_{k} - J_{l}^{\top} S_{2l-1, 2k-1}
+ \sum_{r \neq l, k}
J_{2l, 2r-1}^{\top} S_{2r-1, 2k-1} 
+ J_{2l, 2r}^{\top} S_{2r, 2k-1} \\
[U_{2l-1,2k}]  
&=  
J_{2l-1, 2k-1}^{\top}S_{k} + J_{l}^{\top} S_{2l, 2k}
+ \sum_{r \neq l, k}
J_{2l-1, 2r-1}^{\top} S_{2r-1, 2k} 
+ J_{2l-1, 2r}^{\top} S_{2r, 2k} \\
[U_{2l,2k}]  
&=  
J_{2l, 2k-1}^{\top}S_{k} - J_{l}^{\top} S_{2l-1, 2k}
+ \sum_{r \neq l, k}
J_{2l, 2r-1}^{\top} S_{2r-1, 2k} 
+ J_{2l, 2r}^{\top} S_{2r, 2k}
\end{split}
\end{equation*}
Let $H_{kl} = \sum_{r\neq l,k}\frac{1}{D_{kr}}$. By expanding with \eqref{eq: JS multiplication},
\begin{equation*}
\begin{split}
[U_{2l-1,2k-1}]  
&=  
- J_{2l-1, 2k}^{\top}S_{k}  
+  
\frac{-\lambda_k E_{2l-1,2k}+\lambda_l E_{2l,2k-1}}{D_{kl}}  
-  
2 \lambda_k E_{2l-1,2k} H_{kl} \\
[U_{2l,2k-1}]  
&=  
- J_{2l, 2k}^{\top}S_k  
+  
\frac{-\lambda_k E_{2l,2k}-\lambda_l E_{2l-1,2k-1}}{D_{kl}}  
-  
2 \lambda_k E_{2l,2k}H_{kl} \\
[U_{2l-1,2k}]  
&=  
J_{2l-1, 2k-1}^{\top}S_k  
+  
\frac{\lambda_l E_{2l,2k}+\lambda_k E_{2l-1,2k-1}}{D_{kl}}  
+  
2\lambda_kE_{2l-1,2k-1}H_{kl} \\
[U_{2l,2k}]  
&=  
J_{2l, 2k-1}^{\top}S_k  
+  
\frac{-\lambda_l E_{2l-1,2k}+\lambda_k E_{2l,2k-1}}{D_{kl}}  
+  
2\lambda_k E_{2l,2k-1} H_{kl}
\end{split}
\end{equation*}
Combined with \eqref{eq: JkS multiplication} and $B^{\top}E_{ab}B = W_a W_b$, we compute the sum. The coefficient of $H_{kl}$ is $2R_{l}^{2} R_{k}^{2} \lambda_{k}^{2}$. The coefficient of $S_k$ is
\begin{equation*}
\begin{split}
(-\lambda_{k} W_{2k}W_{2l-1}
+\lambda_{l} W_{2k-1}W_{2l})(-J_{2l-1, 2k}^{\top}) \\
+ (-\lambda_{l} W_{2k-1}W_{2l-1}
-\lambda_{k} W_{2k}W_{2l}) (-J_{2l, 2k}^{\top}) \\
+ (+\lambda_{k} W_{2k-1}W_{2l-1}
+ \lambda_{l} W_{2k}W_{2l}) (J_{2l-1, 2k-1}^{\top}) \\
+ (-\lambda_{l} W_{2k}W_{2l-1}
+ \lambda_{k} W_{2k-1}W_{2l}) (J_{2l, 2k-1}^{\top})
.
\end{split}
\end{equation*}
Summing over $i$, this is equivalent to summing over \(\lambda_{k} (R_{k}^{2} P_{l}^{\top} - R_{l}^{2} P_{k}^{\top}) / D_{kl}\). The \(U, JS\) contributions are of opposite signs and cancel ($\pm (\lambda_{k}^{2} + \lambda_{l}^{2}) R_{k}^{2} R_{l}^{2}$ respectively).
\end{proof}

\subsection{Well-posedness of the SDE}
\label{appendix:Well-posedness of the SDE}
We now show that provided the initial conditions are that $\lambda_k$ are distinct and $u_k > 0$, then the SDE is well-posed. For notation, we consider the Weyl Chamber
$$\lambda \in \mathcal{W}_n = \set{\lambda \in \R^{n}: |\lambda_1| \ge \cdots \ge |\lambda_n| \ge 0 }.$$
And we define $\mathcal{D}_n$ be the domain of $(\lambda, u)$ such that $\lambda \in \mathcal{W}_n$ and $u_k \ge 0$. Its interior $\mathcal{D}_n^{\rm{o}}$ corresponds to $\lambda \in \mathcal{W}^\mathrm{o}_n$ and $u_k > 0$. Recall $\lambda \in \mathcal{W}^\mathrm{o}_n$ means the $\lambda$ are distinct.
\begin{lemma}[Well-posedness of the SDE with interior initial conditions]
\label{lemma:SDE well-posed, interior IC}
    If \((\lambda, u)(0) \in \mathcal{D}^{\mathrm{o}}_n\), there exists a unique strong solution to the SDE. Furthermore, \(\forall t > 0, (\lambda, u)(t) \in \mathcal{D}^{\mathrm{o}}_n\).
\end{lemma}
\begin{proof}[Proof of Lemma \ref{lemma:SDE well-posed, interior IC}]
We follow the approach taken in \cite[Theorem 4.3.2]{andersonIntroductionRandomMatrices2009} to demonstrate the well-posedness of Dyson Brownian motion. Consider the following stopping times.
\begin{equation}\label{eq:Interior Stopping Times}
\begin{split}
    \tau_{M}^{1} &= \inf \left\{ t > 0: \max_{k \in [n]} u_{k} > M   \right\} \\
    \tau_{M'}^{2} &= \inf \left\{ t > 0: \max_{k \in [n]} s_{k} > M'   \right\} \\
    \tau^{3}_{R} &= \inf \left\{  t > 0: \max_{k, \neq l}\frac{1}{|s_{k} - s_{l}|} \ge R  \right\} \\
    \tau_{R'}^{4} &= \inf \left\{  t > 0: \min_{k} u_{k}(t) < \frac{1}{R'}  \right\} 
\end{split}
\end{equation}
As the initial conditions lie in the interior of the domain, these are all non-zero for sufficiently large thresholds.
Up to time $t \le \tau_{M}^{1} \land \tau_{M'}^{2} \land \tau_{R}^{3} \land \tau_{R'}^{4}$, the coefficients of the SDE \eqref{eq:SDE Singular values} are globally Lipschitz continuous and hence a unique pathwise solution exists. This follows as the stopping times ensure the square roots $\sqrt{ u_{k} }$ are bounded away from 0 and the singularities $\frac{1}{s_{k} - s_{l}}$ are bounded above. Then we show that all the stopping times converge to \(\infty\) as \(M, M', R, R' \to \infty\) and hence a solution exists for all time and remains in the interior.

The argument to show that all stopping times considered in \eqref{eq:Interior Stopping Times} converge to $\infty$ is standard: considering a suitable quantity to bound $\Prob{\tau_M \le t}$ above and show it decays to 0. By monotonicity of $\tau_M$ in $M$, a Borel-Cantelli argument implies $\tau_M \to \infty$.

Let $s_k = \lambda_k^2$. By product rule,
\begin{align*}
d s_{k}(t) &= \mathrm{sign}(\lambda_{k})\sqrt{s_{k} u_{k} } d\theta_{k}(t) + \left\{\frac{u_{k}}{4}+  \frac{s_{k}}{2}\sum_{l \neq k} \frac{u_{k} + u_{l}}{s_{k} - s_{l}}  \right\} dt
.
\end{align*}
Define $u = \sum_k u_k, s = \sum s_k$. As \(s_k, u_k \ge 0\), we have $0 \le s_k \le s, 0 \le u_k \le u$.  Part 1 shows $u_k$ cannot explode in finite time as $0 \le u_k \le u$ and $u$ is squared Bessel process. Part 2 shows that $s_k$ cannot explode as $0 \le s_k \le s$ and the drift of $s$ has no singularities and is controlled by $u$. Part 3 shows that $s_k$ cannot collide by considering a Lyapunov function which gets large if some $|s_k - s_l|$ is small. We show that it cannot explode as its drift is bounded above by $u$. Part $4$ shows $u_k$ cannot hit 0 by bounding drift of $- \log u_k$.

\textbf{Part 1: Upper bound on $u_k$} From Lemma \ref{lemma: SDE for u}, $u(t)$ is a $2n$-dimensional squared Bessel process.
$$
du(t) = 2\sqrt{u} d\gamma(t) + 2n \; dt
$$
Define $\tau_{M}^{u} = \inf \{ t > 0: u(t) > M \}$ and consider the SDE for $u$ up to $\tau_{M}^{u}$. Integrate and take expectations to deduce the following bound on $\Prob{\tau_{M}^{u} \le t}$.
\begin{align*}
M \Prob{\tau^{u}_{M} \le t} \le \Expectation{u(t) \Indicator{\tau^{u}_{M} \le t}} &= \Expectation{u(t \land \tau^{u}_{M})}
\le  u(0) + 2nt
\implies
\Prob{\tau_{M}^{u} \le t} \le \frac{u(0) + 2n t}{M}
\end{align*}
As $\tau_{M}^{u}$ is increasing in $M$ it suffices to show that $\Prob{\tau_{M}^{u} \le t} \to 0$ as $M \to \infty$ to conclude $\tau_{M}^{u} \to \infty$. For fixed $t$, we can take any subsequence $M_{i} \to \infty$ such that
$$
\sum_{i} \Prob{\tau^{u}_{M_{i}} \le t}
\le \sum_{i} \frac{u(0) + 2n t}{M_i}
< \infty \implies \Prob{\lim_{ i \to \infty } \tau_{M_{i}^{u}} \le t} = 0
$$
as $t$ was arbitrary $\lim_{ i \to \infty } \tau^{u}_{M_{i}} \to \infty$ and as $\tau_{M}$ is monotonic $\tau_{M}^{u} \to \infty$. A valid choice is $M_{i} = i^{2}$.

As $0 \le u_{k} \le u$ then $\tau_{M}^{1} \ge \tau_{M}^{u}$. Hence, $\tau_{M}^{1} \to \infty$.

\textbf{Part 2: Upper bound on $s_k$} The SDE for $s$ is
\begin{align*}
ds(t) &=  \sum_{k=1}^{n} \mathrm{sign}(\lambda_{k})\sqrt{ s_{k} u_{k} } \; d\theta_{k} + dt \sum_{k=1}^{n} \frac{u_{k}}{4}+  \frac{s_{k}}{2}\sum_{l \neq k} \frac{u_{k} + u_{l}}{s_{k} - s_{l}} \\
&=  \sum_{k=1}^{n} \mathrm{sign}(\lambda_{k})\sqrt{ s_{k} u_{k} } \; d\theta_{k} + \frac{2n-1}{4} u \;dt
\end{align*}
Considering the SDE up to $\tau^{2}_{M'} \land  \tau_{M}^{u}$, we deduce
\begin{align*}
nM' \Prob{\tau^{2}_{M'} \le t \land \tau_{M}^{u}} &\le \Expectation{s(t) \Indicator{\tau^{2}_{M'} \le t \land \tau_{M}^{u}}}
\le s(0) + \frac{2n-1}{4} M t
\end{align*}
To bound for $\Prob{\tau^{2}_{M'} \le t}$, use that $\Prob{X \le t} = \Prob{X \le t, X \le Y} + \Prob{X \le t, X \ge Y}$. Hence $ \Prob{X \le t} \le \Prob{X \le t, X \le Y} + \Prob{Y \le t}$. Therefore,
\[
\Prob{\tau^{2}_{M'} \le t} \le \frac{ s(0) + \frac{2n-1}{4} M t}{n M'} + \Prob{\tau_{M}^{u} \le t} \le \frac{ s(0) + \frac{2n-1}{4} M t}{n M'} + \frac{u(0) + 2n t}{M}
.
\]
To show $\tau_{M'}^{2} \to \infty$, choose $M = \sqrt{ M' }$ and observe
$$
\Prob{\tau^{2}_{M'} \le t} \lesssim \frac{1}{\sqrt{ M' }} \to 0
.
$$

\textbf{Part 3: Collisions of $s_k$}
Consider the Lyapunov function on $s_k$
$$
f(s) = \frac{1}{n}\sum_{k} s_{k} - \frac{1}{n^{2}} \sum_{l \neq k} \log |s_{k} - s_{l}|
$$
which has the property that if $s_{k}, s_{l}$ are close then $f$ must be big. I.e.
$$
-\frac{1}{n^{2}} \log |s_{k} - s_{l}| \le f(s) + 4,\quad f(s) +4 \ge 0
.
$$
In particular, $\tau_{R}^{f} \le \tau_{R}^{3}$ where
$$
\tau_{R}^{f} = \inf \left\{  t > 0: f(s) > \frac{\log R}{n^{2}} - 4  \right\}
.
$$
The drift term of $df$ can be bounded above by $u/2$.
$$
\frac{2n - 1}{4n} u
-
\frac{1}{2n^{2}} \sum_{k} u_{k} \left\{  s_{k} \left( \sum_{l \neq k}  \frac{1}{s_{k} - s_{l}}\right)^{2} + \sum_{l \neq k} \frac{s_{l}}{(s_{k} - s_{l})^{2}} \right\} \le \frac{u}{2}
$$
Applying the same argument as before, let $R' = \frac{\log R}{n^{2}} - 4$ and deduce
\begin{align*}
(R'+4)\Prob{\tau_{R}^{f} \le t \land \tau_{M}^{u}}
&\le \Expectation{(f(s(t)) + 4) \Indicator{\tau_{R}^{f} \le t \land \tau_{M}^{u}}}
\le f(s(0)) + 4 + \frac{M}{2}t
.
\end{align*}
Hence,
$$
\Prob{\tau_{R}^{f} \le t} \le
\frac{f(s(0)) + 4 + \frac{M}{2}t}{R' + 4} 
+ \frac{u(0) + 2n t}{M}
.
$$
Taking $M = \sqrt{ R' }$ shows that $\tau_{R}^{f} \to \infty$ and hence $\tau_{R}^{3} \to \infty$.

\textbf{Part 4: Lower bound on $u_k$}
For each $u_{k}$ consider the function $g(u_{k}) = - \log u_{k}$. The drift term of $g(u_{k})$ is
$$
\sum_{l, r \neq k, l \neq r} \frac{-s_{k} u_{l}}{(s_{k} - s_{l})(s_{k} - s_{r})} + \frac{1}{2} \sum_{l \neq k} \frac{-s_{k}u_{l} + s_{l} (u_{k} + u_{l})}{(s_{k} - s_{l})^{2}}
.
$$
The previous parts show the drift cannot explode and $-\log u_k < \infty \implies u_k > 0$. Formally, define
$$
\tau_{R}^{4, u_{k}} = \inf \left\{  t > 0: - \log u_{k} > \log R  \right\}
.
$$
Conditional on $\tau_{R}^{4} \le \tau_{M}^{1} \land \tau^{2}_{M'} \land \tau_{\widetilde{R}}^{3}$, the drift term is bounded above by
$$
\widetilde{C} := \left( n(n-1)(n-2) + n(n-1) \frac{3}{2} \right) M M' \widetilde{R}^{2}
$$
Let $R' = \log R$ and WLOG $R > 1$. As $u_{k} \le M, g(u_{k}) + \log M \ge 0$. We form the bound
\begin{align*}
(R' + \log M) \Prob{\tau_{R}^{4, u_{k}} \le t \land \tau_{M}^{1} \land \tau^{2}_{M} \land \tau_{\widetilde{R}}^{3}}
&\le \Expectation{(g(u_{k}(t)) + \log M) \Indicator{\tau_{R}^{4, u_{k}} \le t \land \tau_{M}^{1} \land \tau^{2}_{M} \land \tau_{\widetilde{R}}^{3}}} \\
&\le g(u_{k}(0)) + \log M + \widetilde{C}t
.
\end{align*}
Apply a union bound to get a bound for $\Prob{\tau_{R}^{4, u_{k}} \le t}$.
\begin{align*}
\Prob{\tau_{R}^{4, u_{k}} \le t}
&\le \frac{g(u_{k}(0)) + \log M + \widetilde{C}t}{ R'}
+ \Prob{\tau_{M}^{1} \le t}
+ \Prob{\tau_{M'}^{2} \le t}
+ \Prob{\tau_{\widetilde{R}}^{3} \le t} \\
&\le \frac{g(u_{k}(0)) + \log M + \widetilde{C}t}{ R'}
+ \Prob{\tau_{M}^{u} \le t}
+ \Prob{\tau_{M'}^{2} \le t}
+ \Prob{\tau_{\widetilde{R}}^{f} \le t}
\end{align*}
Up to constant factors, the sum is
$$
\lesssim \frac{\log M+ M M' \widetilde{R}^{2}}{R' + \log M}+ \frac{1}{M} + \frac{M}{M'} + \frac{M}{\log \widetilde{R}}
.
$$
It remains to show this decay to $0$ as $R' \to \infty$ where $M, M', \widetilde{R}$ are functions of $R'$. One choice is
$$
M = (\log R')^{1/3},\quad
M' = (R')^{1/2},\quad
\widetilde{R} = \exp((\log R')^{2/3})
.
$$
For any $\varepsilon >0$, $\widetilde{R} = o((R')^{\varepsilon})$, take $\varepsilon = 1/4$. $\log M$ is also harmless as it contributes a $\log \log$ term which cannot grow too large. Therefore the sum decays
$$
\lesssim \frac{o(R^{1/2 + 1/4})}{R'} + \frac{1}{\log(R')^{1/3}} + \frac{1}{(R')^{1/2}} + \frac{1}{\log(R')^{1/3}} \to 0
$$
and $\tau_{R}^{4, u_{k}} \to \infty$. Furthermore $\tau_{R}^{4} \ge \min_{k} \tau_{R}^{4, u_{k}}$, therefore $\tau_{R}^{4} \to \infty$ as well.

\end{proof}

\subsection{A PDE for the Joint Density of Singular Values and Squared Subspace Norms}
\label{appendix:Derivation of the PDE}
By the Kolmogorov forward equation, the SDE implies a PDE for the joint density of $p(\lambda, u)$. This is related to Corollary \ref{corollary:joint density of sigma, u} as by symmetry, $p(\sigma, u) = 2^n p(\lambda,u)$.

\begin{proposition}[PDE for the joint density]
\label{prop:Joint density PDE}
The joint density satisfies the following PDE.
\begin{equation}
\begin{split}
\partial_{t} p_{t}(\lambda, u)
&= \sum_{k=1}^{n}\frac{u_{k}}{8} \partial_{\lambda_{k}}^{2} p
+ \left( 2u_{k} + \frac{1}{2} X_{kk} \right) \partial_{u_{k}}^{2}p
+ Y_{kk} \partial_{\lambda_{k}} \partial_{u_{k}} p\\
&+  \sum_{l \neq k} Y_{kl} \partial_{u_{k}} \partial_{\lambda_{l}} p 
+ \sum_{l < k} X_{kl} \; \partial_{u_{k}} \partial_{u_{l}} p \\
&+ C\;p +\sum_{k=1}^{n} C_{\lambda_{k}} \partial_{\lambda_{k}}p  + C_{u_{k}} \partial_{u_{k}} p
\end{split}
\end{equation}
where
\begin{equation}
\begin{split}
C &= 
\frac{1}{4} \sum_{l \neq k} \frac{\lambda_{k}^{2}(u_{k} - 3 u_{l}) + \lambda_{l}^{2}(u_k + 5u_{l})}{D_{kl}^{2}}
+
\sum_{l, r \neq k , l \neq r} \frac{\lambda_{k}^{2}(u_{k} - u_{l})}{D_{kl} D_{kr}} \\
C_{\lambda_k} &=
\frac{\lambda_k}{4} \sum_{l \neq k} \frac{ u_{l} - 3u_{k}}{D_{kl}} \\
C_{u_k} &=
2
+ \frac{\lambda_{k}^{2}}{2}\left( \sum_{l \neq k} \frac{u_{l}}{D_{kl}}\right)^{2}
- 2 u_k \left( \sum_{l \neq k} \frac{u_{l}}{D_{kl}}\right)
- \frac{1}{2} \sum_{l \neq k} \frac{\lambda_{l}^{2} u_{k}^{2}}{D_{kl}^{2}}
- \sum_{l, r \neq k, l \neq r}
\frac{u_l u_k \lambda_k^2}{D_{kl} D_{kr}}
+
\frac{u_k u_l \lambda_l^2}{D_{kl} D_{lr}}
\end{split}
\end{equation}
and for $l \neq k$
\begin{equation}
\begin{split}
Y_{kk} &= \frac{u_{k}\lambda_{k}}{2} \sum_{l \neq k } \frac{ u_{l}}{D_{kl}} \\
Y_{kl} &= \frac{u_{k}}{2} \frac{\lambda_{l}u_{l} }{D_{kl}}\\
X_{kk} &=
\sum_{l \neq k} \frac{u_{k} u_{l}}{D_{kl}} \left[ \frac{\lambda_{k}^{2} u_{l} + \lambda_{l}^{2}u_{k}}{D_{kl}} + \lambda_{k}^{2}\sum_{r \neq l, k}\frac{u_{r}}{D_{kr}} \right] \\
X_{kl} &= u_{k} u_{l}\left[ \frac{-(\lambda_{k}^{2}u_{l} + \lambda_{l}^{2}u_{k})}{D_{kl}^{2}} + \sum_{r \neq l, k} \frac{2u_{r}\lambda_{r}^{2}}{D_{kr} D_{lr}} \right]
.
\end{split}
\end{equation}
\end{proposition}
\begin{proof}
The SDE can be written as
$$
d
\begin{bmatrix}
\lambda(t)  \\
u(t)
\end{bmatrix}
=
\mu
dt
+
S \;d \begin{bmatrix}
\theta(t) \\
\gamma(t)
\end{bmatrix}
$$
where
$$
\mu = \begin{bmatrix}
f \\
2 + g + h
\end{bmatrix},\quad
S =\begin{bmatrix}
\mathrm{diag}\left( \frac{\sqrt{ u_{k} }}{2} \right) & 0 \\
Z & \mathrm{diag}(2 \sqrt{ u_{k} })
\end{bmatrix}
$$
$$
f_{k} = \frac{\lambda_{k}}{4} \sum_{l \neq k} \frac{u_{l} + u_{k}}{D_{kl}},\quad
g_{k} = \frac{1}{2} \sum_{l \neq k} \frac{\lambda_{k}^{2} u_{l}^{2} - \lambda_{l}^{2}u_{k}^{2}}{D_{kl}^{2}},\quad
h_{k} =
\frac{\lambda_{k}^{2}}{2}
\sum_{l, r \neq k, l \neq r} \frac{u_{l}(u_{r} + 2u_{k})}{D_{kl}D_{kr}}
$$
and (for $l \neq k$)
$$
Z_{kk} =
\sqrt{ u_{k} } \lambda_{k}\sum_{l \neq k } \frac{ u_{l}}{D_{kl}},\quad
Z_{kl} = u_{k} \frac{\lambda_{l} \sqrt{u_{l} }}{D_{kl}}
.
$$
Next
$$
S S^{\top} = \begin{bmatrix}
\mathrm{diag}\left( \frac{u_{k}}{4} \right) & Y^{\top} \\
Y & \mathrm{diag}(4u_{k}) + X \\
\end{bmatrix}
$$
where $Y = Z \; \mathrm{diag}(\sqrt{ u_{k} } / 2)$
$$
Y_{kk} = \frac{u_{k}\lambda_{k}}{2} \sum_{l \neq k } \frac{ u_{l}}{D_{kl}},\quad
Y_{kl} = \frac{u_{k}}{2} \frac{\lambda_{l}u_{l} }{D_{kl}}
$$
and $X = Z Z^{\top}$
$$
X_{kk}
=
\sum_{l \neq k} \frac{u_{k} u_{l}}{D_{kl}} \left[ \frac{\lambda_{k}^{2} u_{l} + \lambda_{l}^{2}u_{k}}{D_{kl}} + \lambda_{k}^{2}\sum_{r \neq l, k}\frac{u_{r}}{D_{kr}}\right],\quad
X_{kl} = u_{k} u_{l}\left[ \frac{-(\lambda_{k}^{2}u_{l} + \lambda_{l}^{2}u_{k})}{D_{kl}^{2}} + \sum_{r \neq l, k} \frac{2u_{r}\lambda_{r}^{2}}{D_{kr} D_{lr}} \right]
.
$$

By the Kolmogorov Forward equation,
\begin{align*}
\partial_{t} p_{t}(\lambda, u)
=
-\sum_{k=1}^{n} \partial_{\lambda_{k}}(f_{k}\; p) - \partial_{u_{k}}((2 + g_{k} + h_{k})\; p) \\
+ \frac{1}{2}\sum_{k=1}^{n} \partial_{\lambda_{k}}^{2} \left( \frac{u_{k}}{4} p \right) + \partial_{u_{k}}^{2} ((4 u_{k} + X_{kk}) \; p)
+ 2 \partial_{u_{k}} \partial_{\lambda_{k}}(Y_{kk} \; p) \\
+ \sum_{l \neq  k} \partial_{\lambda_{k}} \partial_{u_{l}} (Y_{lk} \; p) + \frac{1}{2} \partial_{u_{k}} \partial_{u_{l}} (X_{kl} \; p)
\end{align*}
Furthermore as $\bA_{0} = 0, B_{0} = 0$, $p_{0}(\lambda, u) = \delta(\lambda) \delta(u)$. Now it remains to expand each term by product rule and collect by derivatives of $p$ to get the result.
\begin{align*}
\partial_{\lambda_{k}}(f_{k}\; p) &= \partial_{\lambda_{k}} f_{k} \; p + f_{k} \partial_{\lambda_{k}}p  \\
\partial_{u_{k}}((2 + g_{k} + h_{k})\; p) &= \partial_{u_{k}} (g_{k} + h_{k}) \; p + (2 + g_{k} + h_{k}) \partial_{u_{k}} p \\
\partial_{\lambda_{k}}^{2}\left( u_{k} p \right) &= u_{k} \partial_{\lambda_{k}}^{2} p \\
\partial_{u_{k}}^{2} (u_{k} \; p) &= 2 \partial_{u_{k}}p + u_{k} \partial_{u_{k}}^{2}p \\
\partial_{u_{k}}^{2} (X_{kk} \; p) &= \partial_{u_{k}}^{2} X_{kk} \; p + 2 \partial_{u_{k}} X_{kk} \partial_{u_{k}} p + X_{kk} \partial_{u_{k}}^{2}p\\
\partial_{u_{k}} \partial_{\lambda_{k}}(Y_{kk} \; p) &= \partial_{u_{k}} \partial_{\lambda_{k}} Y_{kk} \;p + \partial_{\lambda_{k}}Y_{kk} \partial_{u_{k}}p + \partial_{u_{k}} Y_{kk} \partial_{\lambda_{k}}p + Y_{kk} \partial_{u_{k}} \partial_{\lambda_{k}} p\\
\partial_{\lambda_{k}} \partial_{u_{l}} (Y_{lk} p) &= \partial_{\lambda_{k}} \partial_{u_{l}} Y_{lk} \;p + \partial_{u_{l}} Y_{lk} \; \partial_{\lambda_{k}}p   + \partial_{\lambda_{k}}Y_{lk} \; \partial_{u_{l}} p + Y_{lk} \partial_{\lambda_{k}} \partial_{u_{l}} p  \\
\partial_{\lambda_{l}} \partial_{u_{k}} (Y_{kl} p) &= \partial_{\lambda_{l}} \partial_{u_{k}} Y_{kl} \;p + \partial_{u_{k}} Y_{kl} \; \partial_{\lambda_{l}}p   + \partial_{\lambda_{l}}Y_{kl} \; \partial_{u_{k}} p + Y_{kl} \partial_{\lambda_{l}} \partial_{u_{k}} p  \\
\partial_{u_{k}} \partial_{u_{l}} (X_{kl} p) &= \partial_{u_{k}} \partial_{u_{l}} X_{kl} \; p+ \partial_{u_{k}}X_{kl} \partial_{u_{l}}  p +  \partial_{u_{l}} X_{kl} \partial_{u_{k}} p + X_{kl} \partial_{u_{k}} \partial_{u_{l}} p
\end{align*}

\textbf{Part 1: No derivative} The contribution is
$$
\sum_{k=1}^{n}
-\partial_{\lambda_{k}} f_{k} 
- \partial_{u_{k}} g_{k} 
- \partial_{u_{k}} h_{k}
+ \frac{1}{2}\partial_{u_{k}}^{2} X_{kk} 
+ \partial_{u_{k}} \partial_{\lambda_{k}} Y_{kk}
+ \sum_{l \neq k} \partial_{\lambda_{l}} \partial_{u_{k}} Y_{kl}
+ \frac{1}{2}
\partial_{u_{k}} \partial_{u_{l}} X_{kl}
$$
We compute
\begin{align*}
\partial_{\lambda_{k}} f_{k} = \frac{1}{4} \sum_{l \neq k} \frac{-(\lambda_{k}^{2} + \lambda_{l}^{2})(u_{l} + u_{k})}{D_{kl}^2}&,\quad
\partial_{u_{k}} g_{k} =  \sum_{l, r \neq k, l \neq r} \frac{\lambda_{k}^{2}u_{l}}{D_{kl} D_{kr}},\quad
\partial_{u_{k}}h_{k} = \sum_{l \neq k} \frac{ - \lambda_{l}^{2}u_{k}}{D_{kl}^{2}} \\
\partial^{2}_{u_{k}} X_{kk} = \sum_{l \neq k} \frac{ 2 \lambda_{l}^{2} u_{l}}{D_{kl}^{2}}&,\quad
\partial_{u_{k}} \partial_{\lambda_{k}} Y_{kk} = \frac{1}{2} \sum_{l \neq k} \frac{-(\lambda_{k}^{2} + \lambda_{l}^{2})u_{l}}{D_{kl}^{2}} \\
\partial_{\lambda_{l}} \partial_{u_{k}} Y_{kl} = \frac{u_{l}(\lambda_{k}^{2} + \lambda_{l}^{2})}{2D_{kl}^{2}}&,\quad
\partial_{u_{k}} \partial_{u_{l}} X_{kl} =  \frac{-2u_{l}\lambda_{k}^{2} - 2u_{k}\lambda_{l}^{2}}{D_{kl}^{2}} + \sum_{r \neq l, k} \frac{2u_{r}\lambda_{r}^{2}}{D_{kr} D_{lr}}
.
\end{align*}
The $Y_{kk},Y_{kl}$ terms cancel. The $D_{kl}^2$ can be gathered into
$$
\frac{1}{4} \sum_{l \neq k} \frac{\lambda_{k}^{2}(u_{k} - 3 u_{l}) + \lambda_{l}^{2}(u_k + 5u_{l})}{D_{kl}^{2}}
.
$$
Relabel
$$
\sum_{l, r \neq k, l \neq r} \frac{u_r \lambda_r^2}{D_{kr} D_{lr}}
=
\sum_{l, r \neq k, l \neq r} \frac{u_k \lambda_k^2}{D_{kl} D_{kr}}
.
$$
Therefore the remaining terms can be grouped as
$$
\sum_{l, r \neq k , l \neq r} \frac{\lambda_{k}^{2}(u_{k} - u_{l})}{D_{kl} D_{kr}}
.
$$
The overall coefficient is
$$
\frac{1}{4} \sum_{l \neq k} \frac{\lambda_{k}^{2}(u_{k} - 3 u_{l}) + \lambda_{l}^{2}(u_k + 5u_{l})}{D_{kl}^{2}}
+
\sum_{l, r \neq k , l \neq r} \frac{\lambda_{k}^{2}(u_{k} - u_{l})}{D_{kl} D_{kr}}
.
$$

\textbf{Part 2: $\partial_{\lambda_k}$} For fixed $k$, the contribution is
$$
- f_k +\partial_{u_k} Y_{kk} + \sum_{l \neq k}  \partial_{u_{l}} Y_{lk}
$$
Recall
$$
f_{k} = \frac{\lambda_{k}}{4} \sum_{l \neq k} \frac{u_{l} + u_{k}}{D_{kl}}
$$
and we compute
$$
\partial_{u_{k}} Y_{kk} = \frac{\lambda_{k}}{2} \sum_{l \neq k } \frac{ u_{l}}{D_{kl}},\quad
\partial_{u_{l}} Y_{lk} = \frac{-\lambda_{k} u_{k}}{2D_{kl}}
.
$$
Therefore the sum is
$$
\frac{\lambda_k}{4} \sum_{l \neq k} \frac{ u_{l} - 3u_{k}}{D_{kl}}
.
$$

\textbf{Part 3: $\partial_{u_k}$} For fixed $k$, the contribution is
$$
-(2 +g_{k} +h_{k}) + \partial_{u_{k}} X_{kk} + 4 + \partial_{\lambda_{k}} Y_{kk} + \sum_{l \neq k} \partial_{\lambda_{l}} Y_{kl} + \partial_{u_{l}} X_{kl}
.
$$
Recall
$$
g_{k} = \frac{1}{2} \sum_{l \neq k} \frac{\lambda_{k}^{2} u_{l}^{2} - \lambda_{l}^{2}u_{k}^{2}}{D_{kl}^{2}},\quad
h_{k} =
\frac{\lambda_{k}^{2}}{2}
\sum_{l, r \neq k, l \neq r} \frac{u_{l}(u_{r} + 2u_{k})}{D_{kl}D_{kr}}
.
$$
We compute
$$
\partial_{u_{k}} X_{kk} = \sum_{l \neq k} \frac{\lambda_{k}^{2}u_{l}^{2}  + 2 \lambda_{l}^{2} u_{k} u_{l}}{D_{kl}^{2}}
+
\sum_{l,r \neq k, r \neq l} \frac{\lambda_{k}^{2} u_{l} u_{r}}{D_{kl}D_{kr}},\quad
\partial_{\lambda_{k}} Y_{kk} = \frac{u_{k}}{2} \sum_{l \neq k} \frac{-(\lambda_{k}^{2} + \lambda_{l}^{2})u_{l}}{D_{kl}^{2}}
$$
and
$$
\partial_{\lambda_{l}}Y_{kl} = \frac{(\lambda_{l}^{2} + \lambda_{k}^{2}) u_{l} u_{k}}{2D_{kl}^{2}},\quad
\partial_{u_{l}} X_{kl} =  \frac{-2u_{l} u_{k} \lambda_{k}^{2} - u_{k}^{2}\lambda_{l}^{2}}{D_{kl}^{2}} + u_{k} \sum_{r \neq l, k} \frac{2u_{r}\lambda_{r}^{2}}{D_{kr} D_{lr}}
.
$$
First observe the constant becomes $2$ and $Y_{kk}, Y_{kl}$ terms cancel exactly. Now focus on the $D_{kl}^2$ denominator terms. This yields
$$
\frac{1}{2} \sum_{l \neq k} \frac{-(\lambda_{k}^{2}u_{l}^{2} - \lambda_{l}^{2}u_{k}^{2}) +2(\lambda_{k}^{2}u_{l}^{2}  + 2 \lambda_{l}^{2} u_{k} u_{l}) - 2(2u_{l} u_{k} \lambda_{k}^{2} + u_{k}^{2}\lambda_{l}^{2})}{D_{kl}^{2}}
.
$$
Simplify the numerator into $(\lambda_{k}^{2}u_{l}^{2} - u_{k}^{2}\lambda_{l}^{2}) - 4u_{l}u_{k} D_{kl}$. Hence we get
$$
\frac{1}{2} \sum_{l \neq k} \frac{\lambda_{k}^{2}u_{l}^{2} - \lambda_{l}^{2}u_{k}^{2}}{D_{kl}^{2}}
-2 \sum_{l \neq k} \frac{u_{l}u_{k}}{D_{kl}}
.
$$
The following sum can be relabelled via $l \leftrightarrow r$.
$$
\sum_{l \neq k} \partial_{u_l} X_{kl}
=
\sum_{l, r \neq k, l \neq r} \frac{-2u_k u_l \lambda_l^2}{D_{kl} D_{lr}}
$$
Now consider the remaining terms which simplifies into
$$
\frac{1}{2} \sum_{l, r \neq k, l \neq r}
\frac{\lambda_k^2(u_{l}u_{r} - 2u_{l}u_{k})}{D_{kl}D_{kr}}
-\frac{4u_k u_l \lambda_l^2}{D_{kl} D_{lr}}
$$
For the first term subtract out the condition $l = r$ to reveal
$$
\frac{\lambda_{k}^{2}}{2}\left( \sum_{l \neq k} \frac{u_{l}}{D_{kl}}\right)^{2}
- \frac{1}{2} \sum_{l \neq k} \frac{\lambda_{k}^{2}u_{l}^{2}}{D_{kl}^{2}}
.
$$
Therefore the overall sum is
$$
2
+ \frac{\lambda_{k}^{2}}{2}\left( \sum_{l \neq k} \frac{u_{l}}{D_{kl}}\right)^{2}
- 2 u_k \left( \sum_{l \neq k} \frac{u_{l}}{D_{kl}}\right)
- \frac{1}{2} \sum_{l \neq k} \frac{\lambda_{l}^{2} u_{k}^{2}}{D_{kl}^{2}}
- \sum_{l, r \neq k, l \neq r}
\frac{u_l u_k \lambda_k^2}{D_{kl} D_{kr}}
+
\frac{u_k u_l \lambda_l^2}{D_{kl} D_{lr}}
.
$$

\textbf{Part 4: $\partial_{\lambda_k}^2$} Only through $\frac{1}{2}\partial_{\lambda_k}^2 (u_k / 4)$. This yields $u_k / 8$.

\textbf{Part 5: $\partial_{u_k}^2$}
This only comes from $\frac{1}{2}\partial_{u_k}^2((4u_k + X_{kk})) p)$. This yields $2u_k + \frac{1}{2} X_{kk}$.

\textbf{Part 6: $\partial_{\lambda_k} \partial_{u_k}$} This only comes from $\partial_{u_k} \partial_{\lambda_k}(Y_{kk} \; p)$. This yields $Y_{kk}$.

\textbf{Part 7: $\partial_{\lambda_k} \partial_{u_l}$} Only from $\partial_{\lambda_k} \partial_{u_l}(Y_{lk} \; p)$. We cannot symmetrise and yields $Y_{lk}$.

\textbf{Part 8: $\partial_{u_k} \partial_{u_l}$} This only comes from $\frac{1}{2}\partial_{u_k} \partial_{u_l} (X_{kl} \;p)$. By symmetrising to $l < k$, we get $X_{kl}$.

\end{proof}

\section{Proofs for the Determinantal Point Process}
\label{appendix:Proofs for the Determinantal Point Process}
In this section we prove the steps required in the proof of Theorem \ref{theorem: Singular values are a DPP} and obtain the recurrence relations as stated in Lemmas \ref{lemma:Recurrence} and \ref{lemma:Moments of the empirical measure}.
For convenience, throughout this section we index $i$ starting from 0 and let the domain of the DPP be $\mathcal{X} = \R$ instead of $\R_{+}$ to simplify the integrals. As the integrands are even, this will only incur factors of 2. This is resolved in Section \ref{section:Simplifying the Correlation Function}.

\subsection{LU Decomposition of the Gram Matrix}
\label{section:LU decomposition of the Gram matrix}
In our setting, the Gram matrix is
$$
G_{ij} = \int_{\R} \xi_{i}(x) \eta_{j}(x) w(x) dx
$$
where
$$
\xi_{i}(x) = x^{2i},\quad
\eta_{i}(x) = \sech(x)^{2i},\quad
w(x)
=
\begin{cases}
\sech(x) & d = 2n \\
x \tanh(x) \sech(x) & d = 2n +1
\end{cases}
$$
To obtain the LU decomposition, we interpret $G_{ij} / (2i)!$ as the coefficient of $t^{2i}$ in the generating function for fixed \(j\). In the even case,
$$
\frac{G_{ij}}{(2i)!} =  [t^{2i}]G_{j}(t),\quad
G_{j}(t) := \int_{\R} \sum_{l=0}^{\infty} \frac{(tx)^{2l}}{(2l)!} \sech(x)^{2j+1} \, dx
.
$$
The generating function is the Laplace transform of \(\sech(t)^{2j+1}\) which we show is the product of $\sec(x)$ and an elementary symmetric polynomial.
Hence its coefficients are given by convolution. Explicitly, given a function $f$ and a polynomial $g$ of degree $j$
\begin{equation*}
    f(t)g(t) = \sum_{i=0}^{\infty} a_{i}t^{i} \sum_{l=0}^{j} b_{l} t^{l}
= \sum_{i=0}^{\infty} t^{i} \sum_{l=0}^{i \land j} a_{i -l}b_{l}
\end{equation*}
and $[t^i] fg = \sum_{l=0}^{i \land j} a_{i-l} b_{l}$. This implies an LU decomposition and is the basis for our proof.
\begin{equation}
\label{equation:Scheme for factorisation}
(D_{1})_{i} \cdot \sum_{l=0}^{i \land j} L_{il} \cdot (D_{2})_{l} \cdot U_{lj} \cdot (D_{3})_{j}
=
(D_{1} L D_{2} U D_{3})_{ij}
.
\end{equation}

\begin{lemma}[LU Decomposition of the Gram Matrix]
\label{lemma:Factorising G}
For \(d  = 2n\),
\begin{equation*}
    G = \left[\int_{\R} x^{2i} \sech(x)^{2j + 1}dx\right]_{i, j = 0}^{n-1}
    =
    D_{1} L D_{2} U D_{3}
\end{equation*}
where \(D\) are diagonal matrices
\begin{equation*}
(D_{1})_{i}
=
(-1)^{i} \left( \frac{\pi}{2} \right)^{2i} (2i)!,\quad
(D_{2})_{i}
=
\left( \frac{2}{\pi} \right)^{2i},\quad
(D_{3})_{i}
=
\frac{\pi}{(2i)!},
\end{equation*}
\(E_{2i}\) are the signed Euler numbers (also known as secant number, see OEIS A000364) and \(L, U\) are lower and upper triangular matrices respectively.
\begin{equation*}
    L_{ij} =
\begin{cases}
\frac{E_{2(i-j)}}{(2(i-j))!} & i \ge j \\
0 & \text{otherwise}
\end{cases}
,\quad
U_{ij} =
\begin{cases}
e_{j-i}(\{ (2k-1)^{2} \}_{k \in[j]}) & i \le j \\
0 & \text{otherwise}
\end{cases}
\end{equation*}

For \(d = 2n+1\) odd,
\begin{equation*}
    G = \left[\int_{\R} x^{2i+1} \sech(x)^{2j + 1} \tanh(x) dx\right]_{i, j = 0}^{n-1}
    =
    D'_{1} L D_{2} U D'_{3}
\end{equation*}
where $L, U, D_2$  are the same but the diagonal matrices $D_1', D_2'$ differ by a factor of $(2i+1)$.
\begin{equation*}
(D'_{1})_{i}
=
(-1)^{i} \left( \frac{\pi}{2} \right)^{2i} (2i+1)!,\quad
(D'_{3})_{i}
=
\frac{\pi}{(2i+1)!}
\end{equation*}

\end{lemma}
\begin{proof}
    Consider the generating function, whose coefficient of \(t^{2i}\) is \(G_{ij}/(2i)!\)
\begin{equation*}
G_{j}(t) 
= \int_{\mathbb{R}} e^{tx} \sech(x)^{2j+1} \, dx
= 2^{2j} \frac{\Gamma\left( j+\frac{1+t}{2} \right)\Gamma\left( j+\frac{1-t}{2} \right)}{\Gamma(2j+1)}
.
\end{equation*}
Lemma \ref{lemma:Laplace Transform of powers of sech} evaluates the Laplace transform and applying the identity $\Gamma(z+1) = z \Gamma(z)$ yields
\begin{equation*}
\Gamma\left( j+\frac{1+t}{2} \right)\Gamma\left( j+\frac{1-t}{2} \right)
=
\Gamma\left( \frac{1+t}{2} \right) \Gamma\left( \frac{1-t}{2} \right)
\frac{(-1)^{j}}{4^{j}}
\prod_{k=1}^{j} (t^{2} - (2k-1)^{2})
.
\end{equation*}
By Euler's reflection formula, the first term is equal to \(\pi \sec(\pi t / 2)\) which has a Taylor series in terms of the signed Euler numbers $E_{2m}$ (also known as secant numbers).
$$
\Gamma\left( \frac{1+t}{2} \right) \Gamma\left( \frac{1-t}{2} \right)
=
\pi \sec\left( \frac{\pi t}{2} \right)
=
\sum_{i=0}^{\infty} t^{2i}
\left( \frac{\pi}{2} \right)^{2i}
\frac{(-1)^{i} E_{2i}}{(2i)!}
$$
The second term can be written in terms of the elementary symmetric polynomials $e_{k}$.
$$
E_{j}(t) :=  \prod_{k=1}^{j} (t^{2} - (2k-1)^{2})
=
\sum_{i=0}^{j} t^{2i} (-1)^{j-i} e_{j-i}(\{ (2k - 1)^{2} \}_{k \in[j]})
$$
Therefore,
$$
G_{j}(t)
=
\pi \sec\left( \frac{\pi t}{2} \right)
E_{j}(t)
\;
\frac{(-1)^{j}}{\Gamma(2j+1)}
.
$$
By comparing coefficients of $t^{2i}$ and comparing with \eqref{equation:Scheme for factorisation}, we obtain the LU decomposition.
$$
G_{ij} =
(2i)! (-1)^{i} \left( \frac{\pi}{2} \right)^{2i}
\cdot
\sum_{l=0}^{i \land j}
\frac{E_{2(i-l)}}{(2(i-l))!}
\cdot
\left( \frac{\pi}{2} \right)^{-2l}
\cdot
e_{j-l}(\{ (2k-1)^{2} \}_{k\in[j]})
\cdot
\frac{\pi}{(2j)!}
$$

For \(d  = 2n+1\) odd, we consider the coefficients of $t^{2i+1}$ of
$$
G_{j}(t) = \int_{\mathbb{R}} e^{tx} \sech(x)^{2j+1} \tanh(x)  dx 
.
$$
As \(\frac{d}{dx} \sech(x)^{2j+1} = -(2j+1) \sech(x)^{2j+1} \tanh(x)\), we apply integration by parts to show
$$
G_{j}(t)
= \left[ e^{tx} \frac{-1}{2j+1} \sech(x)^{2j+1} \right]_{-\infty}^{\infty}
+ \frac{1}{2j+1} \int_{\R} e^{tx} \sech(x)^{2j+1} \, dx 
= \frac{t}{2j+1} \int_{\R} e^{tx} \sech(x)^{2j+1} \, dx
$$
for \(t < 1\) as \(\sech(x) \sim \exp(-|x|)\) for \(|x|\) large and \(e^{tx} \sech(x)^{2j + 1} \sim \exp(tx -|x|(2j+1)) \to 0\) for $|x| \to \infty$.
Except the factor of \(t /(2j+1)\), this is the same as the even case. We compare the coefficients of $t^{2i+1}$ to obtain the result.
\end{proof}

\subsection{Computing the Biorthogonal Functions}
\label{section:Computing the biorthogonal functions}
By Lemma \ref{lemma:Biorthogonal ensemble}, the biorthogonal functions are
\begin{equation*}
    p_{m}(x) = \sum_{k=0}^{m} P_{m, k} \xi_k,\quad
    q_{m}(x) = \sum_{k=0}^{m} Q_{m, k} \eta_k
\end{equation*}
where $P = L^{-1}$ and $Q^{\top} = U^{-1}$.
$L$ is a Toeplitz matrix, and its inverse another Toeplitz matrix which is given by the following lemma.
\begin{lemma}[Inverse of Toeplitz Matrices]
\label{lemma:Products of Toeplitz Matrices}
Given a Toeplitz matrix $A$ we define its generator as
\[
A(t) = \sum_{i \ge 0} a_i t^i \text{ with }a_{i-j}\coloneqq A_{ij},\quad a_{-i} = 0 \text{ for } i \ge 1.
\]
If $A, B$ are lower triangular Toeplitz Matrices, then $C=AB$ is also lower triangular Toeplitz with generator $C(t) = A(t) B(t)$.
If $A$ is invertible ($a_0 \neq 0$) then its inverse is also a lower triangular Toeplitz with generating function $A^{-1}(t)$.
\end{lemma}
\begin{proof}
    $C_{ij} = \sum_{k=0}^{n-1} A_{ik} B_{kj} = \sum_{k=0}^{n-1} a_{i-k} b_{k-j}$ is also Toeplitz with symbol $c_n = \sum_{k=0}^{n} a_{k} b_{n-k}$.
    Therefore, $C(t) = A(t) B(t)$ and $A, B$ are an inverse pair if $A(t) B(t) = 1$.
\end{proof}

\begin{lemma}[Inverse of L]
\label{lemma:Inverse of L}
The following lower triangular Toeplitz matrices are an inverse pair
$$
L_{ij} =
\frac{E_{2(i-j)}}{(2(i-j))!} \Indicator{i \ge j},\quad
L^{-1}_{ij} = \frac{1}{(2(i-j))!} \Indicator{i \ge j}
.
$$
\end{lemma}
\begin{proof}
Their generating functions \(L\) and are an inverse pair.
$$
L(x) = \sum_{i=0}^{\infty} \frac{E_{2i}}{(2i)!} x^{i} = \sech(\sqrt{ x }),\quad
L^{-1}(x) = \cosh(\sqrt{ x }) = \sum_{i=0}^{\infty} \frac{x^{i}}{(2i)!}
.
$$
Hence, by Lemma \ref{lemma:Products of Toeplitz Matrices} their associated lower triangular Toeplitz matrices are also an inverse pair.
\end{proof}

For \(U\), we use the following result to invert a matrix specified by elementary symmetric polynomials. This gives a matrix specified by complete symmetric polynomials.
\begin{lemma}[Inverse of a matrix of symmetric polynomials]
\label{lemma:Inverse of symmetric polynomial matrix}
Given a sequence \(x_1, \dots, x_n\), consider the elementary $e_k$ and complete $h_k$ symmetric polynomials.
Then the upper triangular matrices $A, B$ are an inverse pair.
$$
A_{mn}=e_{n-m}(x_{1:n}),\quad
B_{mn} = (-1)^{n-m} h_{n-m}(x_{1:m+1})
$$
\end{lemma}
\begin{proof}
Recall \(e_0 = h_0 = 1\) and \(e_{-1}, h_{-1} = 0\) and for \(k \ge 1\)
\begin{align}
e_{k}(x_{1}, \dots, x_{n}) &= \sum_{1 \le i_{1} < \dots < i_{k} \le n} \prod_{j=1}^{k} x_{i_{j}} \\
h_{k}(x_{1}, \dots, x_{n}) &= \sum_{1 \le i_{1} \le \dots \le i_{k} \le n} \prod_{j=1}^{k} x_{i_{j}}
.
\end{align}
For notational ease, let $r = n-m$
$$
(AB)_{mn}
= \sum_{k=0}^{n-m} A_{m, m+k} B_{m+k, n}
=\sum_{k=0}^{r} e_{k}(x_{1:m+k}) (-1)^{r-k} h_{r-k}(x_{1:m+k+1})
:= S_{r}(m)
$$
Clearly $S_{0}(m) = 1 = A_{mm} B_{mm}$ as $A_{mm} = B_{mm} = 1$.
Now, for $r > 0$ we consider the recursions
\begin{align*}
e_{k}(x_{1:m+k}) &= e_{k}(x_{1:m+k-1}) + x_{m+k} e_{k-1}(x_{1:m+k-1}) \\
h_{r-k}(x_{1:m+k+1}) &= h_{r-k}(x_{1:m+k}) + x_{m+k+1} h_{r-k-1}(x_{1:m+k+1})
\end{align*}
We decompose
\begin{align*}
e_{k}(x_{1:m+k})h_{r-k}(x_{1:m+k+1}) 
&= e_{k}(x_{1:m+k-1})h_{r-k}(x_{1:m+k}) \\
&+ x_{m+k} e_{k-1}(x_{1:m+k-1})  h_{r-k}(x_{1:m+k}) \\
&+ e_{k}(x_{1:m+k})x_{m+k+1} h_{r-k-1}(x_{1:m+k+1})
.
\end{align*}
Noting $e_{-1} = h_{-1} = 0$ and summing over $k$ with the weight $(-1)^{r-k}$ yields the last two contributions to cancel to 0.
Therefore, \(S_{r}(m) = S_{r}(m-1) = S_{r}(0)\).
It suffices to show $S_{r}(0) = 0$. Similarly,
\begin{align*}
e_{k}(x_{1:k}) &= x_{1} \dots x_{k} \\
h_{r-k}(x_{1:k+1}) &= h_{r-k}(x_{1:k}) + x_{k+1} h_{r-k-1}(x_{1:k+1})
.
\end{align*}
Noting that $h_{r}(\emptyset) = 0$ for $r > 0$ and $h_{-1} = 0$, we can shift the indices and cancel.
$$
S_{r}(0) = \sum_{k=0}^{r} (-1)^{r-k} x_{1} \dots x_{k} h_{r-k}(x_{1:k})
 - \sum_{k=0}^{r} (-1)^{r-k-1} x_{1}\dots x_{k+1} h_{r-k-1}(x_{1:k+1})
= 0
$$
\end{proof}

\begin{lemma}[Inverse of U]
\label{lemma:Inverse of U}
The following lower triangular matrices are an inverse pair
\begin{equation*}
    U_{ij} = e_{j-i}(\{ (2k-1)^{2} \}_{k \in[j]})
    \Indicator{i \le j}, \quad
    U^{-1}_{ij} =
    (-1)^{j-i} h_{j-i}(\{ (2k-1)^{2} \}_{k \in[i+1]})
    .
\end{equation*}
Furthermore, $U^{-1}_{ij} = (-1)^{j-i} W_{2i+1, j-i} \Indicator{i \le j}$ where \(W\) is the sequence which arises in Lemma \ref{lemma:Derivatives of sech}.
\end{lemma}
\begin{proof}
    Consider the sequence \(x_i = (2i-1)^{2}\), by Lemma \ref{lemma:Inverse of symmetric polynomial matrix} $U^{-1}_{ij} = (-1)^{j-i} h_{j-i}(\{ (2k-1)^{2} \}_{k \in[j]})$.
    To finish we observe \(h_k(x_1, \dots, x_{i+1}) = W_{2i+1, k}\).
    For fixed \(i\), as the complete symmetric polynomial can be characterised by the recursion
    \begin{equation*}
        h_{k}(x_1, \dots, x_{i+1}) = x_{i+1} h_{k-1}(x_{1}, \dots, x_{i}) + h_k(x_1, \dots, x_{i})
        .
    \end{equation*}
    By definition of \(W\), it satisfies the same recursion $W_{2i+1, k} = x_{i+1} W_{2i+1, k-1} + W_{2i-1, k}$ and the same base case \(k = 0, W_{2i+1, 0} = h_0(x_1, \dots, x_{i+1}) = 1\). Therefore, the sequences are identical.
\end{proof}

Now we compute the biorthogonal functions.
\begin{lemma}[Biorthogonal functions]
\label{lemma:Biorthogonal functions}
For \(m = 0, \dots, n-1\)
\begin{equation*}
    p_m(x)
    =
    \frac{1}{2}
    \begin{cases}
        \frac{(-1)^m}{(2m)!} \left[ \left( x + i\frac{\pi}{2} \right)^{2m} + \left( x - i\frac{\pi}{2} \right)^{2m} \right], & d = 2n\\
        \frac{(-1)^m}{x(2m + 1)!} \left[ \left( x + i\frac{\pi}{2} \right)^{2m+1} + \left( x - i\frac{\pi}{2} \right)^{2m+1} \right], & d = 2n+1
    \end{cases}
\end{equation*}

\begin{equation*}
        q_{m}(x)
=
\frac{1}{\pi}
\begin{cases}
\frac{(-1)^{m}}{\sech(x)} \sech^{(2m)}(x), & d = 2n \\
\frac{(-1)^{m+1}}{\sech(x) \tanh(x)} \sech^{(2m+1)}(x), & d = 2n + 1
\end{cases}
\end{equation*}
\end{lemma}
\begin{proof}
By Lemma \ref{lemma:Factorising G}, \(G = D_{1} L D_{2} U D_{3}\) and we set
$$
G^{-1} =
(D_{3}^{-1} U^{-1})
(D_{2}^{-1} L^{-1} D_{1}^{-1})
\to
P = D_{2}^{-1} L^{-1} D_{1}^{-1},\quad
Q^{\top} = D_{3}^{-1} U^{-1}
$$
The inverse of \(L, U\) are given in Lemmas \ref{lemma:Inverse of L}, \ref{lemma:Inverse of U}.
In the even case \(d = 2n\),
\begin{align*}
    P_{m, k} =
    \left(\frac{\pi}{2}\right)^{2(m-k)}
    \frac{1}{(2k)!}
    \frac{(-1)^{k}}{(2(m-k))!},\quad
    Q_{m, k} =
    \frac{ (-1)^{m-k}}{\pi} (2k)! W_{2k+1, m-k}
    .
\end{align*}
Therefore,
{\small
\begin{align*}
p_{m}(x)
&= \sum_{k=0}^{m} x^{2k} \left(\frac{\pi}{2}\right)^{2(m-k)}
    \frac{1}{(2k)!}
    \frac{(-1)^{k}}{(2(m-k))!}
=
\frac{1}{(2m)!}
\sum_{k=0}^{m} \binom{2m}{2k} (ix)^{2k} \left( \frac{\pi}{2} \right)^{2m-2k} \\
&= 
\frac{1}{2(2m)!}\left[ \left( \frac{\pi}{2} + ix \right)^{2m} + \left( \frac{\pi}{2} - ix \right)^{2m} \right] 
= 
\frac{(-1)^m}{2(2m)!}\left[ \left( x - i\frac{\pi}{2} \right)^{2m} + \left( x + i\frac{\pi}{2} \right)^{2m} \right]
\end{align*}
\begin{align*}
q_{m}(x)
&= \sum_{k=0}^{m} \sech(x)^{2k} \frac{ (-1)^{m-k}}{\pi} (2k)! W_{2k+1, m-k} \\
&= \frac{(-1)^{m}}{\pi \sech(x)} \sum_{k=0}^{m} \sech(x)^{2k+1} (-1)^{k} (2k)! W_{2k+1, m-k}
= \frac{(-1)^{m}}{\pi \sech(x)} \sech ^{(2m)}(x)
.
\end{align*}
}
In the odd case, we use the following coefficients and apply the same technique.
\begin{align*}
    P_{m, k} =
    \left(\frac{\pi}{2}\right)^{2(m-k)}
    \frac{1}{(2k+1)!}
    \frac{(-1)^{k}}{(2(m-k))!},\quad
    Q_{m, k} =
    \frac{ (-1)^{m-k}}{\pi} (2k+1)! W_{2k+1, m-k}
\end{align*}
\end{proof}

\subsection{Simplifying the Correlation Function}
\label{section:Simplifying the Correlation Function}
Finally, we show how to arrive at the correlation function and kernel stated in Theorem \ref{theorem: Singular values are a DPP}, which we notate here with $K'$. That is
\begin{equation*}
    \rho_k(x_1, \dots, x_k) = \det\left[ K'_{n}\left( x_{i}, x_{j}\right) \right]_{i, j = 1}^{k},\quad k \le n
\end{equation*}
where the kernel is
\begin{equation*}
    K'_n(x, y) = \sum_{m=0}^{n-1} p'_m(x) q'_m(y)
\end{equation*}
and $p'_m, q'_m$ are biorthogonal functions in the sense that \(\int_{0}^{\infty} p'_i(x) q'_j(x) dx = \delta_{ij}\).
\begin{align*}
        p'_m(x)
        &=
        \mathrm{Re} \left[
\begin{cases}
\left( x + i\frac{\pi}{2}\right)^{2m}, & d = 2n\\
\left( x + i \frac{\pi}{2} \right)^{2m + 1}, & d = 2n+1
\end{cases} \right] \\
        q'_{m}(x)
&=
\frac{2}{\pi}
\begin{cases}
\frac{1}{(2m)!} \sech^{(2m)}(x), & d = 2n \\
\frac{-1}{(2m + 1)!} \sech^{(2m+1)}(x), & d = 2n + 1
.
\end{cases}
\end{align*}
When computing the Gram matrix $G$, we used the domain $\mathcal{X} = \R$. As $\xi, \eta, w$ are even functions, we need to adjust the Gram matrix $G$ by a factor of $1/2$. When computing $P, Q$ from $G^{-1}$ this gets converted to a factor of $2$. Hence, this scales the $k$-point correlation function by $2^{k}$.
\begin{equation*}
    \rho_k(x_1, \dots, x_k) = 2^{k}
    \det\left[ K_n\left( x_i, x_j \right) \right]_{i, j = 1}^{k}
    \prod_{i=1}^{k} w(x_i)
\end{equation*}
where $K_n(x, y) = \sum_{m=0}^{n-1} p_{m}(x) q_{m}(y)$ and
\begin{align*}
    p_m(x)
    &=
    \frac{1}{2}
    \begin{cases}
        \frac{(-1)^m}{(2m)!} \left[ \left( x + i\frac{\pi}{2} \right)^{2m} + \left( x - i\frac{\pi}{2} \right)^{2m} \right], & d = 2n\\
        \frac{(-1)^m}{x(2m + 1)!} \left[ \left( x + i\frac{\pi}{2} \right)^{2m+1} + \left( x - i\frac{\pi}{2} \right)^{2m+1} \right], & d = 2n+1
    \end{cases} \\
        q_{m}(x)
&=
\frac{1}{\pi}
\begin{cases}
\frac{(-1)^{m}}{\sech(x)} \sech^{(2m)}(x), & d = 2n \\
\frac{(-1)^{m+1}}{\sech(x) \tanh(x)} \sech^{(2m+1)}(x), & d = 2n + 1
\end{cases}
.
\end{align*}
Any kernel representation is valid as long as the determinant of its Gram matrix is the same.
\begin{itemize}
    \item $p'_m$ - $p_m$ corresponds to the real part of $(x + i \pi/2)^{2m}$ as it is summed with its complex conjugate. We then shift the factorials and $(-1)^m$ from $p_m$ to $q_m$.
    In the odd case, the factor of $1/x$ will be handled in the third point.
    \item $q'_m$ - We absorb the factors from $p_m$ to cancel $(-1)^m$ and introduce the factorials.
    \item $w$ -  Finally, we absorb $\prod_{i=1}^{k} 2w(x_i)$ into the determinant. In the even case we use it to multiply each column by $2w(x_j)$, this multiplies the kernel by $2w(y)$ which recovers $q'_m$. 
    In the odd case we split the weight function into $x$ and $ 2\sech(x) \tanh(x)$. Multiplying the columns $ 2\sech(x_j) \tanh(x_j)$ recovers $q'_m$ and multiplying the rows by $x_i$ recovers $p'_m(x)$. 
\end{itemize}
Furthermore, the functions are now biorthogonal with respect to the standard integral product instead of the one weighted by $w$ as
\begin{equation*}
\delta_{ij}
= \int_{\mathbb{R}} p_{i}(x)q_{j}(x) w(x)  \, dx
= \frac{1}{2} \int_{\mathbb{R}} p'_{i}(x) q'_{j}(x) \, dx 
= \int_{\mathbb{R}_{+}} p'_{i}(x) q'_{j}(x) \, dx 
\end{equation*}

\subsection{Proofs for Recurrence Relations of the Biorthogonal Functions}
\label{appendix:Proofs for recurrence relations for the Biorthogonal functions}
Again these follow by considering a suitable generating function and comparing coefficients.

\begin{proof}[Proof of Lemma \ref{lemma:Recurrence}]
For both $p_m, q_m$, we show they are polynomials, up to a factor, and identify the coefficients of their recurrence relation.
We use the notation $p_m^e, p_m^o$ to denote the even and odd version of $p_m$ and similarly for other variables.

\textbf{Part 1: $p_m$ recurrence}
$p_m$ is a polynomial. In the even case $d=2n$, it is an even function $p_m(x) = p_m(-x)$. Hence it must only have even degree terms and be a polynomial in $x^2$. For $d=2n+1$, it is an odd function hence must be a polynomial in $x^2$ times $x$. Next let $z = x + i \frac{\pi}{2}$ and set $p_m$ and $\psi_m$ as the real and imaginary part of $z^{2m}$ in the even case (and $z^{2m+1}$ in the odd case respectively).
Using $x^{2} = \left( x \pm i \frac{\pi}{2} \right)^{2} + \left( \frac{\pi}{2} \right)^{2} + \mp i \pi x$ shows that in both even and odd cases
\[
x^{2} p_{m} = p_{m+1} + \left( \frac{\pi}{2} \right)^{2} p_{m} + \pi x \, \psi_{m}(x)
.
\]
To establish the recursion, we express $x \psi_m$ in terms of $p_m$. To this end, we introduce the generating functions $P(t), \Psi(t)$ over $p_m, \psi_m$ respectively. We define $R(t) = \tan(\frac{\pi}{2}t) P(t)$ and then show $\frac{d}{dt} R(t) = x\Psi(t) + \frac{\pi}{2} P(t)$. Comparing coefficients yields the result.

In the even case, the generating functions are
\begin{align*}
P^{e}(t) &= \sum_{m=0}^{\infty} \frac{p^e_{m}}{(2m)!} t^{2m} = \mathrm{Re} \left( \cosh\left( t\left( x + i \frac{\pi}{2} \right) \right) \right) = \cosh(tx) \cos\left( \frac{\pi t}{2} \right) \\
\Psi^{e}(t) &= \sum_{m=0}^{\infty} \frac{\psi^{e}_{m}}{(2m)!} t^{2m} = \mathrm{Im}\left( \cosh\left( t\left( x + i \frac{\pi}{2} \right) \right) \right) = \sinh(xt) \sin\left( \frac{\pi t}{2} \right) \\
R^{e}(t) &= \cosh(tx) \sin\left( \frac{\pi t}{2} \right) = \tan\left( \frac{\pi}{2}t \right) P^{e}(t) = \sum_{m=0}^{\infty} r^{e}_{m} \frac{t^{2m+1}}{(2m+1)!}
\end{align*}
where $r^{e}_{m}$ is obtained by convolving the coefficients of $\tan(\pi t / 2)$ and $P^{e}$.
\[
r^{e}_{m} = \sum_{k=0}^{m} \binom{2m+1}{2k} T_{m - k} \left( \frac{\pi}{2} \right)^{2(m-k)+1} p^{e}_{k}(x)
\]
The derivative identity follows as
$$
\frac{d}{dt} R^{e}(t) = x \sinh(tx) \sin\left( \frac{\pi t}{2} \right) + \frac{\pi}{2} \cosh(xt) \cos\left( \frac{\pi t}{2} \right) =x\Psi^{e}(t) + \frac{\pi}{2} P^{e}(t)
.
$$
Comparing coefficients gives an expression for $x \psi^{e}_m$ and establishes the recursion for $x^{2} p^{e}_{m}$.
$$
x\psi^{e}_{m} = r^{e}_{m} - \frac{\pi}{2} p^{e}_{m} = 2m\left( \frac{\pi}{2} \right) p^{e}_{m} + \sum_{k=0}^{m-1} \binom{2m+1}{2k} T_{m - k} \left( \frac{\pi}{2} \right)^{2(m-k)+1} p^{e}_{k}(x)
$$

In the odd case, we consider the generating functions
\begin{align*}
P^{o}(t) &= \sum_{m=0}^{\infty} \frac{p^{o}_{m}}{(2m+1)!} t^{2m+1} = \mathrm{Re} \left( \sinh\left( t\left( x + i \frac{\pi}{2} \right) \right) \right) = \sinh(tx) \cos\left( \frac{\pi t}{2} \right) \\
\Psi^{o}(t) &= \sum_{m=0}^{\infty} \frac{\psi^{o}_{m}}{(2m+1)!} t^{2m+1} = \mathrm{Im}\left( \sinh\left( t\left( x + i \frac{\pi}{2} \right) \right) \right) = \cosh(xt) \sin\left( \frac{\pi t}{2} \right) \\
R^{o}(t) &= \sinh (tx) \sin\left( \frac{\pi t}{2} \right) = \tan\left( \frac{\pi}{2}t \right) P^{o}(t) = \sum_{m=0}^{\infty} r^{o}_{m} \frac{t^{2m+2}}{(2m+2)!}
\end{align*}
where $r^o_m$ is obtained by convolving $\tan$ with $P^o$.
$$
r^{o}_{m} = \sum_{k=0}^{m} \binom{2m+2}{2k+1} T_{m - k} \left( \frac{\pi}{2} \right)^{2(m-k)+1} p^{o}_{k}(x)
$$
Again $\frac{d}{dt} R^{o}(t) = x\Psi^{o}(t) + \frac{\pi}{2} P^{o}(t)$ and comparing coefficients yields the result. Finally a standard result is that the tangent numbers (OEIS A000182) are given by
$$
T_{m} = 2 \left( \frac{2}{\pi} \right)^{2m+2}(2m+1)! \,(1 - 2^{-(2m+2)})\zeta(2m+2)
$$
where $\zeta(s) = \sum_{n=1}^{\infty} \frac{1}{n^{s}}$ is the Riemann Zeta function.

\textbf{Part 2: $q_m$ recurrence}
By Lemma \ref{lemma:Derivatives of sech}, $q^{e}_m(x) / \sech(x)$ and $q^{o}_m(x) / (\sech(x) \tanh(x))$ are polynomial in $\sech(x)^2$.
Let $y = \sech(x)^2$. Again we consider a generating function $Q(t)$ over $q_m$ and show that $y Q(t) = \csc(t)^{2} Q(t) + f(t)$ for some function $f$ depending on the parity. We obtain an explicit formula for the coefficients of $f$ and comparing coefficients yields the result.

In the even case,
$$
Q^{e}(t) = \frac{\pi / 2}{\sech(x)} \sum_{m=0}^{\infty} (-1)^{m} q^{e}_{m}(x) t^{2m} = \frac{\sech(x + it) + \sech(x - it)}{2\sech(x)} =  \frac{\cos(t)}{1 - y \sin(t)^{2}}
$$
and
$$
y Q^{e}(t) = y \frac{\cos(t)}{1 - y \sin(t)^{2}} = \frac{1}{\sin(t)^{2}} \left[ \frac{\cos(t)}{1 - y \sin(t)^{2}} - \cos (t) \right]
=
\csc(t)^{2} Q^{e}(t) - \frac{\cos(t)}{\sin(t)^{2}}
.
$$
Comparing the coefficient of $t^{2m}$ in $yQ(t)$, we deduce
$$
\beta^{e}_{k, m} = [t^{2k-2}] \csc(t)^{2},\quad \beta^{e}_{m+1, m} = [t^{2m}] \left( \csc(t) - \frac{\cos(t)}{\sin(t)^{2}} \right)
.
$$

In the odd case, 
$$
Q^{o}(t) = \frac{\pi i / 2}{\sech(x) \tanh(x)} \sum_{m=0}^{\infty} (-1)^{m}q^{o}_{m}(x) t^{2m} = \frac{\sech(x + it) - \sech(x - it)}{-2i \sech(x) \tanh(x)} =  \frac{\sin(t)}{1 - y \sin(t)^{2}}
$$
and $y Q^{o}(t) = \csc(t)^{2} Q^{o}(t) - \frac{1}{\sin(t)}$. Comparing the coefficient of $t^{2m+1}$, we deduce
$$
\beta^{o}_{k, m} = [t^{2k-2}] \csc(t)^{2},\quad
\beta^{o}_{m+1, m} = [t^{2m}] \csc(t)^{2} - [t^{2m+1}] \csc(t)
.
$$
In both cases, $\beta_{0, m} = 1$ as $[t^{-1}]\csc(t) = 1$.

\textbf{Part 2.1 Explicit formula for the coefficients} For $k \ge 1$, we show
\begin{align*}
    [t^{2k}] \csc(t)^{2} &= \frac{2(2k+1)}{\pi^{2k+2}}\sum_{n=1}^{\infty} \frac{1}{n^{2k+2}} = \frac{2(2k+1) \zeta(2k+2)}{\pi^{2k+2}} \\
    [t^{2k+1}] \csc(t) &= \frac{2}{\pi^{2k+2}} \sum_{n=1}^{\infty} \frac{(-1)^{n+1}}{n^{2k+2}}
=\frac{2(1-2^{-(2k+1)}) \zeta(2k+2)}{\pi^{2k+2}} \\
    -[t^{2k}] \frac{\cos(t)}{\sin(t)^{2}} &= (2k+1) [t^{2k+1}] \csc(t)
    .
\end{align*}
This gives $\beta_{k, m}$ for $1 \le k \le m$. For $k = m + 1$, we combine to give
\begin{align*}
    \beta^{e}_{m+1, m} &= \frac{2(2m+1)\zeta(2m+2)}{\pi^{2m+2}}(1+1-2^{-2m-1}) = \frac{4(2m+1)\zeta(2m+2)}{\pi^{2m+2}}(1-2^{-2m-2}) \\
\beta_{m+1, m}^{o} &= \frac{4\zeta(2m+2)}{\pi^{2m+2}} (2m+1 +1 -2^{-2m-1}) = \frac{4\zeta(2m+2)}{\pi^{2m+2}} (m+1 -2^{-2m-2})
.
\end{align*}
As our generating function used $(-1)^m$, we have shown recurrence for $(-1)^{m} q_m$ and
$$
\sech(x)^{2}(-1)^m q_{m}(x) = \sum_{k=0}^{m+1} \beta_{k, m} ((-1)^{m+1-k}q_{m+1-k}(x))
.
$$
To recover the formulation in the lemma, we cancel the factors of $(-1)^m$.

Next, $\beta_{k, m} \ge 0$ as all term are positive. To show $\beta_{k, m} \le 1$, use $\zeta(2k) \le \zeta(2) = \pi^2 / 6  < 2$ and consider $f(x)= 4(2x+1)/\pi^{2x+2}$ which is decreasing with $f(0) < 1$. Hence $f(x) < 1$ for $x \ge 0$ and $\beta_{k, m} \le f(k-1) < 1$.

\textbf{Part 2.2 Derivation of the coefficients}
Use the pole expansion of $\csc^2$, expand $(t-a)^{-2}$ and compare coefficients to deduce
\begin{align*}
\csc(t)^{2} &= \frac{1}{t^{2}} + \sum_{n=1}^{\infty} \frac{1}{(t - n \pi)^{2}} + \frac{1}{(t + n \pi)^{2}} \\
&= \frac{1}{t^{2}} + \sum_{n=1}^{\infty} \sum_{r \ge 0} \frac{2(2r+1)}{\pi^{2r+2} n^{2r+2}} t^{2r}
.
\end{align*}
The coefficients of $-\frac{\cos(t)}{\sin(t)^{2}} = \frac{d}{dt}\csc(t)$ follow from $\csc$. Again, we consider the pole expansion.
\begin{align*}
\csc(t) &= \frac{1}{t} + \sum_{n=1}^{\infty} \frac{(-1)^{n}}{t-n\pi} + \frac{(-1)^{n}}{t + n \pi} \\
&= \frac{1}{t} + \sum_{n=1}^{\infty} (-1)^{n+1} \sum_{r \ge 0} \frac{2t^{2r+1}}{n^{2r+2} \pi^{2r+2}}
\end{align*}
\end{proof}

\begin{proof}[Proof of Lemma \ref{lemma:Moments of the empirical measure}]
Let $\alpha^N_m$ denote the coefficient of $p_m$ in $x^{2N} p_m$
By biorthogonality, $m_N = \sum_{m=0}^{n-1} \alpha^{N}_m$. Let $z = (x + i \frac{\pi}{2})$, by the binomial theorem
\begin{align*}
x^{2N} &= \sum_{j=0}^{N} \binom{2N}{2j} \left( \frac{\pi}{2} \right)^{2(N-j)} (-1)^{N-j} z^{2j}
- i \frac{\pi}{2} \sum_{j=0}^{N-1} \binom{2N}{2j+1}\left( \frac{\pi}{2} \right)^{2(N-j-1)} (-1)^{N-j-1} z^{2j+1} \\
&= \sum_{j=0}^{N} \binom{2N}{2j} \left( \frac{\pi}{2} \right)^{2(N-j)} (-1)^{N-j} \overline{z}^{2j}
+ i \frac{\pi}{2} \sum_{j=0}^{N-1} \binom{2N}{2j+1}\left( \frac{\pi}{2} \right)^{2(N-j-1)} (-1)^{N-j-1} \overline{z}^{2j+1}
.
\end{align*}
Let $\widetilde{\psi}_{m} = \mathrm{Im}\left( z^{2m+1} \right)$. Therefore,
$$
x^{2N} p_{m}
=
\sum_{j=0}^{N} \binom{2N}{2j} \left( \frac{\pi}{2} \right)^{2(N-j)} (-1)^{N-j} \,p_{m+j}
+ \frac{\pi}{2} \sum_{j=0}^{N-1} \binom{2N}{2j+1}\left( \frac{\pi}{2} \right)^{2(N-j-1)} (-1)^{N-j-1}\, \widetilde{\psi}_{m + j}
.
$$
Note that
$$
\widetilde{\psi}_{m+j}
=
x \psi_{m+j} + \frac{\pi}{2} p_{m+j}
=
\sum_{k=0}^{m+j} \binom{2(m+j)+1}{2k} T_{m+j-k} \left( \frac{\pi}{2} \right)^{2(m + j-k) +1} p_{k}
.
$$
Collecting the coefficients of $p_m$, we deduce
$$
\alpha^N_m
=
\left( \frac{\pi}{2} \right)^{2N} (-1)^{N}
+ \frac{\pi}{2} \sum_{j=0}^{N-1}
\binom{2N}{2j+1}\left( \frac{\pi}{2} \right)^{2N-1} (-1)^{N-j-1}
\binom{2(m+j)+1}{2m} T_{j}
$$
and summing over $m$ yields the result.
\end{proof}

\end{document}